\documentclass[11pt,reqno]{amsart}

\usepackage{geometry}
\usepackage{amsthm,amsmath,amsfonts,amssymb}
\usepackage{natbib}
\usepackage[hidelinks]{hyperref}
\usepackage{graphicx, xcolor}

\theoremstyle{plain}

\newtheorem{theorem}{Theorem}[section]
\newtheorem{lemma}[theorem]{Lemma}
\newtheorem{corollary}[theorem]{Corollary}
\newtheorem{proposition}[theorem]{Proposition}

\theoremstyle{definition}
\newtheorem{definition}[theorem]{Definition}

\newtheorem{assumption}[theorem]{Assumption}
\newtheorem{remark}[theorem]{Remark}

\def\what#1{\widehat{#1}}

\theoremstyle{remark}

\newcommand{\E}{\mathbb{E}}

\renewcommand{\Pr}{\mathbb{P}}
\newcommand{\R}{\mathbb{R}}

\newcommand{\G}{\mathbb{G}}
\newcommand{\bbG}{\mathbb{G}}
\newcommand{\bbP}{\mathbb{P}}

\newcommand{\calF}{\mathcal{F}}
\newcommand{\calE}{\mathcal{E}}
\newcommand{\calG}{\mathcal{G}}
\newcommand{\calX}{\mathcal{X}}

\newcommand{\calI}{\mathcal{I}}

\newcommand{\sfA}{\mathsf{A}}
\newcommand{\sfV}{\mathsf{V}}
\newcommand{\brf}{\check{f}}
\newcommand{\brg}{\check{g}}
\newcommand{\brh}{\check{h}}

\newcommand{\kernel}{w}

\newcommand{\classkernel}{\mathcal{K}}
\newcommand{\classemp}{\mathcal{V}}
\newcommand{\empkernel}{v}

\newcommand{\studclass}{{\classkernel}^{(n)}}

\newcommand{\studkernel}{f^{(n)}}

\newcommand{\discmeas}{\check R}

\newcommand{\bbPcheck}{\bbP_{\check \alpha}} 

\newcommand{\upperpif}{u(\calF, \pi)}

\newcommand{\Var}{\mathrm{Var}}
\newcommand{\Cov}{\mathrm{Cov}}

\begin{document}
\title[Wild regenerative block bootstrap for Harris recurrent Markov chains]{Wild regenerative block bootstrap for Harris recurrent Markov chains}

\date{First version: January 30, 2023. This version: \today}

\author[K. Choi]{Kyuseong Choi$^{*}$}

\address[K. Choi$^{*}$]{Statistics and Data Science, Cornell Tech, Cornell University.}
\email{kc728@cornell.edu}

\author[G. Ciolek]{Gabriella Ciolek}
\address[G. Ciolek]{Mathematics Department, University of Luxembourg.}
\email{gabrielaciolek@gmail.com}

\begin{abstract}

We consider Gaussian and bootstrap approximations for the supremum of additive functionals of aperiodic Harris recurrent Markov chains. The supremum is taken over a function class that may depend on the sample size, which allows for non-Donsker settings; that is, the empirical process need not have a weak limit in the space of bounded functions.
We first establish a non-asymptotic Gaussian approximation error, which holds at rates comparable to those for sums of high-dimensional independent or one-dependent vectors. Key to our derivation is the Nummelin splitting technique, which enables us to decompose the chain into either independent or one-dependent random blocks.  Additionally, building upon the Nummelin splitting, we propose a Gaussian multiplier bootstrap for practical inference and establish its finite-sample guarantees in the strongly aperiodic case.  Finally, we apply our bootstrap to construct a uniform confidence band for an invariant density within a certain class of diffusion processes. 
\end{abstract}



\maketitle

\section{Introduction}

Markov chains constitute a rich class of stochastic processes. Examples include discrete observations from diffusion processes, commonly used in quantitative finance, and Markov decision processes appearing in reinforcement learning, among others. Additionally, in Bayesian statistics, Markov chain Monte Carlo methods are routinely used to approximately sample from Bayesian posterior distributions. Motivated by these diverse applications of Markov chains, this work considers the Gaussian and bootstrap approximations for the supremum of additive functionals of Harris recurrent Markov chains, where the supremum is taken over a function class that may depend on the sample size, which allows for non-Donsker settings; that is, the empirical process need not have a weak limit in the space of bounded functions.



Specifically, let $X_0, X_1, ..., X_{n-1}$ be sampled from an aperiodic Harris recurrent Markov chain with values in a state space $E$ having $\pi$ as its invariant probability measure. For a given class of $\pi$-integrable functions on $E$, $\calF = \calF_n$, which is allowed to depend on the sample size $n$, we are interested in approximating the distribution of the supremum of the empirical process,
\begin{equation}
\bbG_n(f) = n^{-1/2}\sum_{i=0}^{n-1}\{f(X_i) - \pi(f)\}, \quad f \in \calF,
\label{eq: empirical process}
\end{equation}
where $\pi (f) = \int_E f \, d\pi$.
Our first goal is to approximate the distribution of $\bbG_{n, \calF}^\vee := \sup_{f \in \calF} \bbG_n(f)$ with that of a suitable Gaussian analog. Namely, 
let $\{ G(f) \}_{f \in \calF}$ be a suitable version of a centered Gaussian process whose covariance function approximates that of $\bbG_n$, and we establish finite sample upper bounds for the Kolmogorov distance, 
\[
\rho_n := \sup_{t \in \mathbb{R}} \Big| \bbP\big(\bbG_{n,\calF}^\vee\leq t\big) - \bbP\big(G_{\calF}^{\vee} \leq t\big)\Big|,
\]
where $G_{\calF}^{\vee} := \sup_{f \in \calF}G(f)$.

Key to the derivation of our Gaussian approximation results is Nummelin splitting technique \cite{Nummelin1978PTRF,Nummelin1984,Meyn2012}, which enables decomposition of the Markov chain into either independent or one-dependent random blocks, where the construction of such blocks relies on successive stopping times. We then combine the recent results of high-dimensional central limit theorems \cite{chernozhuokov2022improved,ChangChen2024} applied to the discretized empirical process and upper bound the discretization errors. To achieve the latter, we derive new maximal inequalities applicable to Harris recurrent Markov chains, which may be of independent interest.  

Still, the above Gaussian approximation result is not directly applicable to practical inference, as the covariance function of the approximating Gaussian process $\{G(f)\}_{f \in \calF}$ is unknown and, in fact, appears to be nontrivial to estimate analytically. To address this, we propose a novel bootstrap, which we term the \textit{Gaussian wild regenerative block bootstrap},  tailored to Harris recurrent chains and establish its validity for the strongly aperiodic case (i.e., when the Markov chain can be decomposed into independent blocks). The random blocks appearing resulting from Nummelin splitting rely on the transition kernel, which is unknown in practice, so we estimate the transition kernel to construct an approximate split chain, from which we obtain estimated blocks \cite{BertailClemencon2006Bernoulli}. The proposed method then applies a Gaussian wild bootstrap to the estimated blocks. For this new bootstrap, we establish finite-sample guarantees under the Kolmogorov distance.

Finally, we apply our bootstrap method to construct a uniform confidence band for the stationary density of a certain class of diffusion processes~\cite{nickl2017nonparametric}. The proposed confidence band has a coverage guarantee polynomially decaying with the sample size, which should be contrasted with the confidence band constructed using an extreme value limit \cite{bickel1973some} whose coverage guarantee is known to be logarithmically slow. We borrow anti-concentration results for Gaussian processes from \cite{CCK2014bAoS} to 
bypass any need for a limiting distribution.

\subsection{Literature review}

There is now a rich literature on non-asymptotic Gaussian approximation and its corresponding bootstrap techniques for high-dimensional data. For independent $p$-dimensional data, Berry-Esseen type results, consistent when $p \leq O(n^c)$ for some constant $c > 0$, were established in \cite{Bentkus1986, Bentkus2003}, while the seminal paper \cite{CCK2013AoS} improved the rate so as to allow consistency when $p$ can be as large as $O\big(\exp(Cn^c)\big)$ for some constants $c, C > 0$. Substantial improvements in this direction were made on varying assumptions~\cite{CCK2017AoP,chernozhuokov2022improved,lopes2020bootstrapping,lopes2022central,deng2020beyond,fang2021high,chernozhukov2023nearly}, while \cite{chen2018gaussian,chen2020jackknife,song2019approximating} further accommodate U-statistics with non-linear kernels. Bootstrap techniques established in the aforementioned literature help construct uniform confidence bands for densities \cite{CCK2014bAoS} and simultaneous confidence set for high dimensional parameters~\cite{belloni2018uniformly}, test conditional or unconditional many moment inequalities~\cite{chang2017testing,chen2018gaussian,chernozhukov2019inference, koning2019exact} and design change point detection algorithms for high-dimensional vectors~\cite{yu2021finite,yu2022robust}. 

For dependent and time series data, Gaussian approximations and related bootstrap techniques were established in \cite{ZhangWu2017,ZhangCheng2017,kurisu2023gaussian,ChangChen2024}. Specifically, strongly mixing and $m$-dependent high-dimensional data are considered in \cite{ChangChen2024}. Here we focus on Harris recurrent Markov chains~(hence strongly mixing) and provide an improved rate for Gaussian approximation over infinite-dimensional data points, where the dimensionality is characterized by a reasonably sized function class of interest. Notably our result provides a sharper approximation compared to that of \cite{ChangChen2024}, with rates even comparable to those for independent and $m$-dependent data points~\cite{chernozhuokov2022improved,ChangChen2024}. A key analytical tool for achieving a sharper rate is the Nummelin splitting technique~\cite{Nummelin1984}, which splits the chain into independent or one-dependent identically distributed random blocks via successive regeneration times. 




Markov chains are a popular modeling choice as they constitute a substantial portion of stochastic processes; examples include Markov decision processes~\cite{sutton1998reinforcement,kallenberg2002finite,bubeck2012regret} and time series models~\cite{fan2021hoeffding}. Some sampling algorithms~\cite{dalalyan2017theoretical} show non-asymptotic convergence of the discrete samples from continuous time diffusion processes~(also forming Markov chains) to its stationary distribution. Here we focus on a class of diffusion models and provide a uniform confidence band for its stationary distribution using the discrete observations. The approximate block bootstrap technique~\cite{BertailClemencon2006Bernoulli,ciolek2016bootstrap} is one of the key features in our proposed bootstrap procedure. Our uniform inference procedure does not rely on the Smirnov-Bickel-Rosenblatt condition~\cite{gine2009exponential}, thereby allowing the confidence level of uniform confidence band to approach the pre-specified confidence level at a polynomial rate~\cite{CCK2014bAoS}.


%
%
%

\subsection{Organization}

Section \ref{sec:gaussian-approximation} reviews the Nummelin splitting technique for Markov chains satisfying a minorization condition and discusses the setting. The rest of Section \ref{sec:gaussian-approximation} establishes a Gaussian approximation result for the Kolmogorov-type metric $\rho_n$, under both independent and one-dependent random blocks scenario. Section \ref{sec:bootstrap} reviews the approximate Nummelin splitting technique (to construct an approximate split chain from estimated transition kernel). The rest of Section \ref{sec:bootstrap} establishes a finite sample guarantee of our proposed Gaussian wild regenerative block bootstrap procedure for the strongly aperiodic case (i.e., independent blocks). In Section \ref{sec:application}, we apply our bootstrap consistency result to construct a uniform confidence band for the stationary density of certain class of diffusion process. The Appendix contains all the proofs presented in our main text. 

\subsection{Notation}
We use $\mathbb N$ and $\mathbb N_0 = \mathbb N \cup \{ 0 \}$ to denote the set of positive integers and the set of nonnegative integers, respectively. 
For $a,b \in \R$, we use the notation $a \vee b = \max \{a,b \}$ and $a \wedge b = \min \{ a,b \}$. 
For $a  > 0$, let $\lfloor a \rfloor$ denote the largest integer smaller than or equal to $a$. 
For a probability measure $Q$ on a measurable space $(E,\calE)$ and a measurable function $f: E \to \R$, we use the notation $Q f = \int_{E} f dQ$ whenever the latter integral is well-defined. 
In addition, for $q \in [1,\infty)$, let $\| \cdot \|_{Q,q}$ denote the $L^{q}(Q)$-seminorm, i.e., $\| f \|_{Q,q} = (Q(|f|^{q}))^{1/q}$ and for $\alpha \geq 1$, let $\| \cdot \|_{Q, \psi_\alpha}$ denote the Orlicz $\psi_\alpha$ norm with respect to $Q$. For a nonempty set $A \subset \R$ and $\delta > 0$, $A^{\delta}$ denote the $\delta$-enlargement of $A$, i.e., $A^{\delta} = \{ x \in \R: \inf_{y \in \R} |x-y| \le \delta \}$. For a stochastic process $\{ X_t \}_{t \in T}$, define $\|X\|_T := \sup_{t \in T}|X_t|$ and $X_T^\vee := \sup_{t \in T}X_t$. For a given set $S$, let  $\mathbf 1_S$ denote the indicator function of $S$, i.e, $\mathbf 1_{S}(x) =1$ if $x \in S$ and $=0$ otherwise. We write $a \lesssim b$ if $a$ is less or equal to $b$ up to some numerical constant and $a \asymp b$ if $a\lesssim b$ and $b \lesssim a$.

\section{Gaussian Approximation}
\label{sec:gaussian-approximation}

\subsection{Nummelin splitting technique}\label{sec: setting}

We start with reviewing Nummelin splitting technique, which plays a crucial role in our theoretical development. We refer the reader to \cite{Nummelin1984} and Chapter 17 in \cite{Meyn2012} for further details. 

Let $\{X_t\}_{t \in \mathbb N_0}$ be a Markov chain with values in a countably generated measurable space $(E,\calE)$, transition probability $P(x,dy)$ and initial distribution $\lambda$. 
We assume that the chain $\{ X_{t} \}_{t \in \mathbb N_0}$ is aperiodic and Harris recurrent with (unique) invariant probability measure $\pi$. Then, there exists a positive integer $m$, a set $S \in \calE$~(referred to as a \textit{small set}), a probability measure $\nu$ supported on $S$, and some constant $\theta  > 0$ such that 
\begin{equation}
P^{m}(x,A) \geq \theta \nu (A)\quad \ \text{for all} \ x \in S \ \text{and} \ A \in \calE,
\label{eq:minorization}
\end{equation}
where $P^{m}$ denotes the $m$-th iterate of $P$; cf. Proposition 5.4.5 in \cite{Meyn2012}. 

Nummelin splitting technique~\cite{Nummelin1978PTRF} embeds\footnote{$\bbP$ is a probability measure over the split chain with initial distribution $X_0 \sim \lambda$, so marginalizing it over $Y_t$ results in the probability measure of the original chain.} the original chain $\{ X_{t} \}_{t \in \mathbb N_0}$ into 
a split chain $\{ (X_{tm},Y_{t})  \}_{t \in \mathbb N_0}$, which takes values in $E \times \{ 0,1 \}$ and is defined by the conditional probabilities
\begin{equation}\label{eq:split-chain}
    \begin{aligned}
        &\bbP (Y_{t}=1, X_{tm+1} \in dx_{1},\dots,X_{(t+1)m-1} \in dx_{m-1},X_{(t+1)m} \in dy \mid Y_{0}^{t-1}, X_{0}^{tm-1}; X_{tm} = x) \\
        &= \bbP (Y_{t}=1, X_{tm+1} \in dx_{1},\dots,X_{(t+1)m-1} \in dx_{m-1},X_{(t+1)m} \in dy \mid X_{tm} = x) \\
        &= \bbP (Y_{0} = 1, X_{1} \in dx_{1},\dots,X_{m-1} \in dx_{m-1},X_{m} \in dy \mid X_{0}=x) \\
        &=\theta \frac{\mathbf 1_{S}(x)\nu (dy)}{P^{m}(x,dy)} P(x,dx_{1}) \cdots P(x_{m-1},dy).
    \end{aligned}
\end{equation}
Integrating the conditional probabilities with respect to $x_{1},\dots,x_{m-1}$ yields
\begin{equation*}
    \begin{split}
        &\bbP (Y_{t}=1, X_{(t+1)m} \in dy \mid Y_{0}^{t-1},X_{0}^{tm-1}; X_{tm}=x) = \theta \mathbf 1_{S}(x) \nu(dy).
\end{split}
\end{equation*}
By the Bayes rule, it follows that
\[
\begin{split}
&\bbP (Y_{t}=1 \mid Y_{0}^{t-1},X_{0}^{tm-1}; X_{tm}=x) = \theta \mathbf 1_{S}(x), \\
&\bbP (X_{(t+1)m} \in dy \mid Y_{0}^{t-1},X_{0}^{tm-1}; Y_{t}=1,X_{tm}=x) = \nu (dy), \quad \text{for} \ x \in S.
\end{split}
\]
This implies that, given $Y_t=1$,
\[
\text{$\{ X_{i},Y_{j} : i \le tm, j \le t \}$ is independent of $\{X_{i},Y_{j} : i \ge (t+1)m, j \ge t+1 \}$},
\]
Additionally, the distribution of the post $(t+1)m$ process is the same as the distribution of $\{ (X_t,Y_t) \}_{t \in \mathbb{N}_0}$ initialized at $(X_0,Y_0) \sim \nu \otimes \mathsf{Bern}(\theta)$, where $\mathsf{Bern}(\theta)$ denotes the Bernoulli distribution with success probability $\theta$. 

Set $\check{\alpha} = S \times \{ 1 \}$, which is an atom for the split chain $\{ (X_{tm},Y_t) \}_{t \in \mathbb N_0}$. Now, define the stopping times $0 \le \sigma_{\check{\alpha}}(0) < \sigma_{\check{\alpha}}(1) < \cdots$ as
\begin{equation}
\sigma_{\check \alpha} =  \sigma_{\check{\alpha}}(0) := \min \{ t \ge 0 : Y_{t} = 1 \}, \ \sigma_{\check{\alpha}}(i) := \min \{ t > \sigma_{\check{\alpha}}(i-1) : Y_{t} = 1 \}, 
\label{eq:block-stopping-times}
\end{equation}
and the corresponding random blocks
\begin{equation*}
    B_{i} := \big( X_{m(\sigma_{\check{\alpha}}(i-1)+1)},\dots,X_{m(\sigma_{\check{\alpha}}(i)+1)-1} \big), \quad i \in \mathbb N.
\end{equation*}
We think of the blocks $\{ B_i \}_{i \in \mathbb N}$ as random variables taking values in a torus $\check{E} = \bigcup_{k=1}^{\infty} E^{k}$~($E^k$ is the $k$-fold product space of $E$), and $\check E$ is equipped with the smallest $\sigma$-field $\check{\calE}$ containing the $\sigma$-fields $\calE^{k}$ of $E^k$ for $k\in \mathbb N$~(see \cite{Levental1988PTRF} for details). The following fact follows from the preceding discussion (see also Theorem 17.3.1 in \cite{Meyn2012}):
\medskip

\textbf{Fact}: The sequence of random blocks $\{ B_i \}_{i \in \mathbb N}$ is stationary, and one-dependent\footnote{By one-dependence, we mean that $\{ B_{i}\}_{i \leq k-2}$ and $\{ B_{i}\}_{i \geq k }$ are independent for every $k \geq 3$.} if $m >1$ and independent if $m =1$.
Set $\tau_{\check{\alpha}} := \min \{ t \geq 1 : Y_{t}  = 1 \}$, then the stationary distribution, which we denote by $\check{Q}$, agrees with the $\bbP_{\check \alpha}$-distribution\footnote{We use $\bbP_{\check{\alpha}}$ and $\E_{\check{\alpha}}$ to denote the conditional probability and expectation with respect to the split chain given $(X_0,Y_0) \in \check{\alpha}$. This is to be compared to $\bbP$ and $\E$, the probability and expectation of the split chain when $X_0 \sim \lambda$.} of $(X_{m},\dots,X_{m(\tau_{\check{\alpha}}+1)-1})$. 
\medskip

Despite the start and the end of the blocks being random, the fact that they are initiated from the regenerative distribution allows characterizing their distributions by the conditional distribution $\bbP_{\check \alpha}$ (and $\tau_{\check \alpha}$) where the start is fixed. In particular, $\E[\ell(B_i)] = m \E_{\check \alpha}[\tau_{\check \alpha}]$.

For notational convenience, the moments expressed by \textit{norms} of $\check f(B_1)$ will be denoted as $\| \check f \|_{\check Q, q}$ or $\|\check f\|_{\check Q, \psi_1}$, instead of $\|\check f(B_1)\|_{\bbP, q}$ or $\|\check f(B_1)\|_{\bbP, \psi_1}$; for the expectations and the variance of $\check f (B_1)$, we integrate over the probability $\bbP$, meaning that we write $\E[\check f(B_1)]$ or $\Var(\check f(B_1))$. Lastly, $\tau_{\check \alpha}$ need be integrated over probability $\bbPcheck$ as it is defined upon the conditional split chain.

We will use the following notation: for any function $f$ on $E$ and for any block $B=(x_{1},\dots,x_{k}) \in E^{k} \subset \check{E}$, define $\check{f}(B) := \sum_{t=1}^{k} f(x_{k})$.
Also, for each $B \in \check{E}$, let $\ell (B)$ denote the block length. Finally, the invariant  probability measure $\pi$ has the representation
\begin{equation*}
    \pi (A) = \frac{1}{m \E_{\check{\alpha}}[\tau_{\check{\alpha}}]} \E_{\check{\alpha}} \Bigg [ \sum_{t=m}^{m(\tau_{\check{\alpha}}+1)-1} \mathbf 1_{A}(X_{t}) \Bigg ], \quad A \in \calE.
\end{equation*}
See Theorem 10.0.1 in \cite{Meyn2012}; this implies that for any $\pi$-integrable function $f$ on $E$, 
\begin{equation}
    \E \big [ \check f (B_i) \big] = m\E_{\check \alpha}[ \tau_{\check \alpha} ]\pi f.
    \label{eq:block-mean-pi}
\end{equation}

\subsection{Gaussian approximation}
\label{sec : Gaussian approximation}

Suppose that we observe $X_{0},\dots,X_{n-1}$ from the Markov chain described in Section \ref{sec: setting}. Let $\calF = \calF_n$ be a class of $\pi$-integrable functions from $E$ into $\R$, and consider the empirical process
$\bbG_n =\{\bbG_n(f) \}_{f \in \calF}$ defined in (\ref{eq: empirical process}). 
 We assume that the function class $\calF$ admits a finite envelope $F$; that is, $F$ is a finite nonnegative measurable function on $E$ such that $F(x) \geq \sup_{f \in \calF} |f(x)|$ for all $x \in E$.  We are interested in approximating the distribution of the supremum of the empirical process, $\G_{n,\calF}^{\vee}$, by that of a suitable Gaussian process. We will assume a certain measurability condition on $\calF$ to ensure that $\G_{n,\calF}^{\vee}$ is a proper random variable. 

\newcommand{\cf}{f^\circ}
\newcommand{\ccf}{\check{f}^\circ}
\newcommand{\cg}{g^\circ}
\newcommand{\ccg}{\check{g}^\circ}

Define
\begin{equation*}
    \beta := m\E_{\check{\alpha}} [\tau_{\check{\alpha}}]\quad \text{and} \quad \cf := f - \pi f \quad \text{for every} \quad f \in \calF.    
\end{equation*}
Note that (\ref{eq:block-mean-pi}) yields $\E[\ccf(B_1)] = m\E_{\check{\alpha}}[\tau_{\check{\alpha}}] \cdot \pi  (f-\pi f) = 0$. Let $G = \{ G(f) \}_{f \in \calF}$ be a centered Gaussian process indexed by $\calF$ with covariance function given by: for all $f,g \in \calF$,
\begin{equation}\label{eq:covariance}
        \Sigma(f, g) = 
        \beta^{-1}\big \{ \Cov \big(\ccf(B_1), \ccg(B_1)\big) + \Cov \big(\ccf(B_1), \ccg(B_2)\big) + \Cov \big(\ccf(B_2), \ccg(B_1)\big) \big \}.
\end{equation}
When $m=1$, block independence yields $\Sigma (f,g) = \beta^{-1}\E[\ccf(B_1)\ccg(B_2)] = \beta^{-1}\E[ (\brf(B_1) - \ell(B_1) \pi f)(\brg(B_1) - \ell(B_1) \pi g) ]$. Such a Gaussian process exists by the Kolmogorov consistency theorem. In fact, it suffices to verify:
\begin{equation*}
    \textrm{for any $f_{1},\dots,f_{N} \in \calF$, the matrix $\big( \Sigma(f_i,f_j)\big)_{1 \le i,j, \le N}$ is positive semidefinite.}
\end{equation*}
This follows from the fact that the matrix coincides with the limit as $\ell \to \infty$ of the covariance matrix of the vector $\beta^{-1/2}(\ell^{-1/2} \sum_{i=1}^{\ell} \check{f}^\circ_{j}(B_{i}))_{1 \le j \le N}$.

In this paper, we focus on the case where the function class $\calF$ is {VC type}, whose formal definition is given below.

\begin{definition}[VC type class]\label{def:VC}
Let $\calF$ be a class of functions from $E$ into $\R$ with finite envelope $F$. 
We say that the class $\calF$ is \textit{VC type} with characteristics $(\sfA,\sfV)$ for some positive constants $\sfA$ and $\sfV$ if $\sup_{R} N(\calF,\| \cdot \|_{R,2}, \varepsilon \| F \|_{R,2}) \le (\sfA/\varepsilon)^{\sfV}$ for all $0 < \varepsilon \le 1$, where $\sup_R$ is taken over finitely discrete distributions $R$.
\end{definition}


Now, we make the following assumption. 
\begin{assumption} \label{assump : Gauss approx} ~
\begin{enumerate}
\item[(A1)] The function class $\calF$ is \textit{pointwise measurable}; that is, there exists a countable subclass $\calG \subset \calF$ such that for any $f \in \calF$ there exists a sequence $g_{k} \in \calG$ with $g_{k} \to f$ pointwise. 
\item[(A2)] The function class $\calF = \calF_n$ is VC type with characteristics $(\sfA_n,\sfV_n)$ for some constants {$\sfA_n \ge 4$ and $\sfV_n \ge 1$} that may increase with the sample size $n$.
\item[(A3)] The stopping times satisfy the moment conditions $\| \sigma_{\check \alpha} \|_{\bbP, \psi_1}<\infty$ and $ \|\tau_{\check{\alpha}}\|_{\bbP_{\check \alpha}, \psi_1} < \infty$. Furthermore, there exists a finite constant $\underline{\sigma} > 0$ and a (polynomially growing) sequence $D_n \ge 1$ such that
\begin{equation*}
    \begin{aligned}
       &\max\Bigg\{\| \check{F}  \|_{\check Q,\psi_1},\ \E \Bigg [ \sum_{t=0}^{m(\sigma_{\check{\alpha}}(0)+1)-1} F(X_{t}) \Bigg] \Bigg\} \leq D_n, \\
        &\inf_{f \in \calF} \big \{ \Var \big( \ccf(B_1)\big) + 2\Cov \big (\ccf(B_1),\ccf(B_2)\big) \big \}\ge \underline{\sigma}^{2}.
    \end{aligned}
\end{equation*}
\end{enumerate}
\end{assumption} 

Condition (A1) guarantees that the supremum $\bbG_{n,\calF}^\vee$ is a proper random variable.
See Example 2.3.4 in \cite{VW1996} for further discussion. 
Condition (A2) is concerned with the \textit{complexity} of the function class $\calF = \calF_n$. 
A similar condition is assumed in \cite{chernozhukov2016empirical} for the i.i.d. case. 
In principle, we can work with a more general complexity assumption, but for the sake of concise presentation, we shall keep Condition (A2). 

The moment conditions for the stopping times are typically ensured by certain drift conditions; see Section \ref{sec : drift cond} ahead for details. The variance  lower bound $\underline \sigma$ guarantees anti-concentration of the Gaussian supremum \cite{CCK2014bAoS,CCK2015PTRF}. 

Define the centered function class $\check \calF^\circ := \{ \ccf: f \in \calF \}$ with envelope $\check F (\cdot) + \ell(\cdot) \pi F$; then (\ref{eq:block-mean-pi}) yields $\pi F \leq C \| \check F\|_{\check Q, \psi_1}$ for some constant $C$ depending only on $m$ and $\|\tau_{\check \alpha}\|_{\bbPcheck, \psi_1}$, so Condition (A3) implies that the $\|\cdot\|_{\bbPcheck, \psi_1}$ norm of the envelope of the function class $\check \calF^\circ$ is bounded by $D_n$ up to a constant, again depending on $m$ and $\|\tau_{\check \alpha}\|_{\bbPcheck, \psi_1}$. Similarly we have $\E\big[ \sum_{t = 0}^{m(\sigma_{\check \alpha}(0) + 1) - 1} \{F(X_t) + \pi F\} \big] \leq CD_n$ for a positive constant $C$ that depends only on $m$ and $\| \sigma_{\check \alpha} \|_{\bbP, \psi_1}$. 

Finally, define the shorthand
\begin{equation*}
    K_{n} := \mathsf{V}_n \log \big({\mathsf{A}_n} \vee n \big).
\end{equation*}
We are ready to present the non-asymptotic Gaussian approximation of $\bbG_{n, \calF}^\vee$ under the Kolmogorov distance.

\begin{theorem}[Gaussian approximation]
\label{thm:Gaussian-approximation}
Suppose the chain $\{X_t\}_{t \in \mathbb N_0}$ and the function class $\calF$ satisfy Assumption \ref{assump : Gauss approx}. Set
\begin{equation}
\eta_n = 
    \begin{cases}
        \left( \frac{D_n^4K_n^5}{n} \right)^{1/4}, \quad &\textrm{when $m = 1$}, \\
        \left(\frac{D_n^6K_n^7m^4}{n}\right)^{1/6}, \quad &\textrm{when $m > 1$}.
    \end{cases}
    \label{def : cck rate}
\end{equation}
Given positive constants $\underline c, \underline \gamma > 0$, let $\delta_n$ be a nonincreasing sequence and $M_n$ a nondecreasing sequence satisfying $M_n \geq n^{\underline \gamma/2}$ and
\begin{equation}
    \begin{aligned}
        \frac{n^{1/4}\delta_n}{\sqrt{M_n}D_n K_n} \geq n^{\underline c}.
        \label{assumption for Gauss approx}
    \end{aligned}
\end{equation}
Then, for each value of $m$, there exists a zero-mean Gaussian process $G$  with covariance function $\Sigma$ in (\ref{eq:covariance}) that is a tight random variable in $\ell^\infty(\calF)$, such that
\begin{equation}
\rho_n \leq C (\eta_n + \xi_n),
\label{eqn : Gauss approx bound}
\end{equation}
where 
\begin{equation*}
    \begin{aligned}
    \xi_n :=& \delta_n \sqrt{K_n} +  \frac{D_n K_n}{\delta_n\sqrt{n}} + e^{-cn^{\underline \gamma\wedge (\underline c/2)}}.
    \end{aligned}
\end{equation*}
The constants $c, C > 0$ depend only on $\|\sigma_{\check \alpha}\|_{\bbP, \psi_1}, \|\tau_{\check \alpha}\|_{\bbPcheck, \psi_1}, \underline{\sigma}$, and $m$. 

\end{theorem}

    

Note that the term $\xi_n$ in (\ref{eqn : Gauss approx bound}) is independent of the value of $m$, which implies that $\eta_n$ is dominant when $m > 1$. By letting $D_n, K_n$ grow sufficiently slowly, we can achieve a near $n^{-1/4}$ rate when $m = 1$ that is comparable to \cite{chernozhuokov2022improved}. The following corollary presents an interpretable version of the Gaussian approximation when $m = 1$.

\begin{corollary}
\label{cor:simple-Gauss-approx}
Consider the setting in Theorem \ref{thm:Gaussian-approximation} under the strongly aperiodic case $m = 1$~(i.e. independent blocks). Fix a constant $\varepsilon \in (0, 1/4)$. Assume logarithmic growth of $K_n$ and let the sequence $D_n$ satisfy $D_n \leq n^{\varepsilon-\underline{\gamma}/4}$. Then the bound (\ref{eqn : Gauss approx bound}) reduces to
\begin{equation*}
    \rho_n \leq \frac{C\log^c (n)}{n^{1/4 - \varepsilon}}
\end{equation*}
where $c, C > 0$ are constants depending on $\|\sigma_{\check \alpha}\|_{\bbP, \psi_1}, \|\tau_{\check \alpha}\|_{\bbPcheck, \psi_1}$ and $\underline \sigma$.

\label{cor: Simplify Theorem 1}
\end{corollary}

In Corollary \ref{cor: Simplify Theorem 1}, the constant $\underline\gamma$ introduced in Theorem \ref{thm:Gaussian-approximation} can be arbitrarily small as it contributes to the bound of $\rho_n$ via $\exp(-n^{\underline \gamma})$. A smaller constant $\varepsilon$ would allow a near $n^{-1/4}$ rate at the cost of restricting the growth of the parameter $D_n$.

\begin{remark}[Comparison of rates for strongly mixing chains]
Notably, Harris recurrent Markov chains are strongly mixing~(i.e. $\alpha$-mixing; see \cite{bradley2005basic}). \cite{ChangChen2024} considers Gaussian approximation under the Kolmogorov metric for finite high-dimensional vectors that are strongly mixing; the work presumes identical moment conditions as stated in Assumption \ref{assump : Gauss approx}. On the other hand, here we do not necessarily require a geometric mixing condition. Further we are interested in an infinite-dimensional characterized by VC-type parameters $(\sfA_n, \sfV_n)$, which are allowed to grow with sample size. Notably, under the simplified setting stated in Corollary \ref{cor:simple-Gauss-approx}, our approximation rate achieves a near $O(n^{-1/4})$. This is to be compared to the rate achieved in \cite{ChangChen2024}, that scales as $O(n^{-1/9})$.
\end{remark}




The general strategy for finding an upper bound for $\rho_n$ consists of three steps. First, we decompose $\sum_{i = 0}^{n-1}f(X_i)$ into a sum of blocks, i.e., $\sum_{i = 1}^{n^*} \check{f}(B_i)$ for some fixed index $n^*$, and truncated remainder terms~(incomplete blocks). Second, we discretize the function class into finite number of functions $\{ f_j \}_{j \in [N]}\subset \calF$, which again induces remainder terms. Then high-dimensional Gaussian approximation results \cite{chernozhuokov2022improved}, \cite{ChangChen2024} are applied to $\sum_{i =1 }^{n^*}\check{f}_j(B_i)$, $j \in [N]$, while the remainder terms from decomposition and discretization are controlled through careful concentration inequalities. 


\section{Wild regenerative block bootstrap}
\label{sec:bootstrap}

\subsection{Approximate Nummelin splitting technique}
\label{sec:approximate-chain}
We review the approximate Nummelin splitting technique for strongly aperiodic Harris recurrent Markov chains (i.e., independent random blocks). We use this technique to show non-asymptotic consistency of our proposed Gaussian wild regenerative block bootstrap procedure. We refer the reader to \cite{BertailClemencon2006Bernoulli} for further details on approximate Nummelin splitting.

For simplicity, let the step transition probabilities $\{ P (x,dy)\}_{x \in E}$ and the initial distribution $\lambda$ be dominated by a $\sigma$-finite reference measure $\mu$, so that $\lambda(dy) = \lambda'(y)\mu(dy)$ and $P(x,dy) = p(x,y)\mu(dy)$ for all $x \in E$, where $\lambda'$ and $p$ are densities. Note that the minorization condition implies that
$\nu$ is also absolutely continuous with respect to $\mu$, and
that $p(x,y)\geq\theta \nu'(y),$ $\mu(dy)$ a.s. for every $x\in S$, with
$\nu(dy)=\nu'(y)\mu(dy)$.

The split chain constructed in Section \ref{sec: setting} satisfies a tensorization property; that is, the distribution of $Y_{0}^{n}=(Y_{0},\cdots, Y_{n})$
conditioned on $X_{0}^{n + 1}= x_{0}^{n + 1} $ is the product of
Bernoulli distributions. For $b_{0 }^{ n}=(b_{0}
,...,b_{n})\in\left\{  0,1\right\}  ^{n + 1},$ we have
\[
\mathbb{P}\left(  Y^{n}_0=b_0^n\mid X_0^{n+1}=x_0^{n+1}\right)
=\prod_{t=0}^{n}\mathbb{P}\big(Y_{t}=b_{t}\mid X_{tm}=x_{tm},X_{(t+1)m}=x_{(t+1)m}\big),
\]
with, for $t= 0, 1, ..., n$ and transition density $p$,

\begin{equation}
\begin{cases}
\mathbb{P}(Y_{t}=1\mid X_{t}=x_{t}, X_{t+1}=x_{t+1})=\theta & \text{if $x_{t}\notin S$ } \\
\mathbb{P}(Y_{t}=1\mid X_{t}=x_{t},X_{t+1}=x_{t+1})=\theta\nu'(x_{t+1})/p(x_{t},x_{t+1}) &\text{if $x_{t}\in S$}.
\end{cases}
\label{eq:tensor}
\end{equation}

The \textit{approximate Nummelin split chain} is a simulated chain $\{(X_t, \widehat Y_t)\}_{t \in \mathbb N_0}$, where $\widehat{Y}_t$ is sampled from (\ref{eq:tensor}) with an estimated transition density $\hat p$. Given data $X_0^{n + 1} = x_0^{n + 1}$, we first construct an estimator $\hat p$ of the transition density $p$ that satisfies
  \begin{equation}
   \hat p(x,y) \geq \theta \nu'(y), \  \mu (dy) \text{-a.s. \quad and}\quad \hat p(x_{t}, x_{t+1})>0
   \label{cond : transition estimate}
  \end{equation}
for $t = 0, 1, ..., n$. Transition density $p$ in (\ref{eq:tensor}) is replaced with $\hat p$ and samples $\widehat{Y}_t$ are drawn from the distribution $\mathsf{Bern}(\theta\nu'(x_{t + 1})/\hat p(x_t, x_{t + 1}))$ whenever $x_t \in S$ and from $\mathsf{Bern}(\theta)$ otherwise. 

Approximate stopping times\footnote{For notational convenience we omit the subscript $\check \alpha$ that was present in $\sigma_{\check \alpha}(i)$ for the original split chain.} are given by
\begin{equation*}
    \hat{\sigma}(0) := \min \{ t \geq 0 : \widehat{Y}_{t} = 1 \} \text{\;\;and\;\;} \hat{\sigma}(i) := \min \{ t > \hat{\sigma}(i-1) : \widehat{Y}_{t} = 1 \}.
\end{equation*}
Using the approximate stopping times, the approximate random blocks are constructed as follows
\begin{equation*}
    \widehat B_{i} := \big(X_{\hat{\sigma}(i-1)+1}, X_{\hat{\sigma}(i-1)+2},\cdots, X_{\hat{\sigma}(i)}\big), \quad i \in \mathbb N.
\end{equation*}

\subsection{Gaussian multiplier bootstrap consistency}
\label{sec : pivotal stat}
For the strongly aperiodic chain of interest, we propose a bootstrap-based procedure to approximate the covariance function of the reference Gaussian process $G$, a target that is difficult to learn analytically.

Given a function $f:E \to \R$ and a probability measure $\pi$, define the sample mean
\begin{equation}\label{eq:emp-mean}
    \what{\pi}f := n^{-1}\sum_{i = 0}^{n-1}f(X_i);
\end{equation}
$\what{\pi}f$ is an estimator of the mean $\pi f = \int f d\pi$. Let $\zeta_1, \zeta_2, ...$ be i.i.d. standard Gaussian random variables independent of the data and let $\hat i_n := \max \{i\geq 0: \hat \sigma(i) \leq n-1\}$ be the count of approximate random blocks. We introduce the Gaussian wild regenerative block bootstrap statistics,
\begin{equation}
    \bbG_n^\zeta(f) := \frac{1}{\sqrt{n}} \sum_{i = 1}^{\hat{i}_n} \zeta_i \cdot \Big\{ \check{f}(\widehat{B}_i) - \ell(\widehat{B}_i)\widehat{\pi} f\Big\}, \quad \text{for} \  f \in \calF.
    \label{eq:multplier-boot}
\end{equation}
Notably, the statistics (\ref{eq:multplier-boot}) is computable from data as long as the function class $\calF$ is computable. Our goal is to distributionally approximate $\bbG_{n, \calF}^{\zeta, \vee} = \sup_{f \in \calF} \bbG_n^\zeta(f)$ to the Gaussian process $G$ under the Kolmogorov metric 
\begin{equation}\label{eq:boot-metric}
    \rho_n^* := \sup_{t \in \R} \Big| \Pr\big( \G_{n, \calF}^{\zeta, \vee} \leq t | X_0^{n-1} \big) - \Pr\big( G_{\calF}^{\vee} \leq t \big) \Big| 
\end{equation}
with high probability.

\begin{remark}[Multiplier bootstrap for Markov chain data]
The bootstrap statistics (\ref{eq:multplier-boot}) is centered by an estimator of $\ell(B_i) \pi f$, and this is to mimic the covariance of the reference Gaussian process $\E[G(f)G(g)] = \beta^{-1} \E[(\brf(B_1) - \ell(B_1) \pi f)(\brg(B_1) - \ell(B_1) \pi g)]$. Further note that the covariance of the Gaussian process is a natural result of using Nummelin splitting technique; see the proof of Theorem \ref{thm:Gaussian-approximation} for details. See \cite{CCK2017AoP,deng2020beyond} to compare how our block bootstrap procedure for Markov chain data differs from the typical Gaussian multiplier bootstrap techniques designed for independent data.
\end{remark}

Define a stopping time of the original chain $\{X_t\}_{t \in \mathbb N_0}$ on the small set $S$,
\begin{equation}\label{eq:stopping_original}
    \tau_S := \min\{ t \geq 1 : X_{t} \in S \}.
\end{equation}
Unlike the stopping time $\tau_{\check \alpha}$ of the split chain, the stopping time $\tau_S$ of the original chain is not regenerative, so the initial measure matters throughout. 

We now make a second set of assumptions, that ensures bootstrap consistency. 



\begin{assumption}\label{assump : bootstrap}  ~
    \begin{enumerate}
    \item[(B1)] There exists a constant $\underline{\sigma} > 0$ and a (polynomially growing) sequence $D_n \geq 1$ such that
    \begin{equation*}
        \begin{aligned}
            \sup_{x \in E} F(x)  \leq D_n, \quad \inf_{f \in \calF} \Var \big( \ccf(B_1)\big) \ge \underline{\sigma}^{2}.
        \end{aligned}
    \end{equation*}
    \item [(B2)] The moments of stopping times $\sup_{x \in S}\mathbb{E}_{x}[\tau^4_S], \| \tau_{\check \alpha} \|_{\bbPcheck, \psi_1}, \| \sigma_{\check \alpha}\|_{\bbP, \psi_1} $
    are finite\footnote{The expectation and probability $\E_x$ and $\bbP_x$ is with respect to the split chain conditioned on $X_0 = x$.}.
    \item [(B3)] Estimator $\hat p(x,y)$ of the transition density $p(x, y)$ satisfies (\ref{cond : transition estimate}). There exists a constant $\theta' > 0$ and a non-increasing sequence $\alpha_n \geq 0$ such that
    \begin{equation*}
    \mathbb{E}\bigg[\sup_{(x,y) \in (S \times S)}\big|\hat p(x,y) - p(x,y)\big|^2 \bigg] \leq \theta' \cdot\alpha_n.    
    \end{equation*}
    \item [(B4)] The density $\nu' = d\nu/d\mu$ satisfies $\inf_{x \in S} \nu'(x) > 0.$
    \item [(B5)] There exists some positive finite constant $R$ such that
    \begin{equation*}
        \sup_{x, y \in S}\big(|p(x, y)| \vee |\hat p(x,y)| \big) < R.
    \end{equation*}
\end{enumerate}
\end{assumption}

Conditions (B1) and (B2) collectively imply $\| \check F\|_{\check Q, \psi_1} \leq \| F\|_\infty \| \tau_{\check \alpha} \|_{\bbPcheck, \psi_1}  \leq C D_n$ and $\E[\sum_{t = 0}^{\sigma_{\check \alpha}(0)}F(X_t)] \leq C'D_n$ where $C, C'$ depend on $\|\tau_{\check \alpha}\|_{\bbPcheck, \psi_1}$ and $\|\sigma_{\check \alpha}\|_{\bbP, \psi_1}$ respectively. So (B1) and (B2) together imply condition (A3) in Assumption \ref{assump : Gauss approx} when $m = 1$. The moment condition of stopping times in (B2) are again ensured by specific drift conditions; see Section \ref{sec : drift cond} ahead for details. The variance lower bound $\underline \sigma$ in (B1) guarantees anti-concentration of the Gaussian supremum \cite{CCK2014bAoS,CCK2015PTRF}.


The sequence $\alpha_n$ of the estimator $\hat p$ satisfying (B3) depends on the smoothness of the transition density $p$; see \cite{lacour2007adaptive, clemenccon2000adaptive} for further details. {It is known that $s$-Holder transition densities induce an order $\alpha_n = O\big((\log n/n)^{s/(s + 1)}\big)$; see Remark 3.1 of \cite{BertailClemencon2006Bernoulli}}. Condition {(B4)} along with (\ref{cond : transition estimate}) subtly impose a strict positivity condition for the estimator $\hat p$, meaning that
for consistency to hold~(that is, $\alpha_n \downarrow 0$), lower bound $\inf_{x \in S}p(x) >0$ is necessary. Condition (B5) is usually satisfied when transition density $p$ is bounded above. 

\begin{remark}[Assumption \ref{assump : bootstrap}]\label{rem:bertail}
    Assumption \ref{assump : bootstrap} allows a (carefully designed) coupling of the split chain $(X_i, Y_i)$ to well approximate that of the approximate split chain $(X_i, \hat{Y}_i)$. Specifically, the absolute difference $\big|\sum_{i = 1}^{\hat i_n}\brf^2(\what B_i) - \sum_{i = 1}^{i_n} \brf^2(B_i) \big|$ can be bounded by the transition density estimation error $O(\alpha_n)$; see Theorem 3.1 of \cite{BertailClemencon2006Bernoulli} for details.
\end{remark}


For notational simplicity, we introduce shorthands.
The collection of parameters appearing in Assumption \ref{assump : bootstrap}, are collected into a single vector\footnote{$\sup_{x \in S}\E_x[\tau_S^4] < \infty$, along with $\lambda$-integrability imply $\E[\tau_S^4] < \infty$. }
\begin{equation*}
    \vartheta_B := (q, \theta', \sup_{x \in S}\E_x[\tau_S^2], \sup_{x \in S}\E_x[\tau_S^4], \E[\tau_S^4]) \quad \text{where} \quad q := \theta \inf_{x \in S}\nu'(x)/R.
\end{equation*}
Further define $\upperpif := \sup_{f \in \calF} | \pi f |$, which may depend on sample size $n$ through the function class $\calF = \calF_n$. 

From here on and after, we always assume point-wise measurability and the VC-type properties, (A1) and (A2) of Assumption \ref{assump : Gauss approx} respectively, of the function class $\calF$. So Assumption \ref{assump : Gauss approx} follows from Assumption \ref{assump : bootstrap}. We now present the non-asymptotic bootstrap consistency of our Gaussian wild regenerative block bootstrap statistics under the Kolmogorov metric.

\begin{theorem}[Bootstrap consistency] 

Suppose the chain $\{X_t\}_{t \in \mathbb N_0}$, the function class $\calF$ and the estimator $\hat p$ satisfy Assumption \ref{assump : bootstrap}. Given positive constants $\underline c, \underline \gamma, \gamma, \overline \gamma > 0$  such that $\underline \gamma < \gamma < \overline \gamma < 1/2$, let $M_n$ be a sequence satisfying $M_n \geq n^{\underline \gamma/2}$ and
\begin{equation}
     \frac{n^\gamma}{M_n D_n^2 K_n} \geq n^{\underline{c}}.
    \label{m1-bootstrap-cond}
\end{equation}
If $K_n \leq n^{\underline c/2}$ and  $u(\calF, \pi) \leq D_n$, then with probability at least $1 - C\cdot \xi_n^*$ we have
\begin{equation}\label{eq:bootstrap-bound}
    \rho_n^* \leq C'\cdot \bigg\{\frac{\log N}{n^{1/4 - \overline{\gamma}/2}} + \frac{D_n K_n}{n^{1/4 + \overline{\gamma}/2}} +  \exp\big( -c'n^{(1/2 - \overline \gamma)\wedge (\overline \gamma - \gamma)} \big) \bigg\}
\end{equation}
where
\begin{equation*}
    \xi_n^* := \frac{\sqrt{K_n}}{n^{1/4 - \overline \gamma/2}} + \frac{D_nK_n^{5/4}}{n^{1/4 + \overline \gamma/2}} + n^{1/2 - \gamma} D_n^2 \alpha_n  + \exp(-c n^{\underline{\gamma} \wedge (\underline{c}/8)}).
\end{equation*}
The constants $c,c', C, C' > 0$ depend only on $\beta$, $\underline \sigma$, $\underline{c}, \vartheta_B, \|\sigma_{\check \alpha}\|_{\bbP, \psi_1}, \|\tau_{\check \alpha}\|_{\bbPcheck, \psi_1}$. 
\label{thm: bootstrap consistency}
\end{theorem}



\begin{remark}[Theorem \ref{thm: bootstrap consistency}] 
Note that $n^{1/4 - \overline{\gamma}/2}$ is the dominant term as long as $D_n \sqrt{K_n} \leq n^{\overline{\gamma}}.$ So $\overline{\gamma}$ is the budget for the growth of sequences $D_n$, $K_n$. Larger budgets yields slower dominating rate, hence inducing a tradeoff. 
\end{remark}



\subsubsection{Technical tool: Covariance function approximation}
The process (\ref{eq:multplier-boot}) is a Gaussian process when conditioned on data as $\{\zeta_i\}_{i \in [n]}$ are exogenous~(i.e. independent of the data) standard Gaussian random variables. So conditioned on data, bounding the metric $\rho_n^*$ reduces to the comparison of two Gaussian distributions.

After discretizing over the class of functions, comparison of two finite-dimensional Gaussian distributions amounts to the comparison of their covariance matrices~(see Lemma \ref{lem:Gaussian-comparison}). Accordingly, we define the empirical covariance function and its distance from the covariance function: 
\begin{equation}\label{eq:emp-covariance}
    \begin{aligned}
            \widehat{\Gamma}(\brf, \brg) &:= \frac{1}{n}\sum_{i =1 }^{\hat{i}_n} \brf (\widehat{B}_i)\brg (\widehat{B}_i) \quad \text{for $f \in \calF$ and} \\
        \Delta(\check{\calF}) &:= \sup_{\brf, \brg \in \check{\calF}} \Big| \beta^{-1} \Cov\big(\brf(B_1), \brg(B_1)\big) - \widehat{\Gamma}(\brf, \brg) \Big|.
    \end{aligned}
\end{equation}
For the centered function class $\calF^\circ:=\{f^\circ : f \in \calF\}$, the following Lemma controls $\Delta(\check \calF^\circ)$, which serves as an important technical tool for Theorem \ref{thm: bootstrap consistency}\footnote{Notably, $\what{\Gamma}(\ccf, \ccg)$ in $\Delta(\check\calF^\circ)$ is not computable from data as $\pi f$ and $\pi g$ are present.}. 
\begin{lemma} 
Suppose the chain $\{X_t\}_{t \in \mathbb N_0}$, the function class $\calF$ and the estimator $\hat p$ satisfy Assumption \ref{assump : bootstrap}. 
Fix positive constants $\underline c, \underline \gamma, \gamma > 0$ such that $\underline \gamma < \gamma < 1/2$. Let $M_n$ be a non-decreasing seqeuence satisfying $M_n \geq n^{\underline \gamma/2}$ and
\begin{equation*}
    \frac{n^\gamma}{M_n D_n^2 K_n} \geq n^{\underline c}.
\end{equation*}
If $K_n \leq n^{\underline c/2}$ and $\upperpif \leq D_n$, then we have
\begin{equation}
\bbP\big(\Delta (\check \calF^\circ) > n^{-1/2 + \gamma}\big) \leq C\cdot \big\{ n^{1/2 - \gamma} {D}_n^2 \alpha_n + \exp(-c n^{\underline{\gamma} \wedge \underline{c}}) \big\}.
\label{lem 3: Delta bound}
\end{equation}
The constants $c$ and $C$ depend only on $\beta, \underline{c}, \vartheta_B, \|\sigma_{\check \alpha}\|_{\bbP, \psi_1}, \|\tau_{\check \alpha}\|_{\bbPcheck, \psi_1}$.
\label{lem:cov-est}
\end{lemma}


\begin{remark}[Lemma \ref{lem:cov-est}]\label{rem:lem-cov} The variance lower bound in (B1) of Assumption \ref{assump : bootstrap} is not utilized here, so the constants in the bound (\ref{lem 3: Delta bound}) do not depend on $\underline \sigma$. Further, for a transition density $p$ that is $s$-Holder smooth~\cite{lacour2007adaptive,clemenccon2000adaptive}, we may construct an estimator $\hat p$~(see \cite{BertailClemencon2006Bernoulli}) that yields order $\alpha_n = O\big((\log n/n)^{s/(s + 1)}\big)$. In such case, we can simplify the probability bound (\ref{lem 3: Delta bound}) to
\begin{equation*}
  \bbP\big(\Delta (\check \calF^\circ) > n^{-1/2 + \gamma}\big) \leq O\bigg( \frac{D_n^2 ( \log n )^{s/2(s + 1)}}{ n^{s/2(s + 1) -1/2 + \gamma} } \bigg).  
\end{equation*}
\end{remark}


\subsection{Drift condition for controlling moments of stopping times}
\label{sec : drift cond}
The moments conditions of stopping times in Assumptions \ref{assump : Gauss approx} and \ref{assump : bootstrap} can be conveniently verified using drift conditions. Here we provide the precise drift condition that is sufficient for the verification of moments conditions. We refer the reader to \cite{adamczak} for further details. In this section we only consider the strongly aperiodic Markov chains. 

Recall the stopping time $\tau_S = \min\{ t \geq 1 : X_{t} \in S \}$ of the original chain $\{X_t\}_{t \in \mathbb N_0}$. Consider the following two conditions.
\begin{enumerate}
    \item[(C1)] There exists some constant $c > 0$ such that
    \begin{equation*}\label{eq:drift-cond-one}
    \begin{aligned}
        d_0 : = \sup_{x \in S} \E_x[ \exp ( 2\tau_S / c ) ]< \infty , \quad \text{and} \quad \forall x \in E, \quad \|( 2\tau_S/c )\|_{\bbP_x, \psi_1} < \infty.
    \end{aligned}
\end{equation*}
    \item[(C2)] For some constant $c > 0$, there exists $b, \overline V > 0$ and a function $V : E \to \R_+$ with $\|V\|_\infty \leq \overline V$ such that
\begin{equation}\label{eq:drift-cond-two}
    \exp\{-V(x)\} P(\exp V) (x) \leq \exp \left\{ -\frac{2}{c} + b \mathbf 1_S(x) \right\},
\end{equation}
where $P f(x) = \int f(y) dP(x, dy)$.
\end{enumerate}
Condition (C2) is referred to as the \textit{geometric drift condition}~\cite{adamczak}. The following Lemma states that the moment conditions of $\tau_S$ in {(C1)} and the drift condition {(C2)} imply one another.

\begin{lemma}[Geometric drift condition \cite{adamczak}]\label{lem:geom-drift}
Consider a strongly aperiodic Harris recurrent chain $\{X_t\}_{t \in \mathbb N_0}$ defined in Section \ref{sec: setting}.

\begin{enumerate}
    \item[(i)] Condition {(C2)} implies condition {(C1)}, and furthermore,
    \[
    \sup_{x \in S} \| \tau_S \|_{\bbP_x, \psi_1} \leq \max\{ 1, (b + \overline V)/\log 2 \} \cdot c.
    \]
    \item[(ii)] Condition {(C1)} implies condition {(C2)} with $V(x) = \log ( G_S(x, c) )$, where
    \[
    G_S(x, c) := \E_x[ \exp\{ 2(\sigma_S + 1)/c \} ],
    \]
    with $b = 2\log d_0, \overline V = \log d_0$ and $\sigma_S := \inf\{ t \geq 0 : X_{t} \in S \}$.
\end{enumerate}
\end{lemma}

We are ready to provide a drift condition that suffices for the moment conditions of the stopping times $\tau_S$ and $\tau_{\check \alpha}$ in Assumptions \ref{assump : Gauss approx} and \ref{assump : bootstrap}.
\begin{proposition}\label{prop:geom-drift}
    Suppose the strongly aperiodic Harris recurrent Markov chain $\{X_t\}_{t \in \mathbb N_0}$ satisfies the geometric drift condition {(C2)}. Then the moment conditions of stopping times $\tau_S$ and $\tau_{\check \alpha}$ in Assumptions \ref{assump : Gauss approx} and \ref{assump : bootstrap} hold. 
\end{proposition}

\section{Application : Uniform confidence band for diffusion processes}
\label{sec:application}

\subsection{Samples from a regular diffusion process}\label{sec : intro of diff process}

Here we review a class of regular diffusion processes; we refer the reader to \cite{nickl2017nonparametric} for further details on the diffusion process of interest. 

The measurable function $b: [0, 1] \to \R$ is bounded, $\varrho : [0, 1] \to (0, \infty)$ is continuous and $q: [0, 1] \to \R$ satisfies $q(0) = 1, q(1) = -1$. The diffusion process $\{X_t\}_{t \geq 0}$ solves the following stochastic differential equation,
\begin{equation}
   dX_t = b(X_t) dt  + \varrho(X_t) dW_t + q(X_t) dL_t \quad \text{for} \quad t \geq 0;
   \label{eq:diffusion}
\end{equation}
$W_t$ is a standard Brownian motion and $L_t$ a non-anticipative continuous non-decreasing process which increases only when $X_t \in \{ 0, 1 \}$~(see Section 2.1 of \cite{nickl2017nonparametric}). We are interested in a diffusion process that further satisfies the following regularity condition: there exist constants $\overline B, \underline{\varrho} > 0$ such that $(\varrho, b) \in  \Theta$ where
\begin{equation}
    \begin{aligned}
        \Theta := \Big\{ (\varrho, b) &:  b', \varrho', \varrho'' \ \text{exists,} \ b(0) = b(1) = \varrho'(0) = \varrho'(1) = 0, \\
        &\max ( \| b\|_\infty, \|b'\|_\infty, \|\varrho\|_\infty, \|\varrho'\|_\infty, \|\varrho''\|_\infty ) \leq \overline B, \ \inf_{x \in [0, 1]} \varrho^2(x) \geq \underline \varrho \Big\}.
    \end{aligned}
    \label{eq:reg_diff}
\end{equation}
A diffusion process satisfying (\ref{eq:reg_diff}) has a unique stationary distribution~(see \cite{bass1998diffusions}), and the stationary density has a closed form
\begin{equation}\label{eq:diff-stationary-density}
    \pi(x) = \frac{1}{C_0\varrho^2(x)}\exp \bigg( \int_0^x \frac{2b(y)}{\varrho^2(y)} dy \bigg) \quad \text{for} \quad x \in [0, 1],
\end{equation}
with a normalizing constant $C_0$. Notably, (\ref{eq:reg_diff}) and (\ref{eq:diff-stationary-density}) jointly imply (i) the existence of constants $0 < \pi_\ell \leq \pi_u <\infty$ depending on $\underline \varrho$, $\overline B$ such that $\pi \in [\pi_\ell, \pi_u]$ and that (ii) $\pi$ is differentiable up to second order and that $\| \pi '' \|_\infty < \infty$. See \cite{nickl2017nonparametric,bass1998diffusions} for details.

We consider a low-frequency regime, meaning that the accessible data are the discrete samples $X_{\Delta 0}, X_{\Delta 1}, ..., X_{\Delta (n-1)}$ of (\ref{eq:diffusion}) where $\Delta>0$ is a fixed constant. From here on, we omit the subscript $\Delta$ for the low-frequency samples $X_{\Delta t}$ and simply write them as $X_t$ for $t = 0, 1, 2, ..., n-1$. Notably the process $\{X_t\}_{t \in \mathbb N_0}$ is ergodic and form a strongly aperiodic Markov chain~\cite{bass1998diffusions,nickl2017nonparametric}. Proposition 9 of \cite{nickl2017nonparametric} immediately yields the minorization condition for $\{X_t\}_{t \in [n]}$,
as its transition density $p_\Delta(x,y) = p_{\Delta, \varrho, b}(x, y)$~(absolutely continuous with respect to the Lebesgue measure on $\R$) is lower bounded; that is, there exist constants $K', K$ such that $0 < K' < K < \infty$ and $K' \leq p_\Delta(x, y) \leq K$ for any $x, y \in [0, 1]$. Further $p_\Delta$ is differentiable\footnote{i.e. smoothness parameter $s\geq 1$ appearing Remark \ref{rem:lem-cov}}; see Proposition 9 of \cite{nickl2017nonparametric} for more details on $p_\Delta$. The small set is $S = [0, 1]$ and for any measurable subset $A\subset[0, 1]$ and for a uniform measure $\nu$ on $[0, 1]$, we have
\begin{equation}\label{eq:minorization-diff}
    \int_{A} p_\Delta(x, y)dy \geq K' \nu(A) \quad \text{for all $x \in S$.} 
\end{equation}
Thus, the minorization condition (\ref{eq:minorization}) holds for the low-frequency samples with small set $S = [0, 1]$ and parameters $\theta = K'$, $m = 1$.

\subsection{Uniform confidence band}\label{sec : confidence band construction}
Here we introduce a confidence band which (i) simultaneously covers the stationary density $\pi$ of (\ref{eq:diffusion}) over a compact subset and (ii) whose confidence converges polynomially to the desired level. The multiplier bootstrap technique introduced in Section \ref{sec:bootstrap} is key to our proof. 

Our approach differs from the classical asymptotic approach which resorts to the Smirnov-Bickel-Rosenblatt~(SBR) condition~\cite{gine2009exponential}, a (often regarded stringent) condition that ensures the existence of limiting distributions for certain empirical processes. Careful use of Gaussian anti-concentration and concentration inequalities allow us to derive results in a fully non-asymptotic manner, enabling a polynomial convergence to the desired confidence level under less stringent conditions. 

Consider a (strictly smaller) compact subset $\mathcal X \subset [0, 1]$. Fix a symmetric and Lipschitz continuous kernel function $\kernel : \R \to [0, \infty)$, that satisfies $\int \kernel(x) dx = 1$ and $\sup_{x \in \R}|\kernel(x)| < \infty$. For a bandwidth $h = h_n > 0$ that depends on $n$, define a kernel function scaled by $h$ and centered by $x$,
\begin{equation}\label{eq:scale-center-kernel}
    \kernel_{x, h}(\cdot) := h^{-1} \kernel\Big( \frac{\cdot - x}{h} \Big).
\end{equation}


The kernel density estimator~(KDE) of $\pi(x)$ and its point-wise standard deviation $\hat \sigma_n(x)$ are respectively defined as\footnote{Recall the definition (\ref{eq:emp-mean}).} 
\begin{equation*}
    \what{\pi} \kernel_{x, h} = n^{-1} \sum_{i = 0}^{n-1} \kernel_{x, h}(X_i) \quad \text{and} \quad \hat \sigma_n^2(x) := \E\big[ \{ \bbG_n^\zeta(\kernel_{x, h}) \}^2 \big] = \frac{1}{n} \sum_{i = 1}^{\hat i_n} \Big\{ \check \kernel_{x, h}(\what B_i) - \ell(\what B_i) \what{\pi} \kernel_{x, h} \Big\}^2.
\end{equation*}

The key quantity of our confidence band is the following quantile
\begin{equation*}
\hat{c}_n(\alpha) := \inf \big\{ z : \bbP( \| \bbG^\zeta_n\|_{\classkernel^{(n)}} \geq z ) \leq \alpha \big\} \quad \text{where} \quad \studclass := \Big\{ \studkernel_{x, h} = \kernel_{x, h}/\hat \sigma_n(x): \ x \in \calX \Big\};
\end{equation*}
$\classkernel^{(n)}$ is the \textit{studentized} kernel class.
We now introduce our proposed confidence band for $\pi(x)$ over $x \in \calX$:
\begin{equation}\label{eq:unif-ci}
    \calI_\alpha(x) := \Big[ \what{\pi} \kernel_{x, h} - \hat c_n(\alpha) \hat{\sigma}_n(x)/\sqrt{n}, \ \what{\pi}\kernel_{x, h} + \hat c_n(\alpha) \hat{\sigma}_n(x)/\sqrt{n} \Big].
\end{equation}





\begin{remark}[Studentization]\label{rem:studentize}
    The estimator $\hat \sigma_n(x)$ is the appropriate term for studentizing, as it approximates the standard deviation of the scaled KDE, $\sqrt{n}\what \pi \kernel_{x, h}$, in a uniform manner (see Lemma \ref{lem:studentizing-rel-error}); refer to the proof of Theorem \ref{thm:Gaussian-approximation} and observe that the variance of scaled KDE, $\Var( \sqrt{n}\what{\pi}\kernel_{x, h} )$ approximates the term\footnote{Note that blocks are independent for the low-frequency samples and $n^*$ represents the average number of blocks appearing until the $n$th sample.}
    \begin{equation*}
    \Var\Big( n^{-1/2}\sum_{i = 1}^{n^*}\check \kernel_{x, h}(B_i) \Big) = \beta^{-1}\sigma^2(x) \quad \text{where $n^* := \beta^{-1}n$ and $\sigma^2(x) := \Var(\check \kernel_{x, h}(B_1))$.}
    \end{equation*}
\end{remark}



We now provide a polynomial coverage guarantee of $\mathcal I_\alpha(x)$. Whenever the estimator $\hat p_\Delta$ for the transition density $p_\Delta$ satisfies
\begin{equation}\label{eq:transition-est}
    \E\bigg[ \sup_{x, y \in [0, 1]} | \hat{p}_\Delta(x, y) - p_\Delta(x, y) |^2 \bigg] = O(\alpha_n),
\end{equation}
we have the following.

\begin{theorem}[Uniform confidence band]\label{thm : uniform confidence}
Suppose the confidence band $\mathcal I_\alpha(x)$ is constructed from the low-frequency samples $X_0, X_1, ..., X_{n-1}$ of (\ref{eq:diffusion}) with $(b, \varrho) \in \Theta$. The estimator $\hat{p}_{\Delta}$ satisfy conditions (\ref{cond : transition estimate}), (\ref{eq:transition-est}). Given positive constants $n_0, c, C$, for any $n \geq n_0$ we have
\begin{equation}\label{eq:unif-cover}
    \begin{aligned}
    \bbP\big( \forall x \in \mathcal{X}, \ \pi(x) \in \calI_\alpha(x) \big) \geq 1 - \alpha - Cn^{-c}.
    \end{aligned}
\end{equation}
The constants $n_0, C$ depend only on $\calX, \pi, \kernel, \theta, \vartheta_B, \|\sigma_{\check \alpha}\|_{\bbP, \psi_1}, \|\tau_{\check \alpha}\|_{\bbPcheck, \psi_1}$, whereas $c$ is a universal constant.
\end{theorem}

\begin{remark}[Theorem \ref{thm : uniform confidence}]
Notably, (\ref{eq:unif-cover}) implies that the confidence level of the band $\calI_\alpha$ converges to the pre-specified  $\alpha$ value at a polynomial rate. This is in contrast to the confidence band relying on the SBR condition~\cite{gine2009exponential}, which achieves a logarithmic convergence to $\alpha$ by relying on the existence of a limiting distribution of an empirical process.
\end{remark}

\subsubsection{Intuition to justify confidence band}
Here we provide a high-level idea on why our confidence band (\ref{eq:unif-ci}) works. 

The equivalence $\pi(x) \in \calI_\alpha(x) \iff \big| \sqrt{n} ( \what \pi \kernel_{x, h} - \pi(x) ) \big| \leq \hat c_n(\alpha) \hat{\sigma}_n(x)$ and the uniform approximation $\hat \sigma_n(x) \approx \beta^{-1/2}\sigma(x)$~(see Remark \ref{rem:studentize}), followed by undersmoothing the KDE so as to neglect bias, yields the approximation
\begin{equation*}
        \bbP(\forall x \in \calX:\pi(x) \in \calI_\alpha(x)) \approx \bbP( \|\bbG_n\|_{\classkernel} \leq \hat c_n(\alpha) );
\end{equation*}
the function class $\classkernel := \big\{ f_{x, h} = \kernel_{x, h}/(\beta^{-1/2}\sigma(x)) : \ x \in \mathcal{X} \big\}$ is the population version of the studentized class $\studclass$. So \textit{if} the low-frequency samples $\{X_t\}_{t \in \mathbb N_0}$ satisfy Assumptions \ref{assump : Gauss approx} and \ref{assump : bootstrap} over the function class $\classkernel$, then we may apply Gaussian approximation~(Theorem \ref{thm:Gaussian-approximation}) and bootstrap consistency~(Theorem \ref{thm: bootstrap consistency}) to observe the approximations
\begin{equation*}
        \bbP( \|\bbG_n\|_{\classkernel} \leq \hat c_n(\alpha) ) \approx \bbP( \| G \|_\classkernel \leq \hat c_n(\alpha) ) \quad \textrm{and}\quad \bbP( \|G\|_{\classkernel} \leq \hat c_n(\alpha) ) \approx \bbP( \| \bbG_n^\zeta \|_{\classkernel^{(n)}} \leq \hat c_n(\alpha) ).
\end{equation*}
Lastly, apply the definition that $\hat c_n(\alpha)$ is the quantile of $\| \bbG_n^\zeta \|_{\classkernel^{(n)}}$. 

Hence in this application, it is crucial to check whether the Markov chain $\{X_t\}_{t \in \mathbb N_0}$ from (\ref{eq:diffusion}) satisfies Assumptions \ref{assump : Gauss approx} and \ref{assump : bootstrap} over the function class $\classkernel$; the following Lemma shows that this is indeed the case.

\begin{lemma}\label{lem:diff-cond-check}
    Let $\{X_t\}_{t \in \mathbb N_0}$ be the low-frequency samples from (\ref{eq:diffusion}) with $(\varrho, b)\in \Theta$ and assume the transition density is estimated by $\hat p_\Delta$ that satisfies (\ref{cond : transition estimate}) and (\ref{eq:transition-est}) for a non-increasing sequence $\alpha_n$.
    Let the bandwidth be a sequence $h_n = n^{-a}$ for some constant $a> 0$. Then there exists a constant $n_0$, such that for $n\geq n_0$~(i.e., small enough bandwidth), the chain $\{X_t\}_{t \in \mathbb N_0}$ from (\ref{eq:diffusion}) with $(b, \varrho)\in\Theta$ satisfies the Assumptions \ref{assump : Gauss approx} and \ref{assump : bootstrap} over the function class $\classkernel$. The constant $n_0$ depends on $\beta, \pi, \kernel, \theta, \|\tau_{\check\alpha}\|_{\bbPcheck, \psi_1}, a$ and $\calX$.
\end{lemma}
\begin{remark}[Lemma \ref{lem:diff-cond-check}]
The uniform lower bound on the transition density for the process (\ref{eq:diffusion}), i.e. $\inf_{x, y}p_\Delta(x, y) \geq K' > 0$, enables the construction of an estimator $\hat{p}_\Delta$ that satisfies (\ref{cond : transition estimate}) and (B5) of Assumption \ref{assump : bootstrap} simultaneously. Also, the order of $\alpha_n$ depends on the smoothness of the transition density $p_{\Delta}$, and its smoothness (up to second order) is guaranteed as we consider a regular diffusion characterized by $\Theta$ in (\ref{eq:reg_diff})~(see \cite{nickl2017nonparametric} for details). 
\end{remark}

\section{Acknowledgments}
The authors are grateful to Professor Kengo Kato for invaluable technical discussions and guidance.


\bibliographystyle{alpha} 

\bibliography{WRBB}       

\newpage


\appendix

\section{Technical tools}
\label{sec:appendix}

This section collects several technical tools that will be used in the following proofs. 

\subsection{Useful established technical tools}

Here we introduce several established technical tools that will be used throughout the Appendix.

\begin{lemma}[Gaussian-to-Gaussian comparison]
\label{lem:Gaussian-comparison}
Let $Y = (Y_1, ..., Y_N)$, $Z = (Z_1, ..., Z_N)$ be centered Gaussian random vectors with covariance matrices $\Sigma_Y, \Sigma_Z$ respectively. Define the maximal entry-wise difference, $\Delta = | \Sigma_Y - \Sigma_Z |_\infty$. 

\begin{enumerate}
    \item[(i)] If $\min_{1\leq j\leq N}\mathrm{Var}(Z_j) \geq \underline{\sigma}^2$ for some $\underline{\sigma} > 0$, then
    \begin{equation*}
    \sup_{y \in \R^N} \Big| \bbP( Z_1 \leq y ) - \bbP( Z_2 \leq y ) \Big| \leq C \big( \Delta \log^2N \big)^{1/2}    
    \end{equation*}
    where $C$ only depends on $\underline{\sigma} > 0$.
    \item[(ii)] For every $f\in \mathcal{C}_2$,
    \begin{equation*}
    \left|\E\left[ f\Big(\max_{1 \leq j \leq N} Y_j\Big) \right] - \E\left[ f\Big(\max_{1 \leq j \leq N} Z_j \Big)\right]\right| \leq \|f''\|_\infty \Delta/2 + 2\| f'\|_\infty \sqrt{2 \Delta\log N}.        
    \end{equation*}
\end{enumerate}
\end{lemma}

\begin{proof}[Proof of Lemma \ref{lem:Gaussian-comparison}]

For item (i), see Proposition 2.1 of \cite{chernozhuokov2022improved}. For item (ii), see Theorem 1 and its corresponding Comment 1 in \cite{CCK2015PTRF}.

\end{proof}

\begin{lemma}[Levy's anti-concentration]\label{lem:levy}
Consider a separable centered Gaussian process $G(t), t \in T$  for $t$ in a semi-metric space $T$ with $\sigma_t^2 = \E[G^2(t)]$. Suppose $\underline{\sigma} = \inf_{t \in T} \sigma_t \leq \sup_{t \in T} \sigma_t =: \overline{\sigma}$. Then for all $\varepsilon > 0$, and for some constant $C> 0$, the following holds,
\begin{equation}
        \sup_{x \in \R} \bbP\big( | G_T^\vee - x |\big) \leq C\frac{\overline{\sigma}}{\underline{\sigma}^2} \varepsilon\Big\{ \E[\sup_{t\in T}|G(t)/\sigma_t|] + \sqrt{1 \vee \log (\underline{\sigma}/\varepsilon)} \Big\}.
\end{equation}
\end{lemma}

\begin{proof}[Proof of Lemma \ref{lem:levy}]

For an $N$-dimensional centered Gaussian vector $G(t_j)$, $j \in [N]$, when $\underline \sigma^2 \leq \min_{j \in [N]} \sigma_j^2 \leq \max_{j \in [N]}\sigma_j^2 \leq \overline \sigma^2$. Then the constant $C$ specified in item (ii) of Theorem 3 in \cite{CCK2015PTRF}, whenever $\min_{j \in [N]} \sigma_j^2 < \max_{j \in [N]}\sigma_j^2 $, has the closed form $C'{\overline\sigma/\underline \sigma^2}$ for some universal constant $C ' > 0$. Then exploit separability of the Gaussian process to finalize the result---see proof of Theorem 2.1 in \cite{CCK2014bAoS}.
    
\end{proof}

\begin{lemma}[Nazarov's inequality]
 \label{lem: Nazarov}
Let $W = (W_1,\dots,W_N)^T$ be a (not necessarily centered) Gaussian random vector such that $\min_{1 \le j \le N} \Var (W_j) \ge \underline{\sigma}^2$ for some $\underline{\sigma} > 0$. Then, for $W^{\vee} = \max_{1 \le j \le N}W_j$, we have
\begin{equation}
\Pr \big(t < W^{\vee} \le t+\delta \big) \le  \frac{\delta}{\underline{\sigma} } (\sqrt{2\log N}+2), \  t \in \R, \delta > 0. 
\label{eq: nazarov}
\end{equation}
\end{lemma}

\begin{lemma}[Montgomery-Smith inequality]
\label{lem: Montgomery-Smith}
Let $X_{1},\dots,X_{n}$ be i.i.d. random variables taking values in a measurable space $E$, and let $\calF$ be a class of measurable real-valued functions on $E$. Then for every $t > 0$, and for some constant $C > 0$,
\[
\Pr \left ( \max_{1 \le k \le n} \left \| \sum_{i=1}^{k} f(X_{i}) \right \|_{\calF}^{*} > t \right ) \le C \Pr \left (\left \| \sum_{i=1}^{n} f(X_{i}) \right \|_{\calF}^{*} > \frac{t}{C} \right ),
\]
where $\| \cdot \|_{\calF}^{*}$ denotes the measurable cover. 
\end{lemma}

\begin{proof}[Proof of Lemma \ref{lem: Montgomery-Smith}]
See Corollary 4 in \cite{MontgomerySmith1993PMS}, or Theorem 1.1.5 together with Remark 1.1.8 in \cite{delaPenaGine1999}. 
\end{proof}

 \begin{lemma}[Concentration inequality]\label{thm : talagrand}
Let $X_1,\dots,X_n$ be i.i.d. random variables taking values in a measurable space $E$ with common distribution $P$. 
 Let $\calF$ be a pointwise measurable class of functions $f : E \to \R$ with measurable envelope $F$ with $\| F \|_{P, \psi_a} < \infty$ for some $a \in (0, 1]$. Suppose that $P f = 0$ for all $f \in \calF$. Let $Z := \| \sum_{i = 1}^{n} f(X_i)\|_{\calF}$ and $\sigma^2  > 0$ be any positive constant such that $\sigma^2 \geq \sup_{f \in \calF}Pf^2$. Then for every $\alpha \in (0, 1)$ and $ x > 0$, 
 \begin{equation}
    \bbP\big ( Z \geq (1 + \alpha)\E[Z] + x \big)\leq \exp \left(-\frac{cx^2}{n\sigma^2} \right) + 3\exp \left(-c\cdot\bigg(\frac{x}{\|F\|_{P, \psi_a} \cdot \psi_a^{-1}(n) } \bigg)^a \right),
 \end{equation}
where $c$ is a positive constant that depends on $\alpha$ and $a$.

 \end{lemma}

 \begin{proof}[Proof of Theorem \ref{thm : talagrand}]
     See Theorem 4 of \cite{adamczak2008tail}.
 \end{proof}

\begin{theorem}[Strassen]\label{thm: Strassen}

Let $\mu, \nu$ be Borel probability measures on $\R$. Let $\varepsilon, \delta> 0$ and suppose $\mu(A) \leq \nu(A^\delta) + \varepsilon$ for every Borel subset $A \subset \R$. For a given random variable $V \sim \mu$, one can construct a random variable $W \sim \nu$ such that
\[ 
\bbP(|V -W| > \delta ) < \varepsilon. 
\]

\end{theorem}

\begin{proof}[Proof of Theorem \ref{thm: Strassen}]
    See Lemma 4.1 of \cite{CCK2014aAoS}.
\end{proof}

\begin{lemma}
Fix $\beta > 0$ and $\delta > 1/\beta$. Then for every Borel subset $A \subset \R$, there exists a smooth function $g = g_A:\R \to \R$ and a universal constant $C > 0$ so that $\|g'\|_\infty \leq \delta^{-1}$, $\| g''\|_\infty \leq C \beta\delta^{-1}, \|g'''\|_\infty \leq C \beta^2 \delta^{-1}$ and
\[ (1 - \varepsilon) 1_{A} (t) \leq g(t) \leq \varepsilon + (1 - \varepsilon)1_{ A^{3\delta}}(t)\]
where $\varepsilon = \varepsilon_{\beta, \delta} := \sqrt{e^{-\alpha}(1 + \alpha)}$ and $\alpha := \beta^2 \delta^2 - 1$. 
    \label{lem : soft approx}
\end{lemma}


\begin{proof}[Proof of Lemma \ref{lem : soft approx}]
    See Lemma 4.2 of \cite{CCK2014aAoS}
\end{proof}

\subsection{Preliminary results for random blocks}
Maximal inequality and concentration inequality for random blocks are introduced here. 

\subsubsection{Maximal inequalities for random blocks}
\label{sec: maximal inequality}
Recall the setting of Section \ref{sec: setting}.
Define the \textit{occupation measure} $M(B,dy)$ by
\[
M(B,dy) = \sum_{t=1}^{k} \delta_{x_{t}}(dy), \ B=(x_{1},\dots,x_{k}) \in E^{k} \subset \check{E}. 
\]
We associate to any function $f: E \to \R$ a function $\check{f}$ on $\check{E}$ defined by 
\[
\check{f}(B) = \int_{E} f(y) M(B,dy).
\]
Recall for each $B \in \check{E}$, we let $\ell (B)$ denote its size:
\begin{equation}\label{eq:block-length}
    \ell(B) = \int_{E} M(B,dy).
\end{equation}
For example, if $B=(x_{1},\dots,x_{k}) \in E^{k}$, then $\check{f}(B) = \sum_{t=1}^{k} f(x_{k})$ and $\ell (B) = k$. 

Recall that $\check{\calE}$ denotes the smallest $\sigma$-algebra containing 
$\calE^{k}, k=1,2,\dots$. 
 If $B(\omega)$ is a random variable with distribution $\check{Q}$ on $(\check{E},\check{\calE})$, then $M(B(\omega),dy)$ is a random measure; that is, $M(B(\omega),
dy )$ is a (counting) measure on $(E,\mathcal{E})$
and for every $A \in \calE$, $M(B(\omega),A)$ is a random variable (with values in $\mathbb{N}$).

\begin{lemma}
\label{lem: covering number}
Let $\calF$ be a class of real-valued functions on $E$ with finite envelope $F$, and set the function class $\check{\calF} = \{ \check{f} : f \in \calF \}$. Then, the following hold. 

\begin{enumerate}
\item[(i)] For every probability measure $\check{R}$ on $(\check{E},\check{\calE})$ with $\check{R}(\ell) < \infty$, we have 
\[
N(\check{\calF},\| \cdot \|_{\check{R},1},\varepsilon \| \check{F} \|_{\check{R},1}) \leq N(\calF,\| \cdot \|_{R,1},\varepsilon \| F \|_{R,1})
\]
for all $0 < \varepsilon \leq 1$, where 
\[
R(A) = \frac{1}{\check{R}(\ell)} \int_{\check{E}} M(B,A) \, \check{R}(dB), \ A \in \calE.
\]
\item[(ii)]  For all $0 < \varepsilon \leq 1$,
\[
\sup_{\check R} N(\check{\calF},\| \cdot \|_{\check{R},2},\varepsilon \| \check{F} \|_{\check{R},2}) \leq \sup_{R} N(\calF,\| \cdot \|_{R,2},(\varepsilon/4)^{2} \| F \|_{R,2}),
\]
where $\sup_{\check{R}}$ and $\sup_{R}$ are taken over all finitely discrete distributions on $\check{E}$ and $E$, respectively. In particular, if $\calF$ is VC type with characteristics $(\sfA,\sfV)$, then $\check{\calF}$ is VC type with characteristics $(4\sqrt{\sfA},2\sfV)$ for envelope $\check{F}$. 
\end{enumerate}
\end{lemma}

\begin{proof}[Proof of Lemma \ref{lem: covering number}] 
\underline{Item (i)}. Let $\{ f_{1},\dots,f_{N} \}$ be an $(\varepsilon \| F \|_{R,1})$-net for $(\calF,\| \cdot \|_{R,1})$, so that for any $f \in \calF$ there exists $j=1,\dots,N$ such that $\| f - f_{j} \|_{R,1} \le \varepsilon \| F \|_{R,1}$. Then
\[
\| \check{f} - \check{f}_{j} \|_{\check{R},1} \le \int_{\check{E}} \int_{E} |f(y)-f_{j}(y)| M(B,dy) \check{R}(dB) \le \check{R}(\ell) \varepsilon \| F \|_{R,1} = \varepsilon \| \check{F} \|_{\check{R},1}.
\]
This implies the desired conclusion. 

\underline{Item (ii)}.  By Part (i) and Problem 2.5.1 in \cite{VW1996}, for any given probability measure $\check{R}$ on $(\check{E},\check{\calE})$ with $\check{R}\ell < \infty$ (see definition (\ref{eq:block-length})), we have
\[
N(\check{\calF},\| \cdot \|_{\check{R},1},\varepsilon \| \check{F} \|_{\check{R},1}) \leq \sup_{R} N(\calF,\| \cdot \|_{R,1},(\varepsilon/4) \| F \|_{R,1}).
\]
The right-hand side does not decrease if we replace the $L^{1}(R)$ norm with the $L^{2}(R)$ norm by Problem 2.10.4 in \cite{VW1996}. 
In addition, by the second part of the proof of Theorem 2.6.7 in \cite{VW1996}, we have 
\[
\sup_{\check{R}}N(\check{\calF},\| \cdot \|_{\check{R},2},\varepsilon \| \check{F} \|_{\check{R},2}) \le \sup_{\check{R}} N(\check{\calF},\| \cdot \|_{\check{R},1},(\varepsilon/2)^{2} \| \check{F} \|_{\check{R},1}),
\]
This completes the proof.
\end{proof}

The lemma above and Corollary 5.1 in \cite{CCK2014aAoS} immediately imply the following proposition (precisely,
we decompose the blocks into odd and even blocks, and apply Corollary 6.1 in \cite{CCK2014aAoS} to the odd
and even blocks, respectively; recall that, as the blocks are one-dependent, the odd blocks are
independent, and so are the even blocks).

\begin{lemma}[Maximal inequality for blocks]
\label{lem: maximal inequality}
Let $\calF$ be a countable class of $\pi$-integrable functions on $E$ with finite envelope $F$ such that $ \pi f = $ for all $ f \in \calF$.
Suppose $\calF$ is VC type with characteristics $(\sfA,\sfV)$ with $\sfA \geq 4$ and $\sfV \geq 1$. In addition, let $\overline{\sigma}$ be any positive constant such that $\sup_{f \in \calF} \| \check{f} \|_{\check Q,2} \leq \overline{\sigma} \leq \| \check{F} \|_{\check Q,2}$. Then,
\[
\begin{split}
&\E \left [ \left \| \frac{1}{\sqrt{n}} \sum_{i=1}^{n}  \check{f}(B_{i}) \right \|_{\calF} \right ] \\
&\quad\lesssim \overline{\sigma}\sqrt{ \sfV \log \big(\sqrt{\sfA}\| \check{F} \|_{\check Q,2}/\overline{\sigma}\big)} + \frac{\| \max_{1 \leq i \leq n} \check{F}(B_{i}) \|_{\bbP,2}}{\sqrt{n}} \sfV \log \big (\sqrt{\sfA}\| \check{F} \|_{\check Q,2}/\overline{\sigma} \big),
\end{split}
\]
up to a universal constant. 
\end{lemma}

Recall that the $\bbP_{\check \alpha}$-distribution of $(X_{m},\dots,X_{m(\tau_{\check{\alpha}}+1)-1})$ is denoted as $\check{Q}$. For any function $\check f$ on $\check E$ and stationary distribution $\check Q$ of random blocks $B_i$, we use $\| \check f \|_{\check Q, q}$ instead of $\| \check f(B_i) \|_{\check Q, q}$; the same goes for $\psi_1$ norm. We also make clear that the probability $\bbP$, not $\check Q$, is used for the norm expression of moments $\| \max_{1 \leq i \leq n} \check{F}(B_{i}) \|_{\bbP,2}$, as all the blocks are considered.

Choosing $\overline{\sigma} = \sup_{f \in \calF} \| \check{f} \|_{\check Q,2} \vee (\| \check{F} \|_{\check Q,2}/\sqrt{n})$ in Proposition \ref{lem: maximal inequality}, we also obtain the following corollary, which is more useful in our applications. 

\begin{corollary}[Simplified maximal inequality for blocks]
\label{cor: maximal inequality}
Under the setting of Proposition \ref{lem: maximal inequality}, we have 
\[
\E \left [ \left \| \frac{1}{\sqrt{n}} \sum_{i=1}^{n}  \check{f}(B_{i}) \right \|_{\calF} \right ]  \lesssim \overline{\sigma}\sqrt{\sfV \log \big(\sqrt{\sfA} \vee n \big)} + \frac{\| \max_{1 \leq i \leq n} \check{F}(B_{i}) \|_{\bbP,2}}{\sqrt{n}} \sfV \log \big(\sqrt{\sfA} \vee n\big),
\]
up to a universal constant. 
\end{corollary}

\subsubsection{Concentration inequalities for block counts}
\label{sec: deviation inequality}

Define $n^{\sharp} := \left\lfloor{n/m-1}\right\rfloor$ and $i_{n} := \max\{ i \ge 0: \sigma_{\check{\alpha}} (i) \le n^{\sharp}\}$. For this section only, we set $\E_{\check \alpha}[\tau_{\check \alpha}]$ simply as $\beta$ (this is with abuse of notation, as $\beta$ is originally defined as $m\E_{\check \alpha}[\tau_{\check \alpha}]$).

\begin{lemma}[Concentration inequalities for block counts]
\label{lem : block concentration}
Suppose $\|\sigma_{\check \alpha}(0) \|_{\bbP, \psi_1} < \infty$ and $\| \tau_{\check \alpha} \|_{\bbPcheck, \psi_1} < \infty$ hold, then there exists a constant $c > 0$ such that
\begin{equation*}
    \begin{aligned}
            \bbP\Big( | i_n - \beta^{-1}n^\sharp | \geq t \sqrt{n^\sharp} \Big) &\leq 2\exp \bigg\{ -c \bigg( \frac{\beta(t \sqrt{n^\sharp} - 3)/2 - \E[\sigma_{\check \alpha}(0)]}{\| \tau_{\check \alpha} \|_{\bbPcheck, \psi_1}}\bigg)^2\bigg\}\\
            &+2 \exp \bigg\{ - c \min \bigg( \frac{\{ \beta (t \sqrt{n^\sharp} - 1) \}^2}{\lfloor \beta^{-1}n^\sharp + t \sqrt{n^\sharp}\rfloor \cdot \| \tau_{\check \alpha}\|_{\bbPcheck, \psi_1}^2}, \frac{\beta (t \sqrt{n^\sharp} - 1)}{\| \tau_{\check \alpha} \|_{\bbPcheck, \psi_1}} \bigg) \bigg\}
    \end{aligned}
\end{equation*}
for every $3/\sqrt{n^\sharp} < t \leq \beta^{-1}\sqrt{n^\sharp}$.
\end{lemma}


\begin{proof}[Proof of Lemma \ref{lem : block concentration}]
The proof is similar to the proof of Lemma 3.1 in \cite{bertail2018new}. 
Set $\Delta \sigma_{\check \alpha}(i) = \sigma_{\check \alpha}(i) - \sigma_{\check \alpha}(i-1)$ for $i=0,1,\dots$ with the convention of $\sigma_{\check \alpha}(-1)=0$.

\underline{Upper tail}. For any $k \leq n^{\sharp}$,
\begin{equation*}
    \bbP(i_{n}  \geq k) \leq \bbP (\sigma_{\check \alpha}(k)\le n^{\sharp}) \leq \bbP \bigg(\sum_{i=1}^{k}\Delta\sigma_{\check \alpha}(i)\le n^{\sharp} \bigg) =\bbP \bigg(\sum_{i=1}^{k}(\Delta\sigma_{\check \alpha}(i)
    -{\beta})\le n^{\sharp}-k {\beta} \bigg).
\end{equation*}
If $t>\sqrt{n^{\sharp}}(1-{\beta}^{-1})$, then $\bbP \big(i_{n} \ge {\beta}^{-1} n^{\sharp}+t\sqrt{n^{\sharp}}\big) =0$, 
while if $0 \leq t \leq \sqrt{n^{\sharp}}(1-{\beta}^{-1})$, then
\begin{equation*}
    \begin{split}
        &\bbP \Big(i_{n} \ge {\beta}^{-1} n^{\sharp}+t\sqrt{n}\Big) 
        \leq \bbP \Big(i_{n}\geq \lfloor {\beta}^{-1} n^{\sharp}+t\sqrt{n^{\sharp}} \rfloor \Big) \\
        &\leq \bbP \Bigg( \sum_{i=1}^{\lfloor {\beta}^{-1} n^{\sharp}+t\sqrt{n^{\sharp}}  \rfloor} (\Delta \sigma_{\check \alpha} (i) - {\beta}) \le n^{\sharp} - \lfloor {\beta}^{-1} n^{\sharp}+t\sqrt{n^{\sharp}} \rfloor {\beta} \Bigg). 
    \end{split}
\end{equation*}
Since
\begin{equation*}
    n^{\sharp} - \lfloor {\beta}^{-1} n^{\sharp}+t\sqrt{n^{\sharp}} \rfloor {\beta} 
    \leq n^{\sharp} - \big( {\beta}^{-1} n^{\sharp}+t\sqrt{n^{\sharp}} -1\big) {\beta} = - {\beta}(t \sqrt{n^{\sharp}} -1),
\end{equation*}
we have that
\begin{equation*}
    \bbP \Big(i_{n} \ge {\beta}^{-1} n^{\sharp}+t\sqrt{n}\Big) \leq \bbP \Bigg( -\sum_{i=1}^{\lfloor {\beta}^{-1} n^{\sharp}+t\sqrt{n^{\sharp}} \rfloor} \big(\Delta \sigma_{\check \alpha} (i) - {\beta} \big) \geq  {\beta}\big(t \sqrt{n^{\sharp}} - 1 \big) \Bigg).
\end{equation*}
Observe that $\Delta \sigma_{\check \alpha} (i), i=1,2,\dots$ are i.i.d. random variables, whose common distribution is the same as the $\bbP_{\check{\alpha}}$-distribution of $\tau_{\check{\alpha}}$. 


Under the condition $\| \tau_{\check \alpha} \|_{\bbPcheck, \psi_1} < \infty$, we have $\| \Delta \sigma_{\check \alpha}(i) - \beta \|_{\bbPcheck, \psi_1} \leq C\| \tau_{\check \alpha} \|_{\bbPcheck, \psi_1}< \infty$, so we may apply the Bernstein inequality; there exists some constant $c > 0$ such that
\begin{equation*}
    \bbP\Big( i_n \geq \beta^{-1}n^\sharp + t \sqrt{n} \Big) \leq 2 \exp \bigg\{ - c \min \bigg( \frac{\{ \beta (t \sqrt{n^\sharp} - 1) \}^2}{\lfloor \beta^{-1}n^\sharp + t \sqrt{n^\sharp}\rfloor \cdot \| \tau_{\check \alpha}\|_{\bbPcheck, \psi_1}^2}, \frac{\beta (t \sqrt{n^\sharp} - 1)}{\| \tau_{\check \alpha} \|_{\bbPcheck, \psi_1}} \bigg) \bigg\}.
\end{equation*}

\underline{Lower tail}. For any $k \le n^{\sharp}$, we have 
\begin{equation*}
    \begin{split}
    &\bbP\big(i_{n} \le k\big) =\bbP\big(\sigma_{\check \alpha}(k+1)> n^{\sharp}\big)=\bbP\bigg(\sum_{i=0}^{k+1}\Delta\sigma_{\check \alpha} (i) > n^{\sharp} \bigg) \\
    &=\Pr \Bigg(\sigma_{\check \alpha}(0) +\sum_{i=1}^{k+1}(\Delta\sigma_{\check \alpha} (i)-{\beta})> n^{\sharp}-(k+1){\beta} \Bigg).
    \end{split}
\end{equation*}
Hence, we observe
\begin{equation*}
    \begin{split}
    &\bbP \Big(i_{n} \leq {\beta}^{-1} n^{\sharp}-t\sqrt{n^{\sharp}}\Big)  
    \leq \bbP \Big(i_{n} \leq  \lfloor {\beta}^{-1} n^{\sharp}-t\sqrt{n^{\sharp}}  \rfloor + 1 \Big) \\
    &\leq \bbP \Bigg(\sigma_{\check \alpha} (0)+\sum_{i=1}^{\lfloor {\beta}^{-1} n^{\sharp}-t\sqrt{n^{\sharp}} \rfloor + 2}(\Delta\sigma_{\check \alpha} (i)-{\beta})> n^{\sharp}-\Big(\lfloor {\beta}^{-1} n^{\sharp}-t\sqrt{n^{\sharp}}  \rfloor + 2 \Big) \beta \Bigg).
\end{split}
\end{equation*}
Since 
\begin{equation*}
    n^{\sharp}-\Big(\lfloor {\beta}^{-1} n^{\sharp}-t\sqrt{n^{\sharp}} \rfloor + 2 \Big){\beta} \ge {\beta} \big(t\sqrt{n^{\sharp}} - 3\big),
\end{equation*}
we have that whenever $3/\sqrt{n^{\sharp}}< t \le {\beta}^{-1} \sqrt{n^{\sharp}}$, the following inequality holds
\begin{equation*}
    \begin{split}
    &\bbP \Big(i_{n} \leq {\beta}^{-1} n^{\sharp}-t\sqrt{n^{\sharp}} \Big) \leq \bbP \Bigg( \sigma_{\check \alpha} (0) +\sum_{i=1}^{\lfloor {\beta}^{-1} n^{\sharp}-t\sqrt{n^{\sharp}} \rfloor + 2}\big(\Delta\sigma_{\check \alpha} (i)-{\beta}\big) > {\beta} \big(t\sqrt{n^{\sharp}} - 3\big) \Bigg) \\
    &\leq \bbP \Big(\sigma_{\check \alpha} (0) > {\beta} (t\sqrt{n^{\sharp}} - 3)/2  \Big)  + \Pr \Bigg(\sum_{i=1}^{\lfloor {\beta}^{-1} n^{\sharp}-t\sqrt{n^{\sharp}} \rfloor + 2}\big(\Delta\sigma_{\check \alpha} (i)-{\beta}\big) > {\beta} \big(t\sqrt{n^{\sharp}} - 3\big)/2 \Bigg).
\end{split}
\end{equation*}


Under the conditions $\| \sigma_{\check{\alpha}}(0)\|_{\bbP, \psi_1} < \infty$ and $\| \tau_{\check \alpha} \|_{\bbPcheck, \psi_1} < \infty$, apply the Bernstein inequality. Then there exists universal constant $c > 0$ such that
\begin{equation*}
    \begin{aligned}
        \bbP\Big( i_n \leq \beta^{-1}n^\sharp - t \sqrt{n^\sharp} \Big) &\leq 2\exp \bigg\{ -c \bigg( \frac{\beta(t \sqrt{n^\sharp} - 3)/2 - \E[\sigma_{\check \alpha}(0)]}{\| \tau_{\check \alpha} \|_{\bbPcheck, \psi_1}}\bigg)^2\bigg\}\\
        &+2 \exp \bigg\{ - c \min \bigg( \frac{\{ \beta (t \sqrt{n^\sharp} - 3) \}^2}{(\lfloor \beta^{-1}n^\sharp - t \sqrt{n^\sharp}\rfloor + 2) \cdot \| \tau_{\check \alpha}\|_{\bbPcheck, \psi_1}^2}, \frac{\beta (t \sqrt{n^\sharp} - 3)}{\| \tau_{\check \alpha} \|_{\bbPcheck, \psi_1}} \bigg) \bigg\}.
    \end{aligned}
\end{equation*}
\end{proof}

\section{Proofs for Section \ref{sec:gaussian-approximation}}

\begin{proof}[Proof of Theorem \ref{thm:Gaussian-approximation}]

The proof is divided into several steps. In what follows, $c, c'$ and $C$ denote generic positive constants that depends only on $m, \beta, \underline{\sigma}, \underline c, \|\sigma_{\check \alpha}\|_{\bbP, \psi_1}, \|\tau_{\check \alpha}\|_{\bbPcheck, \psi_1}$; its value may change from place to place. Recall $\beta = m \E_{\check \alpha}[\tau_{\check \alpha}]$ and $f^\circ = f - \pi f$.


\medskip


\underline{Step 1} (Construction of a tight version of $G$).  Let $\{ W(\ccf) \}_{\ccf \in \check{\calF}^\circ}$ with $\check{\calF}^\circ = \{ \ccf : f \in \calF \}$ be a Gaussian process with mean zero and covariance function 
\begin{equation}\label{eq:W-cov}
\E\big[W(\ccf)W(\ccg)\big] = \E\big[\ccf(B_{1}) \ccg(B_{1})\big] + \E\big[\ccf(B_{1}) \ccg(B_{2})\big] + \E\big[\ccf(B_{2}) \ccg(B_{1})\big].
\end{equation}
Recall that (\ref{eq:block-mean-pi}) yields $\E[\ccf(B_i)]=0$ as $\E[\ell(B_i)] = m \E_{\check \alpha}[\tau_{\check \alpha}]$. So the above display is a scaled version of (\ref{eq:covariance}). The covariance function of the Gaussian process $W(\ccf)$ yields
\begin{equation*}
    \E\big[(W(\ccf) - W(\ccg))^2\big] = \|\ccf - \ccg\|_{\check Q, 2}^2 + 2\E\big[ (\ccf(B_1) - \ccg(B_1))(\ccf(B_2) - \ccg(B_2)) \big].
\end{equation*}
Note that the above display is the sum of covariances, as $\E[\ccf (B_1)] = \E[\brf(B_1)] - m\E[\ell(B_1)]$
By the Cauchy-Schwarz inequality and the fact that $B_1 \stackrel{d}{=} B_2$, we have the inequality
\begin{equation}
    \begin{aligned}
        \E\big[ ( W(\ccf) - W(\ccg))^2 \big] \leq 3\| \ccf  - \ccg\|_{\check{Q}, 2}^2.
    \end{aligned}
    \label{eq:d-continuous}
\end{equation}

Observe that the set of constant functions $\{ f : x \mapsto \pi f: f \in \calF \}$ is VC-type 
~(Lemma 2.6.15 of \cite{VW1996}); so the translated function class $\calF^\circ := \{ f - \pi f: f \in \calF \}$ is also VC-type~(Lemma 2.6.18 of \cite{VW1996}).
Without loss of generality and with abuse of notation\footnote{Recall that the original function class $\calF$ was VC-type with parameters $(\sfA_n, \sfV_n)$.}, we set the VC-type parameters of $\calF^\circ$ as $(\sfA_n, \sfV_n)$. Then by Lemma \ref{lem: covering number}, we observe $\check \calF^\circ$ is VC-type with parameters $(4\sqrt{\sfA_n}, 2\sfV_n)$. Additionally, Lemma \ref{lem: covering number} and Problem 2.5.1 of \cite{VW1996} yields
\begin{equation}
\log N(\check{\calF}^\circ,\| \cdot \|_{\check{Q},2}, \eta \| \check{F}^\circ\|_{\check Q, 2}) \le   C \sfV_n \log (\sfA_n/\eta), \quad 0 < \eta < 1,
\label{eq:entropy}
\end{equation}
where $\check F^\circ$ is the envelope of the class $\check \calF^\circ$, i.e. $\check F^\circ := \check F(\cdot) + \ell(\cdot) \pi F$.
Hence, Dudley's criterion on sample continuity of Gaussian processes~\cite{Dudley2014UCLT} yields that there exists a version of $\{ W(\ccf) \}_{\ccf \in \check{F}^\circ}$ that is tight in $\ell^\infty (\check{\calF}^\circ)$. The desired version for the process $G$ can be constructed by setting $G(f) = \beta^{-1/2}W(\check{f}^\circ)$ for $f \in \calF$. 

\underline{Step 2} (Decomposition). Observe the decomposition:
\begin{equation}
\sum_{t=0}^{n-1} f^\circ(X_{t}) = \sum_{i=1}^{i_{n}} \ccf(B_{i}) + \sum_{t=0}^{\{ m(\sigma_{\check{\alpha}}(0)+1)-1 \} \wedge (n-1)} f^\circ(X_{t})  + \sum_{t=m(\sigma_{\check{\alpha}}(i_{n})+1)}^{n-1} f^\circ(X_{t}).
\label{pre decomposition}
\end{equation}

Set $n^{*} := \beta^{-1}n^{\sharp}$ and $\beta_n = n/n^*$. Further observe that $|\beta_n - \beta| \le C/n$ and $|\beta_n^{-1}-\beta^{-1}| \le C/n$, while $\beta_n = \beta$ when $m = 1$.  The first term on the right-hand side of (\ref{pre decomposition}) can be further decomposed as
\[
\sum_{i=1}^{i_{n}} \ccf(B_{i}) = \sum_{i=1}^{n^{*}} \check{f}(B_{i}) + \big ( \mathbf 1 (i_n > n^*) - \mathbf 1 (i_n < n^*) \big) \times \sum_{i=i_{n} \wedge n^{*}}^{i_{n} \vee n^{*}} \ccf(B_{i}).
\]
Hence, the empirical process can be expressed as the sum of four parts,
\begin{equation}\label{eq:sum-decomp}
    \begin{aligned}
    \G_{n}(f) &=\sqrt{\frac{n^{*}}{n}} \frac{1}{\sqrt{n^{*}}} \sum_{i=1}^{n^{*}}\ccf(B_{i})  + \big ( \mathbf 1 (i_n > n^*) - \mathbf 1 (i_n < n^*) \big) \cdot \frac{1}{\sqrt{n}}\sum_{i=i_{n} \wedge n^{*}}^{i_{n} \vee n^{*}} \ccf(B_{i})  \\
    &\quad + \frac{1}{\sqrt{n}} \sum_{t=0}^{\{ m(\sigma_{\check{\alpha}}(0)+1)-1 \} \wedge {(n-1)}} f^\circ(X_{t})   + \frac{1}{\sqrt{n}} \sum_{t=m(\sigma_{\check{\alpha}}(i_{n})+1)}^{n-1}  f^\circ(X_{t})  \\
    &=:\beta_n^{-1/2}\G_{n}^{(1)} (\ccf) + \big ( \mathbf 1 (i_n > n^*) - \mathbf 1 (i_n > n^*) \big) \cdot \G_{n}^{(2)}(\ccf) + \G_{n}^{(3)}(f^\circ) + \G_{n}^{(4)}(f^\circ).    
    \end{aligned}
\end{equation}
Conclude that, for any $x \in \R$,
\begin{equation}
    \begin{split}
        \bbP \big( \bbG_{n,\calF}^{\vee} \le x\big) &\le \bbP \big( \beta_n^{-1/2} \bbG_{n,\check \calF^\circ}^{(1),\vee} \le x + 3\delta_n \big) + \bbP\big( \| \bbG_{n}^{(2)} \|_{\check \calF^\circ} > \delta_n \big) \\
        &\quad + \bbP\big( \| \bbG_{n}^{(3)} \|_{\calF^\circ} > \delta_n \big) + \bbP\big( \| \bbG_{n}^{(4)} \|_{\calF^\circ} > \delta_n \big). 
    \end{split}
\label{eqn : basic decomp}
\end{equation}
By Condition (A3) and inequality $\pi F \leq C \| \check F\|_{\check Q, \psi_1}$ from (\ref{eq:block-mean-pi}) yield $\E\big [ \| \bbG_{n}^{(3)} \|_{\calF^\circ} \big] \le n^{-1/2}D_n$, so that by Markov's inequality,
\begin{equation*}
    \bbP\big( \| \bbG_{n}^{(3)} \|_{\calF^\circ} > \delta_n \big) \leq \delta_n^{-1} n^{-1/2}D_n.
\end{equation*}
In addition, as $i_n \leq n$ by construction, we have
\begin{equation*}
    \sum_{t=m(\sigma_{\check{\alpha}}(i_{n})+1)}^{n-1}  f^\circ(X_{t}) \leq \check F(B_{i_n+1}) + \ell (B_{i_n + 1})\pi F \leq \max_{1 \leq i \leq n+1} \{\check F(B_i) + \ell(B_{i_n + 1})\pi F\},
\end{equation*}
so that $\E\big [ \| \bbG_{n}^{(4)} \|_{\calF^\circ} \big] \leq C n^{-1/2} D_n\log n$. Again, by Markov's inequality, 
\begin{equation*}
    \bbP\big( \| \bbG_{n}^{(4)} \|_{\calF^\circ} > \delta_n \big) \leq C \delta_n^{-1} n^{-1/2}D_n \log n.
\end{equation*}

In what follows, we will compare $ \bbP \big( \beta_n^{-1/2} \bbG_{n,\check \calF^\circ}^{(1),\vee} \leq x + 3\delta_n \big)$
 with $\bbP(G_{\calF}^{\vee} \leq x)$ in Step 3 and upper bound $\bbP(\| \bbG_{n}^{(2)} \|_{\check \calF^\circ} > \delta_n)$ in Step 4.

\underline{Step 3} (Discretization error). We first discretize the function class $\check \calF^\circ$. Let 
\begin{equation}\label{eq:disc-class}
    \check \calF_N^\circ := \{ \ccf_1,\dots, \ccf_N \}  \quad \text{and} \quad \check\calF_{\varepsilon}^\circ := \{ \ccf - \ccg : \| \ccf - \ccg \|_{\check{Q},2} < 2\varepsilon\| \check F^\circ\|_{\check Q,2} \} 
\end{equation}
where the first is the $(\varepsilon \| \check F^\circ\|_{\check Q,2})$-net of $\check\calF^\circ = \{\ccf : f \in \calF \}$ with respect to the norm $\| \cdot \|_{\check Q,2}$, with $\varepsilon = \varepsilon_n = n^{-1/2}$ and $N=N_n = N(\check{\calF}^\circ, \| \cdot \|_{\check{Q},2},\varepsilon \| \check F^\circ\|_{\check Q,2})$, i.e.,  for every $\ccf \in  \check \calF^\circ$ there exists $j=1,\dots,N$ such that $\| \ccf - \ccf_{j} \|_{\check{Q},2} \le \varepsilon \| \check F^\circ\|_{\check Q,2}$. The second term of (\ref{eq:disc-class}) encodes the discretization error. Then
\begin{equation*}
    \Big|  \G_{n,\check\calF^\circ}^{(1),\vee}  - \bbG_{n, \check \calF^\circ_N}^{(1), \vee} \Big| \le \|\bbG_{n}^{(1)}\|_{\check \calF^\circ_\varepsilon} \quad  \text{and} \quad \Big| W_{\check{\calF}^\circ}^{\vee}  - W_{\check \calF^\circ_N}^{\vee} \Big| \leq \| W\|_{ \check{\calF}^\circ_{\varepsilon}}.
\end{equation*}
where we define\footnote{$\| \bbG_n^{(1)}\|_{\check \calF^\circ_\varepsilon}$ is well defined due to the linearity of averages. However, $\|W \|_{\check \calF^\circ_\varepsilon}$ is with abuse of notation for simplicity.}
\begin{equation}\label{eq:gauss-disc-err}
    \|W\|_{\check \calF^\circ_{\varepsilon}} := \sup_{\ccf, \ccg: \|\ccf - \ccg\|_{\check Q, 2} < 2\varepsilon\|\check F^\circ\|_{\check Q, 2} } | W(\ccf) - W(\ccg) |.
\end{equation}
Notably $\{ W(\ccf) - W(\ccg)\}_{ (\ccf, \ccg): \| \ccf - \ccg\|_{\check Q, 2} < 2\varepsilon\|\check F^\circ \|_{\check Q, 2}} $ is a Gaussian process, as it is a linear transformation of the original Gaussian process $W$. 

Observe
\[
 \bbP \big( \beta_n^{-1/2} \bbG_{n,\check \calF^\circ}^{(1),\vee} \le x + 3\delta_n \big) \le \bbP \big ( \beta_n^{-1/2} \bbG_{n, \check \calF^\circ_N}^{(1), \vee} \le x + 4 \delta_n \big) + \bbP \big(  \|\bbG_{n}^{(1)}\|_{\check \calF^\circ_\varepsilon} > \beta_n^{1/2} \delta_n).
\]
Let $\tilde{W} =\big (\tilde W_j\big)_{1 \le j \le N}$ be a centered Gaussian vector with the same covariance matrix as $\big (\bbG_n^{(1)}(\ccf_j)\big)_{1 \le j \le N}$, i.e., 
\[
\E[\tilde{W}_j \tilde{W}_k] = \E[\ccf_j(B_1) \ccf_k(B_1)] + \frac{n^* - 1}{n^*}\E[\ccf_j(B_1) \ccf_k(B_2)] + \frac{n^* - 1}{n^*}\E[\ccf_j(B_2) \ccf_k(B_1)].
\]
Now under the variance lower bound condition in (A3) of Assumption \ref{assump : Gauss approx}, apply the high-dimensional central limit theorems for independent data when $m=1$, Theorem 2.1 in \cite{chernozhuokov2022improved}, and one-dependent data when $m > 1$, Corollary 1 in \cite{ChangChen2024}. As a result, for $\tilde{W}^\vee = \max_{1 \le j \le N}\tilde{W}_j$, we observe
\begin{equation*}
    \bbP \big ( \beta_n^{-1/2} \bbG_{n, \check \calF^\circ_N}^{(1), \vee} \le x + 4 \delta_n \big) \le \bbP \big( \beta_n^{-1/2} \tilde{W}^{\vee} \le x + 4\delta_n\big) + C\eta_{n}.
\end{equation*}

We shall compare $\beta_n^{-1/2} \tilde{W}^{\vee}$ with $\beta^{-1/2}W^{\vee}_{\check \calF^\circ_N} = \beta^{-1/2}\max_{1 \le j \le N}W(\ccf_j)$. From Proposition 2.1 in \cite{chernozhuokov2022improved}, it is sufficient to bound entry-wise maximum discrepency between the covariance matrices of $\big( W(\ccf_j) \big)_{j \in [N]}$ and $\big( \tilde W_j \big)_{j \in [N]}$ respectivley, which can be bounded as follows,
\begin{equation*}
    \max_{1 \le j,k \le N} \left |\beta_n^{-1}\E[\tilde{W}_j\tilde{W}_k] - \beta^{-1}\E[W(\ccf_j)W(\ccf_k)] \right| \le Cn^{-1}\E[\check F(B_1)^2] \le Cn^{-1}D_n^2.
\end{equation*}
Hence, by Proposition 2.1 in \cite{chernozhuokov2022improved} combined with the Nazarov inequality (Lemma \ref{lem: Nazarov}), one has
\[
\begin{split}
\bbP \big( \beta_n^{-1/2} \tilde{W}^{\vee} \le x + 4\delta_n\big) &\le \bbP\big( \beta^{-1/2}W^{\vee}_{\check \calF^\circ_N} \le x+4\delta_n\big) + Cn^{-1/2}D_n \log N \\
&\le \bbP\big( \beta^{-1/2}W^{\vee}_{\check \calF^\circ_N} \le x - \delta_n \big) + C \delta_n \sqrt{\log N} +  Cn^{-1/2}D_n \log N \\
&\le \bbP\big( \beta^{-1/2}W^{\vee}_{\check \calF^\circ} \le x \big) + \bbP\big( \| W \|_{\check \calF^\circ_{\varepsilon}} > \beta^{1/2}\delta_n \big) \\
&\qquad + C \delta_n \sqrt{K_n} +  Cn^{-1/2}D_n K_n, 
\end{split}
\]
where we used the fact that $\log N \le C K_n$.

It remains to upper bound the remainders for discretizing $\bbP \big(  \|\bbG_{n}^{(1)}\|_{\check \calF^\circ_\varepsilon} > \beta_n^{1/2} \delta_n)$ and $\bbP\big( \| W \|_{\check \calF^\circ_{\varepsilon}} > \beta^{1/2}\delta_n \big)$.

\newcommand{\calT}{\mathcal T}

The term $\bbP(\| W\|_{\check \calF^\circ_\varepsilon} > \beta^{1/2}\delta_n)$ is controlled using Markov inequality. The moment $\E[\|W\|_{\check \calF_\varepsilon^\circ}]$ is bounded by Dudley's entropy integral bound \cite{Dudley2014UCLT} and it is sufficient to bound the covering number of the function class $\{ (\check f^\circ, \check g^\circ): \| \check f^\circ - \check g^\circ \|_{\check Q, 2} < 2 \varepsilon \| \check F^\circ \|_{\check Q, 2} \}$ with respect to the $L_2$ norm of the Gaussian process $\{ W(\check f^\circ) - W(\check g^\circ)\}_{(\check f^\circ, \check g^\circ):\| \check f^\circ - \check g^\circ \|_{\check Q, 2} < 2\varepsilon \| \check F^\circ\|_{\check Q, 2}}$.

Note that (\ref{eq:d-continuous}) implies $\E\big[ ( W(\check f^\circ) - W(\check g^\circ) )^2 \big] \leq C' \varepsilon^2 D_n^2$ for any $\check f^\circ, \check g^\circ$ such that $\| \check f^\circ - \check g^\circ \|_{\check Q, 2} < 2 \varepsilon \| \check F^\circ \|_{\check Q, 2}$ and also implies that for any $f_1, f_2, g_1, g_2 \in\calF$, the inequality holds
\begin{equation*}
        \E\Big[\big\{\big(W(\check f^\circ_1) - W(\check g^\circ_1)\big) - \big(W(\check f^\circ_2) - W(\check g^\circ_2)\big)\big\}^2\Big] \leq 6 \big( \| \check f_1^\circ - \check f_2^\circ \|_{\check, Q, 2}^2 + \| \check g_1^\circ - \check g_2^\circ \|_{\check Q, 2}^2 \big);
\end{equation*}
this inequality allows us to bound the covering number of interest by the covering number of $\check \calF^\circ$ with respect to the $\|\cdot\|_{\check Q, 2}$ norm. Using (\ref{eq:entropy}), Dudley's entropy integral bound for Gaussian process~\cite{Dudley2014UCLT} yields
\begin{equation*}
    \begin{aligned}
        \E\big[ \|W\|_{\check{\calF}^\circ_\varepsilon} \big] &\leq C \int_0^{C' \varepsilon D_n} \sqrt{\sfV_n \log \left( \frac{\sfA_n D_n}{\eta} \right)}  d\eta \\
        &\leq C n^{-1/2}D_n \sqrt{\sfV_n \log (D_n \sfA_n)}   \leq C' n^{-1/2} D_n  K_n.
    \end{aligned}
\end{equation*}
So for some constant $C > 0$, we have
\begin{equation}
   \bbP\big(  \|W\|_{\check{\calF}^\circ_\varepsilon} > \delta_n \big) \leq C \delta_n^{-1} D_n n^{-1/2} K_n.
   \label{eq:gaussian-Markov}
\end{equation}

The term $\bbP( \| \bbG_n^{(1)} \|_{\check \calF^\circ_\varepsilon} > \beta_n^{1/2}\delta_n )$ can be further bounded by $\bbP( \|\bbG_n^{(1)} \|_{\check \calF^\circ_\varepsilon} > C'\delta_n )$ by using $| \beta_n/\beta - 1 |\leq c/n$. When blocks are one dependent,
\begin{equation}
    \begin{aligned}
    &\bbP\big( \|\bbG_{n}^{(1)}\|_{ \check\calF^\circ_\varepsilon}> C'\delta_n \big)\\
    \leq& \bbP\Bigg( \bigg\| \frac{1}{\sqrt{n^*}}\sum_{i \in [n^*]: \textrm{odd}}\check h^\circ(B_i)  \bigg\|_{\check{\calF}^\circ_\varepsilon} > \delta_n/2 \Bigg) + \bbP\Bigg(\bigg\| \frac{1}{\sqrt{n^*}}\sum_{i \in [n^*]: \textrm{even}}\check h^\circ (B_i) \bigg\|_{\check{\calF}^\circ_\varepsilon}> \delta_n/2 \Bigg)\\
    \leq&2\bbP\Bigg( \bigg\| \sum_{i \in [n^*]: \textrm{odd}}\check h^\circ(B_i)  \bigg\|_{\check{\calF}^\circ_\varepsilon} >\sqrt{n^*} \delta_n/2 \Bigg),
    \end{aligned}
    \label{gauss approx first concentration}
\end{equation}
where the last inequality implicitly assumes $|\{i \in [n^*]:\mathrm{odd}\}| \geq |\{i \in [n^*]:\mathrm{even}\}| $, which is without loss of generality. So bounding $\bbP (\| \bbG_n^{(1)}\|_{\check \calF^\circ_\varepsilon} > \delta_n)$ for both $m = 1$ and $m > 1$ are equivalent up a constant $C$.

We employ the Fuk-Nagaev type inequality~(Theorem \ref{thm : talagrand}) 
to bound the terms in (\ref{gauss approx first concentration}).
 The variance term satisfies $\sup_{\check h^\circ\in \check{\calF}^\circ_\varepsilon}\E\big[  \big(\check h^\circ(B_i)\big)^2 \big] \leq CD_n^2/n.$
 Define $Z = \| \sum_{i \in [n^*]: \textrm{odd}}\check h^\circ(B_i) \|_{\check{\calF}^\circ_\varepsilon},$ and recall that the summands $\check h^\circ(B_i)$ are centered. The inclusion $\check \calF^\circ_\varepsilon \subseteq \check \calF^\circ_{-}:=\{\ccf - \ccg : f, g \in \calF \}$ implies $N(\check \calF^\circ_\varepsilon, \|\cdot\|_{\discmeas, 2}, \eta \| 2\check F^\circ \|_{\discmeas, 2}) \leq N(\check \calF^\circ_{-}, \|\cdot\|_{\discmeas, 2}, \eta \| 2\check F^\circ \|_{\discmeas, 2})$ for any discrete measure $\discmeas$ on $\check E$. So observe
\begin{equation}\label{eq:VC-F-eps}
    N(\check \calF_{-}^\circ , \| \cdot \|_{\discmeas, 2}, \sqrt{2}\eta \| 2\check F^\circ \|_{\discmeas, 2}) \leq N^2(\check \calF^\circ, \| \cdot \|_{\discmeas, 2}, \eta\| \check F^\circ \|_{\discmeas, 2}) \leq \Big( \frac{4\sqrt{\mathsf{A}_n}}{\eta} \Big)^{2\mathsf{V}_n}, 
\end{equation}
and we conclude $\check \calF^\circ_\varepsilon$ is VC-type with $(2\sfV_n, 4\sqrt{2\sfA_n})$. Using Corollary \ref{cor: maximal inequality}, we have $ \E[Z] \leq C\sqrt{\log n}D_n K_n$. 

Note that the condition (\ref{assumption for Gauss approx}) yields 
\begin{equation*}
    \begin{aligned}
        \bbP\big( Z > \sqrt{n^*}\delta_n/2 \big) &\leq \bbP\big( Z > 3\E [Z]/2 + \big\{ \sqrt{n^*}\delta_n/2 - 3\sqrt{\log n}D_nK_n/2 \big\} \big) \\
        &\leq \bbP\big( Z > 3 \E [Z]/2 + c n^{\underline c} \sqrt{\log n} D_n K_n \big).
    \end{aligned}
\end{equation*}
Apply Fuk-Nagaev type concentration~(Theorem \ref{thm : talagrand}) and we have,
 \begin{equation*}
     \begin{aligned}
         \bbP\big( Z > \sqrt{n^*}\delta_n/2 \big) \leq \exp\left( -c n^{2\underline c} \log n K_n^2 \right) + 3\exp\left( -\frac{cn^{\underline c}K_n}{\sqrt{\log n}} \right).
     \end{aligned}
 \end{equation*}

\underline{Step 4} (Random index).  
Lastly we bound $\bbP(\| \bbG_n^{(2)}\|_{\check \calF^\circ} > \delta_n)$. We use the Montogomery-Smith inequality (see Lemma \ref{lem: Montgomery-Smith}) for random indices. The following set inclusion
\begin{equation}
    \begin{aligned}
        &\left\{ |i_n - n^*| \leq M_n \sqrt{n} \right\} \cap\bigg\{ \max_{n^* - M_n\sqrt{n} \leq k\leq n^* + M_n\sqrt{n}}  \bigg\| \sum_{i = n^* - M\sqrt{n}}^{k} \ccf(B_i) \bigg\|_{\check{\calF}^\circ}\leq \sqrt{n}\delta_n/2\bigg\}\\
        &\subset \bigg\{ \bigg\| \sum_{i = i_n \wedge n^*}^{i_n \vee n^*}\ccf(B_i) \bigg\|_{\check{\calF}^\circ} \leq \sqrt{n} \delta_n  \bigg\} \label{eq:set_inclusion}
    \end{aligned}
\end{equation}
implies the following decomposition
\begin{equation*}
    \begin{aligned}
         & \bbP\bigg( \max_{n^* - M_n\sqrt{n} \leq k\leq n^* + M_n\sqrt{n}}  \bigg\| \sum_{i = n^* - M_n\sqrt{n}}^{k} \ccf(B_i) \bigg\|_{\check{\calF}\circ} > \sqrt{n}\delta_n/2  \bigg) + \bbP\big(|i_n - n^*| > M_n \sqrt{n} \big)\\
        & = \bbP\bigg( \max_{1 \leq k \leq 2M_n\sqrt{n}} \bigg\|\sum_{i = 1}^{k} \ccf(B_i) \bigg\|_{\check{\calF}^\circ} > \sqrt{n}\delta_n/2 \bigg) + \bbP\big( |i_n - n^*| > M_n\sqrt{n} \big)\\
        &\leq 2 C \bbP\bigg(  \bigg\|  \sum_{i \in [2M_n \sqrt{n}]: \textrm{odd}}\ccf(B_i) \bigg\|_{\check{\calF}^\circ}  > c\sqrt{n}\delta_n \bigg) + \bbP\big( |i_n - n^*| > M_n \sqrt{n} \big) .
    \end{aligned}
\end{equation*}
Define $Z' :=  \| \sum_{i \in [2M_n\sqrt{n}]:\textrm{odd}}\ccf(B_i) \|_{\check{\calF}^\circ}$, so that Corollary \ref{cor: maximal inequality} yields
\begin{equation*}
     \E[Z'] \leq C M_n^{1/2}n^{1/4}D_n K_n \quad \text{and} \quad \sup_{\check f \in \check \calF} \E\big[ \big( \ccf(B_i)\big)^2 \big] \leq C' D_n^2.
\end{equation*}
Then the condition (\ref{assumption for Gauss approx}) yields
\begin{equation*}
    \begin{aligned}
        \bbP\big( Z' > c\sqrt{n}\delta_n\big) &\leq \bbP \big( Z' > 3 \E[Z']/2 + \big\{  c\sqrt{n}\delta_n - 3 M_n^{1/2}n^{1/4}D_nK_n/2  \big\} \big)\\
        &\leq \bbP \big( Z' > 3 \E[Z']/2 + c' n^{\underline c} M_n^{1/2}n^{1/4}D_n K_n \big).        
    \end{aligned}
\end{equation*}
Apply the Fuk-Nagaev type concentration~(Theorem \ref{thm : talagrand}) on $Z'$, and observe
\begin{equation*}
\begin{aligned}
    \bbP\big( Z' > c\sqrt{n}\delta_n \big) \leq \exp\big( -c' n^{2\underline c}K_n^2 \big) + 3 \exp \bigg( -\frac{c' n^{\underline c}M_n^{1/2}n^{1/4}}{\log n} \bigg).
\end{aligned}
\end{equation*}

Lastly, recall the assumption $\|\sigma_{\check{\alpha}}\|_{\bbP, \psi_1} \vee \|\tau_{\check{\alpha}}\|_{\bbPcheck, \psi_1} < \infty$ and the definition $n^* = \beta^{-1}n^\sharp$. Since $\sqrt{m}/2 \leq \sqrt{n}/\sqrt{n^\sharp} \leq \sqrt{m}$ hold, we observe the inequality
\begin{equation*}
    \begin{aligned}
        \bbP\big( |i_n - n^*| > M_n \sqrt{n} \big) \leq \bbP\Big( i_n > n^* + \{\sqrt{m} M_n / 2\} \cdot \sqrt{n^\sharp} \Big) + \bbP\Big( i_n > n^* - \{\sqrt{m}M_n/2\} \cdot \sqrt{n^\sharp}  \Big).
    \end{aligned}
\end{equation*}
Invoke Lemma \ref{lem : block concentration} so that
\begin{equation*}
    \bbP\big( |i_n - n^*| > M_n\sqrt{n} \big) \leq C \exp (-c n^{\underline \gamma}).
\end{equation*}
\end{proof}

\begin{proof}[{Proof of Corollary \ref{cor: Simplify Theorem 1}}] Set $M_n = n^{\underline \gamma/2}$ and observe that $D_n \leq n^{\varepsilon-\underline\gamma/2}$ implies $n^\varepsilon / (\sqrt{M_n}D_n) \geq n^{\underline \gamma/4}$. This is sufficient for the condition (\ref{assumption for Gauss approx}) as $K_n$ is assumed to grow at a logarithmic rate. Notice that when $m = 1$, $\eta_n \leq O(n^{-1/4 + (\varepsilon-\underline\gamma/2)})$ and $\xi_n \leq O(n^{-1/4 + \varepsilon})$.

\end{proof}

\section{Proofs for Section \ref{sec:bootstrap}}

\newcommand{\wB}{\widehat{B}}
\newcommand{\ccone}{\check{\mathbf{1}}}

\subsection{Technical tool} 

We provide a proof of a technical Lemma that is used for showing consistency of our proposed wild regenerative block bootstrap technique. 

\begin{proof}[Proof of Lemma \ref{lem:cov-est}] The proof is divided into several steps. {In what follows, the $c,C$ and $C'$ are generic positive constants that depend on $\underline c$, $\beta$, $\| \tau_{\check \alpha}\|_{\check Q, \psi_1}$, $q$, $\theta'$, $\sup_{x \in S}\E_x[\tau_S^k]$ for $k = 2, 4$ and $\E[\tau_S^4]$}; its value may vary from place to place. Recall $n^* = {\beta}^{-1}n^\sharp$ and $\beta_n = n/n^*$, where ${\beta} = m\E_{\check{\alpha}}[\tau_{\check{\alpha}}]$ and $n^\sharp = \lfloor n/m - 1\rfloor$. So $\beta_n = \beta = \E_{\check{\alpha}}[\tau_{\check{\alpha}}]$ since $m = 1$ is assumed here. Fix $\delta_n' = n^{-1/2 + \gamma}$ for some $\gamma \in (0, 1/2)$. Recall $\upperpif := \sup_{f \in \calF}|\pi f|$ and $\mathbf 1(x) := 1$ for any $x$. So $\ccone(B) = \ell(B)$ is simply the length of the block $B \in \check E$. Recall from (\ref{eq:emp-covariance}) that 
\begin{equation*}
    \what{\Gamma}(\ccf, \ccg) = \frac{1}{n}\sum_{i = 1}^{\hat i_n} \Big\{ \brf(\what B_i) - \ell(\what B_i) \pi f \Big\} \cdot \Big\{ \brg(\what B_i) - \ell(\what B_i) \pi g \Big\}\quad \textrm{and} \quad \what\Gamma(\ccone, \ccone) = \frac{1}{n} \sum_{i = 1}^{\hat i_n} \ell^2(\what{B}_i).
\end{equation*}
Observe the decomposition
\begin{equation*}
    \what \Gamma(\ccf, \ccg) = \what \Gamma(\brf, \brg) - \pi f \what \Gamma(\ccone, \brg) - \pi g \what \Gamma(\ccone, \brf) + \pi f \pi g \what \Gamma(\ccone, \ccone)    
\end{equation*}
and that $\beta^{-1}\Cov(\ccf(B_1) , \ccg(B_1)) = \beta^{-1}\E[\ccf(B_1) \ccg(B_1)]$ as $\E[\ccf(B_1)] = 0$ for any $f$. Using the decomposition in the above display, we observe
\begin{equation*}
    \begin{aligned}
        \Delta(\check \calF^\circ) &\leq \sup_{f, g\in \calF} \Big| \what{\Gamma}(\brf, \brg) - \beta^{-1}\E[\brf(B_1) \brg(B_1)] \Big|\\
        &+ 2\sup_{f \in \calF} \Big| \what{\Gamma}(\brf, \ccone) - \beta^{-1}\E[\brf(B_1)\ell(B_1)] \Big| \cdot \upperpif\\
        &+ \Big| \what{\Gamma}(\ccone, \ccone) - \beta^{-1}\E[\ell^2(B_1)] \Big| \cdot \upperpif^2,
    \end{aligned}
\end{equation*}
which in turn yields
\begin{equation*}
    \begin{aligned}
        \bbP\big( \Delta(\check \calF^\circ) > \delta_n' \big) &\leq \bbP \Big( \sup_{f, g\in \calF} \Big| \what{\Gamma}(\brf, \brg) - \beta^{-1}\E[\brf(B_1) \brg(B_1)] \Big| > c\delta_n'\Big) \\
        &+ \bbP \Big( \sup_{f \in \calF} \Big| \what{\Gamma}(\brf, \ccone) - \beta^{-1}\E[\brf(B_1)\ell(B_1)] \Big|  > c\delta_n'/u(\calF, \pi) \Big)\\
        &+ \bbP \Big( \Big| \what{\Gamma}(\ccone, \ccone) - \beta^{-1}\E[\ell^2(B_1)] \Big| > c\delta_n'/\upperpif^2 \Big)\\
        &=: \mathbf U_1 + \mathbf U_2 + \mathbf U_3.
    \end{aligned}
\end{equation*}
The techniques for bounding the terms $\mathbf U_j$, $j = 1, 2, 3$ are identical; so we bound the first term $\mathbf U_1$ with thorough details, while referring to those details when bounding the remaining two terms.

\underline{Step 1} (Bounding $\mathbf U_1$). Define
\begin{equation}\label{eq:cov-est-decomp}
    \begin{aligned}
        \mathbf U_{1, 1}&= \bbP\Bigg( \sup_{\brf,\brg \in \check\calF}\bigg| \beta^{-1} \E\left[ \brf(B_1) \brg(B_1) \right] - \frac{1}{n}\sum_{i = 1}^{n^*} \brf(B_i) \brg(B_i) \bigg| > c\delta_n' \Bigg)    \\
        \mathbf U_{1, 2}& = \bbP\Bigg( \sup_{\brf,\brg \in\check \calF}\bigg| \frac{1}{n}\sum_{i = 1}^{n^*} \brf(B_i) \brg(B_i) - \frac{1}{n}\sum_{i = 1}^{i_n} \brf(B_i) \brg(B_i) \bigg| > c\delta_n' \Bigg)\\
        \mathbf U_{1, 3}&= \bbP\Bigg(\sup_{\brf, \brg \in \check\calF}\bigg| \frac{1}{n}\sum_{i = 1}^{i_n}\brf(B_i)\brg(B_i) - \frac{1}{n} \sum_{i = 1}^{\hat i_n}\brf(\widehat{B}_i)\brg(\widehat{B}_i) \bigg| > c\delta_n' \Bigg)
    \end{aligned}
\end{equation}
so that $\mathbf U_1 \leq \mathbf U_{1, 1} + \mathbf U_{1, 2} + \mathbf U_{1, 3}$.

First we bound $\mathbf U_{1, 1}$. Recall $n^* = \beta^{-1}n$ when $m = 1$. Set
\begin{equation}\label{eq:Zone}
    Z_{1} = \bigg\| \sum_{i = 1}^{n^*} \left\{ \brh(B_i) - \E\left[ \brh(B_i) \right] \right\} \bigg\|_{\check \calF^2}
\end{equation}
where $\check\calF^2$ is the class of pointwise product of functions. Note that for any discrete measure $\check R$ on $\check E$, we have
\begin{equation*}
N(\check \calF^2, \| \cdot \|_{\discmeas, 2}, 2\eta \| \check F^2 \|_{\discmeas, 2} )
\leq N(\check \calF, \| \cdot \|_{\discmeas, 2}, \eta \| \check F \|_{\discmeas, 2} )^2 \leq \Big( \frac{4\sqrt{\sfA_n}}{\eta} \Big)^{2 \sfV_n},   
\end{equation*}
and we can see that $\check \calF^2$ is VC-type with $(8 \sqrt{\sfA}_n, 2 \sfV_n)$. Then by maximal inequality for blocks~(Corollary \ref{cor: maximal inequality}), we have
\begin{equation*}
    \begin{aligned}
        \E\left[ Z_1 \right] \leq C\cdot\big(\sqrt{n}\overline{\sigma}_1 \sqrt{K_n} +  \|\max_{i \in [n^*]}\check{F}^2(B_i)\|_{\bbP, 2} K_n\big)
    \end{aligned}
\end{equation*}
where $\overline{\sigma}_1 = \sup_{\brh\in \check{\calF}^2} \| \brh \|_{\check{Q}, 2} \vee ( \| \check{F}^2 \|_{\check{Q},2}/\sqrt{n^*} ) $. 
Note that $\| \check{F}^2 \|_{\check Q, 2} = \| \check F \|_{\check Q, 4}^2 \leq C \| \check{F} \|^2_{\check Q, \psi_1} \leq  C D^2_n < \infty$, as
\[ \| \check F \|_{\check Q, \psi_1} \leq \sup_{x \in E} | F(x) | \cdot \| \tau_{\check \alpha} \|_{\bbPcheck, \psi_1} \leq CD_n. \]
Further, observe the bound
\begin{equation*}
    \begin{aligned}
        \|\max_{i \in [n^*]} \check F^2(B_i) \|_{\bbP, 2} = \|(\max_{i \in [n^*]}\check F(B_i))^2\|_{\bbP, 2} &= \| \max_{i \in [n^*]} \check F(B_i)\|_{\bbP, 4}^2\\
        &\leq C \| \max_{i \in [n^*] } \check F(B_i) \|_{\bbP, \psi_1}^2 \leq C' (\log n)^2 D_n^2,
    \end{aligned}
\end{equation*}
which finalizes the bound for $\E[Z_1]$, which is 
\begin{equation}\label{rom1-maximal-sG}
    \E[Z_1] \leq C \big(\sqrt{n} D_n^2  \sqrt{K_n} + (\log n)^2D_n^2 K_n\big) \leq C'\sqrt{n} D_n^2 K_n.
\end{equation}
Observe the inequality
\begin{equation*}
    \begin{aligned}
        \bbP\big( Z_1  > cn\delta_n' \big) &\leq \bbP\Big( Z_1 > 3\E[Z_1]/2 + \big\{ cn\delta_n' - 3\sqrt{n}D_n^2K_n/2  \big\} \Big).
    \end{aligned}
\end{equation*}
and note that the variance satisfies $\sup_{\brh \in \check \calF^2} \E[ \brh^2(B_i) ] \leq \| \check F \|^4_{\check Q, 4} \leq C  D_n^4 $. As $\| \sup_{\check h \in \check \calF^2} | \check h(B_i) | \|_{\check Q, \psi_{1/2}} \leq C \| \check F \|_{\check Q, \psi_1}\leq D_n < \infty$ for all $i \in [n]$, we invoke Theorem \ref{thm : talagrand} with $a = 1/2$~(so that $\psi_{1/2}^{-1}(n) = \log^2n$). The condition (\ref{m1-bootstrap-cond}) implies $n^\gamma / (D_n^2 K_n) \geq n^{\underline c_1}$ for some $\underline c_1 > 0$~(that depends up to $\underline c$) which implies
\begin{equation}\label{eq:Zone-upper}
    \begin{aligned}
        \bbP\big( Z_1  > cn\delta_n' \Big) &\leq \bbP\big( Z_1 > 3 \E[Z_1]/2 + C\big\{ n^{1/2 + \underline c_1}D_n^2K_n \big\} \Big)\\
        &\leq \exp\Big( -c' n^{2\underline c_1}K_n^2 \Big) + 3 \exp \bigg\{ - \bigg( \frac{c' n^{1/2 + \underline c_1}D_n^2 K_n}{(\log n )^2\| \check F^2 \|_{\check Q, \psi_{1/2}} } \bigg)^{1/2} \bigg\}\\
        &\leq \exp\Big( -c' n^{2\underline c_1}K_n^2 \Big) + 3 \exp \bigg( - \frac{c' n^{1/4 + \underline c_1/2}\sqrt{D_n K_n}}{\log n} \bigg).
    \end{aligned}
\end{equation}

Next, we bound $\mathbf U_{1, 2}$. We invoke Lemma \ref{lem: Montgomery-Smith} to deal with random indices. Note that
\begin{equation*}
    \sum_{i = 1}^{i_n}\brh(B_i) = \sum_{i = 1}^{n^*} \brh(B_i) +  \big( \mathbf 1 (i_n > n^*) - \mathbf 1(i_n < n^*) \big) \cdot \sum_{i = i_n \wedge n^*}^{i_n \vee n^*}\brh(B_i) 
\end{equation*}
for $\brh \in \check \calF^2$, so that
\begin{equation*}
     \bbP\left( \bigg\| \sum_{i = 1}^{n^*} \brh(B_i) - \sum_{i =1 }^{i_n} \brh(B_i)\bigg\|_{\check{\calF}^2}> cn\delta_n' \right) \leq \bbP\left( \bigg\| \sum_{i = i_n \wedge n^*}^{i_n \vee n^*} \brh(B_i) \bigg\|_{\check{\calF}^2} > c n \delta_n'\right). 
\end{equation*}
We can further observe the following set inclusion,
\begin{equation*}
    \begin{aligned}
        &\left\{ |i_n - n^*| \leq M_n \sqrt{n} \right\} \cap\bigg\{ \max_{n^* - M_n\sqrt{n} \leq k\leq n^* + M_n\sqrt{n}} \bigg\| \sum_{i = n^* - M\sqrt{n}}^{k} \brh(B_i) \bigg\|_{\check{\calF}^2} \leq cn\delta_n'/2 \bigg\}\\
        &\subset \bigg\{ \bigg\| \sum_{i = i_n \wedge n^*}^{i_n \vee n^*}\brh(B_i) \bigg\|_{\check{\calF}^2} \leq cn \delta_n'\bigg\} 
    \end{aligned}
\end{equation*}
which further yields
\begin{equation}
    \begin{aligned}
        \mathbf U_{1, 2} & \leq \bbP\Bigg( \max_{n^* - M_n\sqrt{n} \leq k\leq n^* + M_n\sqrt{n}} \bigg\| \sum_{i = n^* - M_n\sqrt{n}}^{k} \brh(B_i) \bigg\|_{\check{\calF}^2} > cn\delta_n'/2  \Bigg) + \bbP\left(|i_n - n^*| > M_n \sqrt{n} \right)\\
        & = \bbP\left( \max_{1 \leq k \leq 2M_n\sqrt{n}} \bigg\| \sum_{i = 1}^{k} \brh(B_i) \bigg\|_{\check{\calF}^2} > cn\delta_n' \right) + \bbP\left( |i_n - n^*| > M_n\sqrt{n} \right)\\
        &\leq C \bbP\left(  \Bigg\|  \sum_{i \in [2M_n\sqrt{n}]:\textrm{odd}}\brh(B_i) \Bigg\|_{\check{\calF}^2}  > cn\delta_n'\right) + \bbP\left( |i_n - n^*| > M_n\sqrt{n} \right) = \mathbf U_{1, 2}^{(1)} + \mathbf U_{1, 2}^{(2)}
    \end{aligned}
    \label{eq:U_2-decomp}
\end{equation}
where the last inequality is due to Lemma \ref{lem: Montgomery-Smith}~($B_i$ are i.i.d.)
We bound the terms $\mathbf U_{1, 2}^{(1)}, \mathbf U_{1, 2}^{(2)}$ of (\ref{eq:U_2-decomp}) in order. Set the supremum of an empirical process
\begin{equation}\label{eq:Ztwo}
    Z_2 =  \Bigg\|\ \sum_{i \in [2M_n \sqrt{n}]:\textrm{odd}} \left\{ \brh(B_i) - \E\left[ \brh(B_i) \right]\right\}  \Bigg\|_{\check{\calF}^2} .
\end{equation}
Then by subtracting and adding $\E[ \brh(B_i) ]$, then applying triangle inequality, we have
\begin{equation*}
    \begin{aligned}
        \mathbf U_{1, 2}^{(1)} &\leq C\bbP\bigg( Z_2 > cn\delta_n' - \sup_{\brh \in \check \calF^2} \bigg| \sum_{i \in [2M_n \sqrt{n}]:\textrm{odd}} \E[ \brh(B_i) ] \bigg| \bigg)\\
        &\leq C\bbP\bigg( Z_2 > cn\delta_n' - M_n \sqrt{n}\E[\check F^2(B_i)] \bigg).
    \end{aligned}
\end{equation*}
Note that $\E[\check F^2(B_i)] \leq C D_n^2$. Further notice the bound $\E[Z_2] \leq C\big(\sqrt{M_n}n^{1/4}D_n^2\sqrt{K_n} + \log^2_nD_n^2K_n\big)$ so that we have $\E[Z_2] \leq C \sqrt{M_n}n^{1/4}D_n^2 K_n$. Condition (\ref{m1-bootstrap-cond}) implies $n^\gamma/(M_n D_n^2 K_n) \geq n^{\underline c_2}$ for some $\underline c_2 >0$~(that depends up to $\underline c$) so that we have
\begin{equation*}
    \begin{aligned}
        \mathbf U_{1, 2}^{(1)} &\leq C\bbP\Big( Z_2 > 3\E[Z_2]/2 + \Big\{cn\delta_n' - 4 M_n \sqrt{n} D_n^2 - 3\E[Z_2]/2 \Big\} \Big)\\
        &\leq C\bbP \Big( Z_2 > 3 \E[Z_2]/2 + C'\Big\{ n\delta_n' - M_n \sqrt{n}D_n^2K_n \Big\} \Big)\\
        &\leq C\bbP \Big( Z_2 > (1 + \alpha) \E[Z_2] + C' n^{1/2 + \underline c_2}M_n D_n^2 K_n\Big).
    \end{aligned}
\end{equation*}
Apply Theorem \ref{thm : talagrand} with $a = 1/2$, so that
\begin{equation}\label{eq:Ztwo-upper}
    \begin{aligned}
        \mathbf U_{1, 2}^{(1)} \leq C\exp\Big( -cn^{2\underline c_2}M_n^2 K_n^2 \Big) + 3C \exp\bigg( -\frac{cn^{1/4 + \underline c_2/2 }\sqrt{D_nM_nK_n}}{\log n} \bigg).
    \end{aligned}
\end{equation}

To bound $\mathbf U_{1, 2}^{(2)} = \bbP\left( |i_n - n^*| > M_n \sqrt{n} \right)$, we invoke Lemma \ref{lem : block concentration} givens assumptions $\| \sigma_{\check \alpha} \|_{\bbP, \psi_1} < \infty$ and $\| \tau_{\check \alpha} \|_{\bbPcheck, \psi_1} < \infty$. Whenever $3/\sqrt{n} < M_n < \beta^{-1} \sqrt{n}$, we have $\mathbf U_{1, 2}^{(2)} \leq C \exp(- c M_n^2)$.
Then condition $n^\gamma / (M_n D_n^2 K_n) \geq n^{\underline{c}}$ for some $\underline c > 0$, along with the condition $M_n \geq n^{\underline \gamma/2}$ implies
\begin{equation*}
    \begin{aligned}
        \mathbf U_{1, 2}^{(2)} \leq C \exp (-c n^{\underline \gamma}).
    \end{aligned}
\end{equation*}

Lastly, we bound $\mathbf U_{1, 3}$ using Markov inequality. The moment $\E\big[ \sup_{\check h \in \check \calF^2} | n^{-1}\sum_{i =1 }^{i_n}\check h(B_i) - n^{-1}\sum_{i = 1}^{\hat i_n} \check h(\widehat{B}_i) | \big]$ can be bounded by first pulling out the sup-norm of $\check h \in \check \calF^2$ and then directly applying the techniques of Lemma 4.2 \cite{BertailClemencon2006Bernoulli}, which immediately results in the bound (see Remark \ref{rem:bertail})
\begin{equation}\label{eq:bertail-coupling}
    \begin{aligned}
        \E\Bigg[ \sup_{\check{h}\in \check{\calF}^2}\bigg| \frac{1}{n} \sum_{i = 1}^{i_n}\check{h}(B_i) - \frac{1}{n} \sum_{i = 1}^{\hat{i}_n}\check{h}(\widehat{B}_i)\bigg| \Bigg] \leq C  D_n^2 \alpha_n.
    \end{aligned}
\end{equation}
Using the tensorization property (\ref{eq:tensor}) of the split chain as well as the conditions (B3) and (B4) of Assumption \ref{assump : bootstrap}, Theorem 3.1 of \cite{BertailClemencon2006Bernoulli} first constructs a coupling between the split chain $\{(X_t, Y_t)\}_{t \in \mathbb N_0}$ and the approximate split chain $\{(X_t, \hat{Y}_t)\}_{t \in \mathbb N_0}$. Under this coupling, Lemma 4.2 of \cite{BertailClemencon2006Bernoulli} shows that proximity between the two chains are governed by the transition density estimation error $\alpha_n$, where they bound an analogue of the left-hand side moment of (\ref{eq:bertail-coupling}). We refer the reader to Theorem 3.1 and Lemma 4.2 of \cite{BertailClemencon2006Bernoulli} for a more detailed exposition. Lastly, Markov inequality yields $\mathbf U_{1, 3} \leq C n^{1/2 - \gamma} {D}_n^2 \alpha_n$.


\underline{Step 2} (Bounding $\mathbf U_2$ and $\mathbf U_3$). We show that each terms $\mathbf U_2$ and $\mathbf U_3$ attain a bound identical (in terms of rate) to that of $\mathbf U_1$. First decompose $\mathbf U_2$ as done in (\ref{eq:cov-est-decomp}); this yields a tail probability of empirical process terms analogous to $Z_1$ and $Z_2$ defined in (\ref{eq:Zone}) and (\ref{eq:Ztwo}) respectively, but with function class as $\check \calF \cdot \ccone := \{ \brf \cdot \ccone: f \in \calF \}$, a pointwise product between the classes $\check \calF$ and $\{\ccone\}$ where the latter is a singleton, and not $\check \calF^2$. 
Then for any discrete probability measure $\check R$ on $\check E$, observe
\begin{equation*}
    N(\check \calF \cdot \ccone, \|\cdot\|_{\check R, 2}, 2\eta \|\ccone \cdot \check F\|_{\check R, 2} ) \leq N(\check \calF, \|\cdot \|_{\check R, 2}, \eta \|\check F\|_{\check R, 2}) \leq \bigg( \frac{4\sqrt{\sfA_n}}{\eta} \bigg)^{\sfV_n}.
\end{equation*}
Further note
\begin{equation*}
    \| \check F \cdot \ccone \|_{\check Q, 2}^2 \leq \E[\ell^4(B_1)] \cdot D_n^2. 
\end{equation*}
Set $Z_1'$ to be identically defined as in (\ref{eq:Zone}) but with the supremum taken over the function class $\check \calF \cdot \ccone$, then we observe 
\begin{equation*}
    \bbP( Z_1' > n\delta_n/\upperpif ) \leq \bbP( Z_1' > (1 + \alpha) \E[Z_1'] + C\{ n\delta_n/\upperpif - \sqrt{n}D_nK_n \} ).
\end{equation*}
Then the assumption $\upperpif \leq D_n$ implies that $\bbP(Z_1' > cn\delta_n'/\upperpif)$ has an identical upper bound as $\bbP(Z_1 > cn\delta_n')$ shown in (\ref{eq:Zone-upper}). 

Next, define $Z_2$ as in (\ref{eq:Ztwo}) but with the function class $\check \calF \cdot \ccone$. Then $\bbP( Z_2' > n\delta_n'/\upperpif - C M_n \sqrt{n} D_n)$ attains the same upper bound as shown in (\ref{eq:Ztwo-upper}). Lastly, invoking \cite{BertailClemencon2006Bernoulli}, we may observe
\begin{equation*}
    \begin{aligned}
        \E\Bigg[ \sup_{f \in \calF}\bigg| \frac{1}{n} \sum_{i = 1}^{i_n}\check{f}(B_i)\ccone(B_i) - \frac{1}{n} \sum_{i = 1}^{\hat{i}_n}\check{f}(\widehat{B}_i)\ccone({\what{B}_i})\bigg| \Bigg] \leq C  D_n \alpha_n
    \end{aligned}
\end{equation*}
and Markov's inequality proves that $\mathbf U_2$ attains the same upper bound as $\mathbf U_1$. One may show the same bound for $\mathbf U_3$ with the same line of reasoning, which we omit.


\end{proof}

\subsection{Bootstrap consistency}\label{subsec:boot-consistency}

Let $\underline c, \underline \gamma, \gamma > 0$ be the positive constants stated in Theorem \ref{thm: bootstrap consistency}. We introduce an event on which bootstrap consistency holds. Define the following events
\begin{equation}\label{eq:events}
\begin{aligned}
    \calE_0(\calF) &= \big\{ \| \bbG_n \|_{\calF} \leq n^\gamma/D_n \big\}, \quad
    \calE_1(\calF) = \big\{ \Delta(\check \calF^\circ) \leq n^{-1/2 + \gamma} \big\},\\
    \calE_2(\calF) &= \bigg\{ \bigg| n^{-1}\sum_{i = 1}^{\hat i_n}\big( \check F^\circ(\what B_i) \big)^2 - \beta^{-1}\E\big[ (\check F^\circ(B_1))^2 \big] \bigg| \leq n^{-1/2 + \gamma} \bigg\},\\
    \calE_3(\calF) &= \bigg\{ \bigg|n^{-1}\sum_{i = 1}^{\hat i_n}\ell^2(\what B_i) - \beta^{-1}\E[\ell^2(B_1)] \bigg| \leq n^{-1/2 + \gamma} \bigg\}.
\end{aligned}
\end{equation}
We omit the dependence of the events on $\gamma\in (0, 1/2)$, and define the intersected set
\begin{equation}\label{eq:event-sec3}
    \calE = \cap_{j = 0}^3 \calE_j(\calF).
\end{equation}
For the constant $c > 0$ present in Lemma \ref{lem:cov-est}, define the shorthand
\begin{equation*}
    \overline{\Delta}_n := n^{1/2 - \gamma} D_n^2 \alpha_n + \exp(-c n^{\underline{\gamma} \wedge \underline c}).
\end{equation*}
Recall the pivotal statistics $\bbG_n^\zeta(f)$ in (\ref{eq:multplier-boot}) is a function of the data $X_0^{n-1}$ and the Gaussian multipliers $\zeta_1^{\hat i_n}$. 



\begin{proof}[{Proof of Theorem \ref{thm: bootstrap consistency}}]
The proof is divided into several steps. In what follows $c, c', C, C'$ are generic positive constants that depend on $\|\sigma_{\check \alpha}\|_{\psi_1}, \|\tau_{\check \alpha}\|_{\check Q, \psi_1}, \underline{\sigma}$,$\beta$ and $\E[\ell^2(B_1)]$.

\underline{Step 1} (High-probability event). We first show that the conditions specified in Theorem \ref{thm: bootstrap consistency} imply the conditions of Theorem \ref{thm:Gaussian-approximation}. As $\overline \gamma > \gamma$ and as $K_n \leq n^{\underline c/2}$, we observe
\begin{equation*}
    \frac{n^{\overline \gamma/2}}{\sqrt{M_n} D_n K_n} \geq n^{\underline c/4}.
\end{equation*}
Let's set the sequence $\delta_n$ in Theorem \ref{thm:Gaussian-approximation} as $\delta_n = n^{-1/4 + \overline \gamma/2}$. Assumption \ref{assump : bootstrap} is sufficient for Assumption \ref{assump : Gauss approx} (see the dicussion that follows Assumption \ref{assump : bootstrap}), hence the conditions of Theorem \ref{thm:Gaussian-approximation} are met. So when $m = 1$, we observe the Gaussian approximation
\begin{equation}\label{eq:gauss-approx-bootstrap}
    \begin{aligned}
        \rho_n \leq C\cdot \tilde \xi_n \quad \text{where} \quad \tilde\xi_n := \bigg\{\frac{\sqrt{K_n}}{n^{1/4 - \overline \gamma/2}} + \frac{D_nK_n^{5/4}}{n^{1/4 + \overline \gamma/2}} + \exp \big(-c n^{\underline \gamma \vee {(\underline c/8)}}\big) \bigg\}.
    \end{aligned}
\end{equation}

So (\ref{eq:gauss-approx-bootstrap}) and the definition of $\rho_n$ implies
\begin{equation*}
    \begin{aligned}
        \bbP(\calE_0^c(\calF)) &= \bbP\big( \| \bbG_n \|_\calF > n^\gamma/D_n \big) \leq \bbP\big( \| G \|_{\calF} > n^\gamma/D_n \big) + C\cdot\tilde \xi_n.
    \end{aligned}
\end{equation*}
We now bound $\bbP\big( \| G \|_{\calF} > n^\gamma/D_n \big)$. We observe $\E[G^2(f)] = \beta^{-1}\E[ (\brf(B_1) - \ell(B_1) \pi f)^2 ] \leq 2\beta^{-1} \E[\brf^2(B_1)] + 2\beta^{-1}\E[\ell^2(B_1)](\pi f)^2 \leq C' D_n^2$ by inquiring (\ref{eq:block-mean-pi}). Recall that we set $(\sfA_n, \sfV_n)$ as the parameters for the VC-type function class $\calF^\circ$. Invoke inequality (\ref{eq:d-continuous}) and apply Lemma \ref{lem: covering number} to attain
\begin{equation*}
    \E[\|G\|_\calF] \leq C\int_0^{C'D_n} \sqrt{\sfV_n \log \big( \sfA_n/\eta \big)}d\eta \leq C D_n \sqrt{\sfV_n \log \big( \sfA_n/D_n \big)}
\end{equation*}
via Dudley's entropy integral bound~\cite{Dudley2014UCLT}. Condition (\ref{m1-bootstrap-cond}) implies $n^{\gamma}/D_n^2 \geq n^{\underline c}$, so that Borel-Sudakov-Tsirel'son inequality~\cite{ledoux2001concentration} yields
\begin{equation*}
    \bbP \big( \| G \|_{\calF} > n^\gamma/D_n \big) \leq \exp \big( - c n^{2\underline c} \big).
\end{equation*}

Next we show $\bbP(\calE_j^c(\calF)) \leq O(\overline{\Delta}_n)$ for $j = 1, 2, 3$. For a simpler presentation, we express some guarantees by asymptotics, where the hidden constants in the big-O notation coincide to the parameters appearing in Lemma \ref{lem:cov-est}: $\beta, \underline{c}, \vartheta_B, \|\sigma_{\check \alpha}\|_{\bbP, \psi_1}, \|\tau_{\check \alpha}\|_{\bbPcheck, \psi_1}$. As we are assuming the same conditions as in Lemma \ref{lem:cov-est}, we automatically have $\bbP(\calE_1^c(\calF)) \leq O(\overline \Delta_n)$. Next, notice that $\bbP(\calE_3^c(\calF)) \leq \mathbf U_3$ in the proof of Lemma \ref{lem:cov-est}, hence $\bbP(\calE_3^c(\calF)) \leq O(\overline \Delta_n)$. Lastly, to bound $\bbP(\calE_2^c(\calF))$, it is sufficient to check~(again see proof of Lemma \ref{lem:cov-est} on bounding the term $\mathbf U_1$) that the size and the envelope of the singleton $\{\check F(\cdot) + \ell(\cdot) \pi F\}$ matches or is dominated by that of $\calF$ respectively, which can be trivially verified.

The following steps go through the derivation of a non-asymptotic upper bound of $\rho_n^*$ on the event $\calE$ in (\ref{eq:event-sec3}). 

\underline{Step 2} (First decomposition). Set the sequence $\delta_n = n^{-1/4 + \overline \gamma/2}$. We rely on the discretization of the function class $\check \calF^\circ$ as done in (\ref{eq:disc-class}) which needs $\ccf$ to be apparent in the bootstrap statistics $\bbG_n^\zeta(f)$. For a simpler presentation, set the shorthand
\begin{equation}\label{eq:non-cent-boot}
    \tilde \bbG_n^\zeta(\brf) := n^{-1/2} \sum_{i = 1}^{\hat i_n} \zeta_i \brf(\what B_i) \quad \text{for any function $f$ on $E$.}
\end{equation}
Then observe 
\begin{equation*}
    \begin{aligned}
        \bbG_n^\zeta(f) - \tilde \bbG_n^\zeta(\ccf) = -\frac{1}{\sqrt{n}}\sum_{i = 1}^{\hat i_n}\zeta_i \ell(\what B_i) \cdot \big( \what \pi f - \pi f\big) = -\tilde \bbG_n^\zeta(\ccone) \cdot \bbG_n(f) /\sqrt{n}
    \end{aligned}
\end{equation*}
as $\bbG_n(f) = n^{-1/2}\sum_{i = 0}^{n-1}\{f(X_i) - \pi f\}$. Further observe that $\sup_{f \in \calF}|\pi f - \what{\pi f}| = n^{-1/2} \| \bbG_n \|_{\calF}$ holds, which yields the first decomposition
\begin{equation}\label{eq:boot-first-decomp}
\begin{aligned}
    \bbP\big( \bbG_{n, \calF}^{\zeta, \vee} \leq x | X_0^{n-1} \big) \leq& \bbP\big( \tilde \bbG_{n, \calF^\circ}^{\zeta, \vee} \leq x + \delta_n | X_0^{n-1}  \big) + \bbP\big( \| \bbG_n \|_{\calF} \cdot | \tilde \bbG_n^\zeta(\ccone) | > \sqrt{n}\delta_n | X_0^{n-1} \big).
\end{aligned}
\end{equation}
Here we first bound the second term on the right-hand side of (\ref{eq:boot-first-decomp}); we apply a second decomposition on the first term in the next step. 

Recall $\delta_n = n^{-1/4 + \overline \gamma/2}$, and the second term of the right-hand side of (\ref{eq:boot-first-decomp}) is bounded as
\begin{equation*}
    \begin{aligned}
        &\bbP\big( \| \bbG_n \|_\calF\cdot| \tilde \bbG_n^\zeta(\ccone) |   > \sqrt{n}\delta_n | X_0^{n-1} \big) \\
        &\leq \bbP\big( \| \bbG_n \|_\calF > n^{\overline \gamma} |X_0^{n-1} \big) + \bbP\big( | \tilde \bbG_n^\zeta(\check{\mathbf{1}}) | > n^{1/4 - \overline \gamma/2} | X_0^{n-1} \big).
    \end{aligned}
\end{equation*}
We observe $\bbP\big( \| \bbG_n \|_\calF > n^{\overline \gamma} |X_0^{n-1} \big) = 0$ on the event $\calE_0$. Further, notice that $\tilde \bbG_n^\zeta(\check{\mathbf{1}})$ is a centered Gaussian with variance $n^{-1}\sum_{i = 1}^{\hat i_n}\ell^2(\what B_i)$. On the event $\calE_3(\calF) = \big\{ \big|n^{-1}\sum_{i = 1}^{\hat i_n}\ell^2(\what B_i) - \beta^{-1}\E[\ell^2(B_1)] \big| \leq n^{-1/2 + \gamma} \big\}$, we see that the variance is bounded by a constant $C$, for large enough $n$. Recall $\overline \gamma \in (0, 1/2)$, thereby implying $1/4 - \overline\gamma/2 \in (0, 1/4)$. So the simple Gaussian tail inequality yields
\begin{equation*}
    \bbP\big( | \tilde \bbG_n^\zeta(\check{\mathbf{1}}) | > n^{1/4 - \overline \gamma/2} | X_0^{n-1} \big) \leq 2 \exp\big\{ -cn^{1/2 - \overline \gamma} \big\}
\end{equation*}
whenever $X_0^{n-1} \in \calE_3(\calF)$,

\underline{Step 3} (Second decomposition). Next we bound the first term on the right-hand side of (\ref{eq:boot-first-decomp}). We first discretize the process $\{\tilde \bbG_n^\zeta(\ccf)\}_{\ccf \in \calF^\circ}$ in (\ref{eq:boot-first-decomp}) via the $\big( \varepsilon \| \check F^\circ \|_{\check Q, 2} \big)$-net of $\check \calF^\circ$, $\check \calF_N^\circ = \{ \ccf_1, ..., \ccf_N \}$ where $N = N(\check \calF^\circ, \| \cdot\|_{\check Q, 2}, \varepsilon \| \check F^\circ\|_{\check Q, 2})$. 


Recall the Gaussian process $\{W(\ccf)\}_{\ccf \in \check \calF^\circ}$ that has the covariance function (\ref{eq:W-cov}). Applying item (i) of Lemma \ref{lem:Gaussian-comparison} and Lemma \ref{lem: Nazarov} successively, observe that for any $X_0^{n-1}$, the first term of the right-hand side of (\ref{eq:boot-first-decomp}) can further be bounded by 
\begin{equation*}
    \begin{aligned}
        &\bbP\big( \tilde \bbG_{n, \check \calF^\circ}^{\zeta, \vee} \leq x + \delta_n | X_0^{n-1}  \big)\\
        &\leq \bbP\big( \tilde \bbG_{n, \check \calF_N^\circ}^{\zeta, \vee} \leq x + 2\delta_n | X_0^{n-1}  \big) + \bbP \big( \| \tilde \bbG_n^\zeta \|_{\check \calF^\circ_\varepsilon} > \delta_n | X_0^{n-1} \big)\\
        &\leq \bbP \big( \beta^{-1/2} W^{\vee}_{\check \calF_N^\circ} \leq x + 2\delta_n \big) +  C\cdot \Delta(\check \calF_N^\circ)^{1/2}\log N + \bbP \big( \| \tilde \bbG_n^\zeta \|_{\check \calF^\circ_\varepsilon} > \delta_n | X_0^{n-1} \big)\\
        &\leq \bbP \big( \beta^{-1/2} W^\vee_{\check \calF_N^\circ} \leq x - \delta_n \big) + C \cdot \big\{ \delta_n \sqrt{\log N} + \Delta(\check \calF_N^\circ)^{1/2}\log N\big\} + \bbP \big( \| \tilde \bbG_n^\zeta \|_{\check \calF^\circ_\varepsilon} > \delta_n | X_0^{n-1} \big)\\
        &\leq \bbP \big( \beta^{-1/2}  W^\vee_{\check \calF^\circ} \leq x \big) + C \cdot \big\{ \delta_n \sqrt{\log N} + \Delta(\check \calF_N^\circ)^{1/2}\log N\big\}\\
        & + \bbP \big( \beta^{-1/2} \| W \|_{\check \calF^\circ_\varepsilon} > \delta_n \big) + \bbP \big( \| \tilde \bbG_n^\zeta \|_{\check \calF^\circ_\varepsilon} > \delta_n | X_0^{n-1} \big).
    \end{aligned}
\end{equation*}
Recall the definition of the discretization error (\ref{eq:gauss-disc-err}); we reuse the result (\ref{eq:gaussian-Markov}) to have
\begin{equation*}
    \bbP \big( \beta^{-1/2} \| W \|_{\check \calF^\circ_\varepsilon} > \delta_n \big) \leq C \delta_n^{-1} D_n n^{-1/2} K_n = C n^{-1/4 - \overline \gamma/2} D_n K_n.
\end{equation*}
On the event $\calE_1(\calF) = \{ \Delta(\check \calF^\circ) \leq n^{-1/2 + \gamma} \}$, we see $\Delta(\check \calF_N^\circ)^{1/2} \leq \Delta(\check \calF^\circ)^{1/2} \leq n^{-1/4 + \gamma/2}$. As $\delta_n = n^{-1/4 + \overline \gamma/2}$, we see
\begin{equation*}
    \delta_n \sqrt{\log N} + \Delta(\check \calF^\circ_N)^{1/2} \log N \leq \frac{2\log N}{n^{1/4 - \overline \gamma/2}}.
\end{equation*}
In what follows, we bound the remaining term $\bbP \big( \| \tilde \bbG_n^\zeta \|_{\check \calF^\circ_\varepsilon} > \delta_n | X_0^{n-1} \big)$.

\underline{Step 4} (Bounding discretization error). We use Borel-Sudakov-Tsirel'son inequality \cite{ledoux2001concentration} to bound $\bbP \big( \| \tilde \bbG_n^\zeta \|_{\check \calF^\circ_\varepsilon} > \delta_n | X_0^{n-1} \big)$ on the event $\calE$. To do so, we first bound the variance of $\tilde \bbG_n^\zeta(\check h^\circ)$ for $\check h ^\circ \in \check \calF_\varepsilon^\circ$. Recall $\check \calF_\varepsilon^\circ = \{ \ccf - \ccg: \|\ccf - \ccg\|_{\check Q, 2} \leq 2 \varepsilon\|\check F^\circ\|_{\check Q, 2} \}$, so by adding and subtracting $\beta^{-1}\E\big[\{\check h^\circ (B_1) \}^2\big]$, we observe that on $X_0^{n-1} \in \calE_1(\calF)$,
\begin{equation*}
    \E\big[ \{ \tilde \bbG_n^\zeta(\check h^\circ) \}^2 |X_0^{n-1}\big] \leq 4 \Delta(\check \calF^\circ) + CD_n/n \leq C \big\{ n^{-1/2 + \gamma} + D_n^2/n \big\}\leq C' \cdot n^{-1/2 + \gamma} ;
\end{equation*}
note that (\ref{m1-bootstrap-cond}) naturally implies that the order of $D_n^2/n$ is dominated by $O(n^{-1/2 + \gamma})$. 

Next we bound the mean $\E\big[ \| \tilde \bbG_n^\zeta\|_{\check \calF_\varepsilon^\circ} \big]$. We start by bounding the covering number of the class $\check \calF^\circ_\varepsilon$ with respect to the following metric: for any $\check h_1^\circ, \check h_2^\circ \in \check \calF_\varepsilon^\circ$ and $X_0^{n-1}$, define a metric $\rho_2$ as
\begin{equation*}
    \big(\rho_2(\check h_1^\circ, \check h_2^\circ)\big)^2 := \E\big[ \{ \tilde \bbG_n^\zeta(\check h_1^\circ) - \tilde \bbG_n^\zeta(\check h_2^\circ) \}^2 |X_0^{n-1}\big] = \frac{1}{n}\sum_{i = 1}^{\hat i_n} \big\{ \check h_1^\circ(\what B_i) - \check h_2^\circ(\what B_i) \big\}^2.
\end{equation*}
Let $\check Q_n$ be an empirical distribution $\hat i_n^{-1}\sum_{i = 1}^{\hat i_n}\delta_{\what B_i}$ on $\check E$, so that we observe 
\begin{equation*}
    \big(\rho_2(\check h_1^\circ, \check h_2^\circ)\big)^2 \leq \| \check h_1^\circ - \check h_2^\circ \|_{\check Q_n, 2}^2 \quad \text{almost surely.}
\end{equation*}
This implies $N(\check \calF_\varepsilon^\circ, \rho_2, \eta) \leq N(\check \calF_\varepsilon^\circ, \|\cdot \|_{\check Q_n, 2}, \eta)$. Next invoke the inclusion $\check \calF_\varepsilon^\circ \subseteq \check \calF_{-}^\circ$ so that we have
\begin{equation*}
    N(\check \calF_\varepsilon^\circ, \|\cdot \|_{\check Q_n, 2}, \eta) \leq N^2(\check \calF^\circ, \|\cdot\|_{\check Q_n, 2}, \eta'\|\check F^\circ\|_{\check Q_n, 2})
\end{equation*}
where $\eta' = \eta/\|2\sqrt{2}\check F^\circ\|_{\check Q_n, 2}$. Collecting the inequalities on the covering numbers, we have $\log N(\check \calF_\varepsilon^\circ, \rho_2, \eta) \leq C \sfV_n \log (\| \check F^\circ\|_{\check Q_n, 2} \sfA_n/\eta) $. Dudley's entropy integral~\cite{Dudley2014UCLT} yields
\begin{equation*}
    \E\big[ \| \tilde \bbG_n^\zeta \|_{\check \calF_\varepsilon^\circ} |X_0^{n-1}\big] \leq C \int_0^{C' n^{-1/4 + \gamma/2}} \sqrt{ \sfV_n \log (\|\check F^\circ\|_{\check Q_n, 2}\sfA_n/\eta) } d\eta,
\end{equation*}
and we further invoke the event $\calE_2(\calF)$ to bound the right-hand side of the inequality. Notice
\begin{equation*}
    \frac{\hat{i}_n}{n} \| \check F^\circ \|_{\check Q_n, 2}^2 = \frac{1}{n} \sum_{i = 1}^{\hat{i}_n}\big(\check F^\circ(\widehat B_i)\big)^2 \quad \text{where}\quad \hat i_n \geq 1,
\end{equation*}
and on the event $X_0^{n-1} \in \calE_2(\calF) = \big\{ n^{-1}\sum_{i = 1}^{\hat i_n}\big( \check F^\circ(\what B_i) \big)^2 - \beta^{-1}\E\big[ (\check F^\circ(B_1))^2 \big] \big| \leq n^{-1/2 + \gamma} \big\}$ we have $\|\check F^\circ\|_{\check Q_n, 2} \leq C (nD_n^2 + n^{1/2 + \gamma})^{1/2}$. For the sequence $a_n = (nD_n^2 + n^{1/2 + \gamma})^{1/2}\cdot n^{-1/4 + \gamma/2}$ (defined for concise presentation), we conclude
\begin{equation*}
    \E\big[ \| \tilde \bbG_n^\zeta \|_{\check \calF_\varepsilon^\circ} |X_0^{n-1}\big] \leq C n^{-1/4 + \gamma/2}  \sqrt{ \sfV_n \log (a_n\sfA_n) } \quad \text{on $\calE_1(\calF) \cap \calE_2(\calF)$.}
\end{equation*}
Note that $a_n$ and $\sfA_n$ are at most polynomially growing, $\sfV_n$ is at most logarithmically growing.

Recall that we set $\delta_n = n^{-1/4 + \overline \gamma/2}$ where $\overline \gamma > \gamma$. We apply Borel-Sudakov-Tsirel'son inequality~\cite{ledoux2001concentration} and observe
\begin{equation*}
    \begin{aligned}
        \bbP\big( \|\tilde \bbG_{n}^{\zeta}\|_{\check \calF^\circ_\varepsilon} > \delta_n| X_0^{n-1} \big) &= \bbP\big( \|\tilde\bbG_{n}^{\zeta}\|_{\check \calF^\circ_\varepsilon}> \E\big[ \|\tilde\bbG_{n}^{\zeta}\|_{\check \calF^\circ_\varepsilon} |X_0^{n-1}\big] + \big\{ \delta_n - \E\big[ \|\tilde\bbG_{n}^{\zeta}\|_{\check \calF^\circ_\varepsilon} |X_0^{n-1}\big] \big\} |X_0^{n-1} \big)\\
        &\leq C\exp\left\{ - c \frac{\delta_n^2}{n^{-1/2 + \gamma}}\right\} \leq C\exp \{ -c' n^{\overline \gamma - \gamma} \}.
    \end{aligned}
\end{equation*}

\end{proof}

\subsection{Drift conditions}

Here we show that the geometric drift condition is sufficient for satisfying Assumptions \ref{assump : Gauss approx} and \ref{assump : bootstrap}. 

\begin{proof}[Proof of Proposition \ref{prop:geom-drift}]

    First consider a generic random variable $X$ with distribution $\bbP_X$ and recall the definition of the $\psi$-Orlicz norm $\| X \|_{\bbP_X, \psi} = \inf_c \{ c > 0: \E[\psi(|X|/c)] \leq 1 \}$ (see \cite{ledoux2001concentration}). For some random variable $Y$ and the corresponding conditional distribution of $X|Y$ (say $\bbP_{X|Y}$), we have $\|X\|_{\bbP_{X|Y}, \psi} = \inf_c \{c>0:\E[\psi(|X|/c)|Y] \leq 1\}$. Assume that $\sup_{Y} \|X\|_{\bbP_{X|Y}, \psi} = C < \infty$. Then the following inequality
    \begin{equation*}
        \E[\psi(|X|/C)] = \E\big[\E[\psi(|X|/C)|Y]] \leq \E\big[ \E[\psi(|X|/\|X\|_{\bbP_{X|Y}, \psi})|Y] \big] \leq 1,
    \end{equation*}
    yields the inequality $\|X\|_{\bbP_X, \psi} \leq \sup_Y \|X\|_{\bbP_{X|Y}, \psi}$.

    Equipped with this general inequality of norms, we follow the discussion of \cite{adamczak}; note that our setting is a special case of their's, where $m = 1$ (i.e., independent blocks) with constant additive functionals. Following Proposition 1 of \cite{adamczak}, we observe 
    \begin{equation*}
        \| \sigma_{\check \alpha} + 1 \|_{\bbP_x, \psi_1} \leq C \cdot \big\{ \sup_{x \in S} \| \tau_S \|_{\bbP_x, \psi_1} + 1 \big\}.
    \end{equation*}
    where $C$ is a constant that depends on the minorization parameter $\theta$. So we observe $\|\sigma_{\check \alpha}\|_{\bbP, \psi_1} \leq \sup_{x}\|\sigma_{\check \alpha} + 1\|_{\bbP_x, \psi_1} \leq C \cdot \big\{ \sup_{x \in S} \|\tau_S\|_{\bbP_x, \psi_1} + 1\big\}$. So whenever the geometric drift condition (C2) is satisfied, Lemma \ref{lem:geom-drift} implies that $\|\sigma_{\check \alpha}\|_{\bbP, \psi_1} < \infty$.

    Further note the inequality $\| \tau_{\check \alpha }\|_{\bbPcheck, \psi_1} \leq r_\theta\big[\sup_{x \in S} \| \tau_S\|_{\bbP_x, \psi_1} + 1\big]$ from Proposition 2.2 of \cite{adamczak}, where $r_{\theta}$ depends only on the minorization parameter $\theta$. So again, whenever geometric drift condition (C2) is satisfied, Lemma \ref{lem:geom-drift} implies 
    \begin{equation*}
        \| \ell (B_1 ) \|_{\check Q, \psi_1}  = \| \tau_{\check \alpha }\|_{\bbPcheck, \psi_1} \leq r_\theta\Big[\sup_{x \in S} \| \tau_S\|_{\bbP_x, \psi_1} + 1\Big] \leq r_\theta\Big[\max \{ 1, ( b + \overline V )/ \log 2 \}\cdot c + 1 \Big].
    \end{equation*}
    Overall, the geometric drift condition {(C2)} implies that the moments $\|\sigma_{\check \alpha}\|_{\bbP, \psi_1}$, $\|\tau_{\check \alpha}\|_{\bbPcheck, \psi_1}$, $\sup_{x}\E_x[\tau_S^4]$ are all finite, and by marginalizing over the initial distribution $\lambda$, we have $\E[\tau_S^4]<\infty$. 
\end{proof}

\section{Proofs for  Section \ref{sec:application}}

\subsection{Checking the assumptions for the diffusion process}


The following Proposition states the uniform (over $x\in \calX$) order of the variance of $\check \kernel_{x, h}(B_1)$; recall the definition $\sigma^2(x) := \Var(\check \kernel_{x, h}(B_1))$. Note that the bandwidth $h$ depends on $n$, hence the variance $\sigma^2(x)$ depends on $n$.

\begin{proposition}\label{prop:block-var}
Let $\{X_t\}_{t \in \mathbb N_0}$ be the low-frequency observations from (\ref{eq:diffusion}) with $(\varrho, b)\in \Theta$ and for some constant $a>0$, $h = n^{-a}$ be the bandwidth for the kernel $\kernel_{x, h}$ defined in (\ref{eq:scale-center-kernel}). Let $n_0, c, C$ be positive constants that depend on $\calX, a, \pi_u, \kernel, \theta, \beta$. Then for $n \geq n_0$, 
\begin{equation}\label{eq:block-lower-upper}
    c h^{-1}  \leq \sigma^2(x) \leq C h^{-1} \quad \text{for all $x \in \calX$}.
\end{equation}
\end{proposition}

\begin{proof}[Proof of Proposition \ref{prop:block-var}] 
The proof proceeds in two steps and $n_0, c, C$ are generic constants that depend on $\calX, a, \pi_u, \kernel, \theta, \beta$, which may vary from place to place.

\underline{Step 1} (Lower bound). By definition, the length of blocks are lower bounded by $\ell(B_1) \geq 1$, and the kernels are non-negative, i.e. $\kernel_{x, h} \geq 0$. Also the first component of each blocks $B_i$ follow the regeneration distribution $\nu$. So we can lower bound the second moment of $\check \kernel_{x, h}(B_i)$ by the following
    \begin{equation*}
        \begin{aligned}
            \E\big[ \check \kernel_{x, h}^2(B_1) \big] &= \E\Bigg[ \bigg\{ \sum_{t = 1}^{\ell(B_1)}\kernel_{x, h}(X_t) \bigg\}^2 \Bigg] \geq \int \kernel_{x, h}^2(y)d \nu(y)
        \end{aligned}
    \end{equation*}
The density $\nu'$ of $\nu$ is uniform over the interval $[0, 1]$ (see discussion preceding (\ref{eq:minorization-diff})), hence uniform over the strictly smaller compact subset $\mathcal X \subsetneq [0, 1]$. So fixing any small enough bandwidth $h$~(i.e. for large sample size $n$ as bandwidth is set to decrease with $n$), we have $\nu'(hz + x) = 1$ for all $z \in [0, 1]$ and for $x \in \calX$. The size of $\calX$ and the rate at which $h = o(1), n\to \infty$ determines how large we should take $n$ so as to well define $\nu'(hz + x)$. Say there exists a positive constant $n_0$ depending on $h$ and $\calX$ such that $\nu'(hz + x) = 1$ for any $n\geq n_0$ and $x, z$. Then we observe
\begin{equation*}
    \begin{aligned}
        \int \kernel_{x, h}^2(y)d \nu(y)& = h^{-1}\int_0^1 h^{-1}\kernel^2( (y - x)/h ) d\nu(y)\\
        &= h^{-1}\int_0^1 \kernel^2(z) \nu'(h z + x) dz \geq h^{-1} \int_0^1 \kernel^2(z) dz.
    \end{aligned}
\end{equation*}
Recall $\E[\check \kernel_{x, h}(B_1)] = m \E_{\check \alpha}[\tau_{\check \alpha}]\cdot \pi \kernel_{x, h}$ from (\ref{eq:block-mean-pi}). We can observe
\begin{equation*}
    \begin{aligned}
        \pi \kernel_{x, h} = \int \kernel(z) \pi( h z + x)dz \leq \pi_u \int \kernel(z) dz = \pi_u,
    \end{aligned}
\end{equation*}
as $w$ integrates to $1$. So we have
\begin{equation*}
    \sigma^2(x) = \Var(\check \kernel_{x, h}(B_1)) \geq h^{-1}\int_0^1 \kernel^2(z) dz - \big( \beta \pi_u \big)^2.
\end{equation*}
Reusing $n_0$, assume that the constant $n_0$ further depends on $\beta, \pi_u, \kernel$ such that the $h^{-1}\int \kernel^2(z)dz - \big(\beta\pi_u\big)^2 > h^{-1}$ whenever $n\geq n_0$ (which should exist as $h = o(1), n\to\infty$). Then $\sigma^2(x) \geq c h^{-1}$ for $n \geq n_0$. 

\underline{Step 2} (Upper bound). From \cite{nickl2017nonparametric}, the reflected diffusion process (\ref{eq:diffusion}) is shown to satisfy the minorization condition, drift condition, and the strong aperiodicity condition. Further employing Corollary 6.1 in \cite{baxendale2005renewal}, there exists a constant $\rho\in (0, 1)$ depending on the minorization parameter $\theta$, such that
\begin{equation}\label{eq:rho-mixing}
    \bigg\Vert P^n f - \int f d\pi \bigg\Vert_{\bbP_\pi, 2} \leq \rho^n \bigg\Vert f - \int f d\pi \bigg\Vert_{\bbP_\pi, 2} \quad \text{ for $f \in L^2(\pi)$.}
\end{equation}
where $\| \cdot \|_{\bbP_\pi, 2}$ implies that integration occurs over the split chain initiated with the stationary distribution $\pi$, not $\lambda$. 

Fix $f = \kernel_{x,h_0}$ for any $x \in \calX$ and $h_0 > 0$. Apply Cauchy-Schwarz inequality\footnote{$\Var_\pi, \Cov_\pi$ refers to the variance and covariance with respect to the probability $\bbP_\pi$.} and the mixing condition (\ref{eq:rho-mixing}) to observe
\begin{equation*}
    \begin{aligned}
        \Var_{\pi} \Big( \sum_{t = 0}^{n-1}\kernel_{x, h_0}(X_t) \Big) &= \sum_{t =0}^{n-1}\Var_\pi(\kernel_{x, h_0}(X_t)) + 2 \sum_{j < k} \Cov_\pi ( \kernel_{x, h_0}(X_j) \kernel_{x, h_0}(X_k) )\\
        &\leq n \bigg( 1 + 2 \sum_{\ell = 1}^{n-1}\rho^\ell \bigg) \Var_\pi (\kernel_{x, h_0}(X_0)).
    \end{aligned}
\end{equation*}
Take the limit $n\to\infty$ and invoke the fact that the asymptotic variance of $n^{-1/2}(\kernel_{x, h_0}(X_0) + ... + \kernel_{x, h_0}(X_{n-1}))$ is $\beta^{-1}\Var(\check \kernel_{x, h_0} (B_1))$~(see the remark after Theorem 6 in \cite{adamczak2008tail}), which yields 
\begin{equation}\label{eq:var-mixing-bound}
    \beta^{-1}\Var(\check \kernel_{x, h_0}(B_1)) \leq \frac{1 + \rho}{1 - \rho}\|\kernel_{x, h_0}\|_{\bbP_\pi,2}^2 \quad \text{for any $x$ and $h_0$.}
\end{equation}
Further note that
\begin{equation*}
    \begin{aligned}
        \|\kernel_{x, h_0} \|_{\bbP_\pi, 2}^2 = \int_0^1 \kernel_{x, h_0}^2(y) \pi(y) dy &= h_0^{-1} \int h_0^{-1}\kernel^2((x - y)/h_0)\pi(y) dy\\
        &= h_0^{-1}\int\kernel^2(z)\pi( x - h_0 z ) dy \leq C h_0^{-1},
    \end{aligned}
\end{equation*}
where $C =  \pi_u \int\kernel^2(z) dz$. As a last step, plug in the asymptotic bandwidth $h_0 = h$ into (\ref{eq:var-mixing-bound}) to derive our desired upper bound.

\end{proof}

Equipped with the proof of Proposition \ref{prop:block-var}, we are ready to prove Lemma \ref{lem:diff-cond-check}. 
\begin{proof}[Proof of Lemma \ref{lem:diff-cond-check}]
Here, $C, n_0$ are generic positive constants that can differ in value from place to place, while depending on the parameters $\beta, \pi, \kernel, \theta, \|\tau_{\check\alpha}\|_{\bbPcheck, \psi_1}, a$ and $\calX$.

Condition (A1) is satisfied since the function class $\classkernel$ is indexed by $x \in \calX \subset \R$ where $\R$ is separable. Choose a separable~(hence countable) subset of $\calX$, say $\calX_0$ and let the countable sub-class be $\calG = \{ \kernel_{x_k, h} : x_k \in \calX_0\}$. For any chosen $\kernel_{x, h}$, select a sequence of functions from $\calG$, $\kernel_{x_k, h}$ such that $x_k \to x$. Then by lipschitz continuity of $\kernel$, with lipschitz constant $L$, we have 
\begin{equation*}
    | \kernel_{x_k, h}(y) - \kernel_{x, h}(y) | \leq h^{-2}L |x - x_k| \quad \text{for all $y \in \calX$.}
\end{equation*}
Lastly apply the lower bound of $\sigma(x)$ stated in Proposition \ref{prop:block-var} and we can show (A1) holds for $\classkernel$. 

Condition (A2) is satisfied by exploiting the kernel's symmetric property. Recall $\{ax + b : a, b \in \R\}$ is parametrized by finite vector space~(dimension $2$) so that it is a VC-type~(see Lemma 2.6.15 of \cite{VW1996}). As symmetric $\kernel$ is expressed as the sum of monotone functions, the class of functions $\{\kernel_{x, h} : x \in \calX\}$ is VC-type~(see Lemma 2.6.18 of \cite{VW1996}). The class of constant functions $\{1/(\beta^{-1/2}\sigma(x)): x\in \calX\}$ is contained in the class of functions parametrized by dimension $1$, so again it is a VC-type. Taking the point-wise product of functions $\kernel_{x, h}$ and $1/(\beta^{-1/2}\sigma(x))$, we see that the class $\classkernel$ is VC-type~(see Lemma 2.6.18 of \cite{VW1996}). So condition (A2) is satisfied with parameters that do not depend on $n$. So we may regard $K_n \leq C \log n$. 

We know that condition (A3) of Assumption \ref{assump : Gauss approx} is satisfied once (B1) and (B2) of Assumption \ref{assump : bootstrap} hold (see discussion following Assumption \ref{assump : bootstrap}). Hence, given that we shown (A1) and (A2) holds, it now suffices to check the conditions in Assumption \ref{assump : bootstrap}.

Note that $\Var(\brf_{x, h}(B_1)) = \beta$ and using Proposition \ref{prop:block-var} as well as $\|\pi\|_\infty < \infty$, $\int \kernel (z) dz  = 1$, change of variables yields
\begin{equation}\label{eq:bias-hn}
    \begin{aligned}
        \pi f_{x, h} = \frac{1}{\beta^{-1/2}\sigma(x)}\int_0^1 h^{-1}\kernel\big( (z - x)/h \big)\pi(z) dz \leq C h^{1/2} \quad \text{for $n \geq n_0$.}
    \end{aligned}
\end{equation}
Then notice
\begin{equation*}
    \begin{aligned}
        \big|\Var(\check f_{x, h}(B_1)) - \Var(\check f_{x, h}^\circ(B_1))\big| \leq 2\pi f_{x, h} \cdot \|\tau_{\check \alpha}\|_{\bbPcheck, 2} \Var^{1/2}(\check f_{x, h} (B_1)) + (\pi f_{x, h})^2 \cdot \|\tau_{\check \alpha}\|_{\bbPcheck, 2}^2
    \end{aligned}
\end{equation*}
where $\|\tau_{\check \alpha}\|_{\bbPcheck, 2} < \infty$, thereby yielding
\begin{equation*}
    \Var(\ccf_{x, h} (B_1)) \geq \beta - C h^{1/2}.
\end{equation*}
Then set $\underline \sigma = \beta/2$ so that there exists $n_0$ such that $n\geq n_0$ implies $h^{1/2} < \beta/2C$, thereby implying $\Var(\ccf_{x, h}(B_1)) \geq \underline \sigma > 0$. So the variance lower bound in (B1) of Assumption \ref{assump : bootstrap} holds. By Proposition \ref{prop:block-var}, we observe $\|F\|_\infty \leq D_n = O(h^{-1/2})$, so condition (B1) is satisfied with $D_n \leq Ch^{-1/2}$. 

Next, we check that the drift condition (C2) in (\ref{eq:drift-cond-two}) holds for the low-frequency chain of interest, so as to verify (B2) of Assumption \ref{assump : bootstrap} via Proposition \ref{prop:geom-drift}. Set $V: [0, 1]\to \R_+$ to be any constant function, then we can find positive constants $c, b$ satisfying $b \geq 1/c$, so that
\begin{equation*}
    e^{-V(x)} \int e^{V(y)} p_\Delta(x, y) dy = 1 \leq e^{-1/c + b}.
\end{equation*}
As the small set $S = [0, 1]$ is the whole argument space, setting $b = 1/c$ will satisfy condition (C2) in (\ref{eq:drift-cond-two}); so Proposition \ref{prop:geom-drift} implies that all the moment conditions of stopping times $\tau_S$, $\tau_{\check \alpha}$, $\sigma_{\check \alpha}$ in Assumptions \ref{assump : Gauss approx} and \ref{assump : bootstrap} are satisfied. 

Condition {(B3)} is automatically satisfied by (\ref{eq:transition-est}). The uniform density $\nu'$~(of $\nu$ in (\ref{eq:minorization-diff})) is sufficient for {(B4)}. Lastly, the fact that $p_\Delta$ is uniformly upper bounded gives a way of constructing an estimator that is uniformly upper bounded, thereby implying the existence of $R$ in (B5).



\end{proof}

\subsection{Studentization and coupling results}\label{sec:coupling}
Here we present how the estimator $\hat\sigma_n(x)$ approximates $\beta^{-1/2}\sigma(x)$ uniformly over $x \in \calX$. Then we provide the two coupling results that is used later to prove the validity of our confidence band.

First we present a version of Theorem \ref{thm: bootstrap consistency} and Lemma \ref{lem:cov-est} tailored for the samples $\{X_t\}_{t \in \mathbb N_0}$ of the diffusion process (\ref{eq:diffusion}) over the function class $\classkernel$. Note that in Proposition \ref{prop:block-var} and Lemma \ref{lem:diff-cond-check}, we set $h = n^{-a}$ for some $a>0$; using these results, we can claim that there exists a positive constant $n_0$ depending on $\calX, a, \pi, \kernel, \theta, \beta$, such that for $n \geq n_0$ we have $D_n \leq O(h^{-1/2})$ and $K_n \leq O(\log n)$ (up to constants depending on $\kernel$). The following result actively uses this fact.

\begin{corollary}\label{cor:Delta-bound}
Suppose the chain $\{X_t\}_{t \in \mathbb N_0}$ consists of low-frequency samples from the diffusion (\ref{eq:diffusion}) with $(\varrho, b)\in \Theta$, while satisfying all the conditions listed in Lemma \ref{lem:diff-cond-check}. Given positive constants $\underline c, \gamma, \underline{\gamma}, \overline \gamma > 0$ such that $\underline{\gamma} < \gamma < \overline \gamma < 1/2$, let $M_n$ and $h = n^{-a}$ be sequences satisfying $M_n \geq n^{\underline{\gamma}/2}$ and
\begin{equation}\label{eq:G-bootstrap-cond}
    n^\gamma h \geq n^{\underline{c}} M_nK_n.
\end{equation}
Given positive constants $n_0, c, c', C, C'$, we have the following for any $n \geq n_0$.
\begin{itemize}
    \item[(i)] For given $c, C$, we have
    \begin{equation}\label{eq:G-Delta-bound}
    \bbP\big( \Delta(\check{\classkernel^\circ}) > n^{-1/2 + \gamma} \big) \leq C\cdot \big\{ n^{1/2 - \gamma} h^{-1}{\alpha_n} + \exp ( -cn^{\underline \gamma \wedge \underline c} ) \big\} 
    \end{equation}    
    \item[(ii)] For given $c, c', C, C'$, we have
    \begin{equation*}
        \rho_n^* \leq C'\cdot \bigg\{\frac{\log N}{n^{1/4 - \overline{\gamma}/2}} + \frac{h^{-1/2} \log n}{n^{1/4 + \overline{\gamma}/2}} +  \exp\big( -c'n^{(1/2 - \overline \gamma)\wedge(\overline{\gamma} - \gamma)} \big) \bigg\}
    \end{equation*}
    with probability at least $1 - C\cdot \xi_n^*$ where $\xi_n^*$ satisfies
    \begin{equation}\label{eq:xi_n_star}
        \xi_n^* \leq \frac{\sqrt{\log n}}{n^{1/4 - \overline \gamma/2}} + \frac{h^{-1/2}\log^{5/4} n}{n^{1/4 + \overline \gamma/2}} + n^{1/2 - \gamma} h^{-1} \alpha_n  + \exp\big(-c n^{\underline{\gamma} \wedge (\underline{c}/8)}\big).
    \end{equation}
\end{itemize}
The constants $n_0, c, c', C, C'$ depend only on $\calX, a, \pi, \kernel, \theta, \underline{c}, \vartheta_B, \|\sigma_{\check \alpha}\|_{\bbP, \psi_1}, \|\tau_{\check \alpha}\|_{\bbPcheck, \psi_1}$.
\end{corollary}

With abuse of notation (see discussion in \ref{subsec:boot-consistency}), we set
\begin{equation}\label{eq:over-Delta}
    \overline \Delta_n = n^{1/2-\gamma}h^{-1}\alpha_n + \exp(-cn^{\underline \gamma \wedge \underline c})
\end{equation}
so as to simplify the presentation of (\ref{eq:G-Delta-bound}).

\begin{proof}[Proof of Corollary \ref{cor:Delta-bound}]
    Here, $c, C, n_0$ are generic positive constants that depend on the parameters stated in Corollary \ref{cor:Delta-bound}; its values may differ from place to place. Using Lemma \ref{lem:diff-cond-check}, we see that $\{X_t\}_{t \in \mathbb N_0}$ satisfies conditions in Assumption \ref{assump : bootstrap}. Also we can observe $D_n$ has the order of $O(h^{-1/2})$ as $\sigma(x) \geq c h^{-1/2} > 0$ by Proposition \ref{prop:block-var} whenever $n \geq n_0$. Using change of variables, Proposition \ref{prop:block-var} implies 
\begin{equation*}
    \pi f_{x, h} = \frac{1}{\beta^{-1/2}\sigma(x)} \int_0^1 h^{-1} \kernel\big( (z - x)/h \big) \pi(z) dz \leq C h^{1/2}.
\end{equation*}
So $u(\classkernel, \pi) \leq D_n$. Further, $K_n \leq C\log n$ by Lemma \ref{lem:diff-cond-check}. Then directly apply Lemma \ref{lem:cov-est} and Theorem \ref{thm: bootstrap consistency}.
\end{proof}

\underline{Estimation of $\beta^{-1/2}\sigma(x)$}: We present a concentration result for the relative error between $\hat{\sigma}_n(x)$ and $\beta^{-1/2}\sigma(x)$ that is uniform over $x \in \calX$. Recall that the estimation of the scaled variance $\beta^{-1}\sigma^2(x) = \beta^{-1}\Var(\check \kernel_{x, h}(B_1))$ is done via
\begin{equation}\label{eq:hat-sigma}
    \hat{\sigma}^2(x) = \E\big[\{ \bbG_n^\zeta(\kernel_{x, h})\}^2 |X_0^{n-1} \big] = \frac{1}{n}\sum_{i = 1}^{\hat i_n} \Big\{ \check\kernel_{x, h}(\what B_i) - \ell(\what B_i) \what{\pi} \kernel_{x, h} \Big\}^2.
\end{equation}

\begin{lemma}[Relative error]\label{lem:studentizing-rel-error}
    Assume the identical setting and conditions presented in Corollary \ref{cor:Delta-bound}. Suppose the sequences $n^{1/2 - \gamma}h^{-1}{\alpha_n}$ and $\sqrt{h}n^{1/4 + \overline{\gamma}/2}$ decrease and increase at a polynomial rate respectively, and the sequence $h$ further satisfies $\sqrt{n}h^{1/2} \leq n^\gamma$.
    Then there exist polynomially decaying sequences $\varepsilon_{2, n}, \delta_{2, n}$ such that
    \begin{equation*}\label{eq:normalizer-est}
        \bbP\Bigg( \sup_{x \in \calX} \bigg| \frac{\hat{\sigma}_n(x)}{\beta^{-1/2}\sigma(x)} - 1 \bigg| > \varepsilon_{2, n}\Bigg) \leq \delta_{2, n}.
    \end{equation*}
\end{lemma}

\begin{proof}[Proof of Lemma \ref{lem:studentizing-rel-error}]
Here, $c, C, n_0$ are generic positive constants that depend on the parameters specified in Corollary \ref{cor:Delta-bound}; its values can differ from place to place.

    Observe 
\begin{equation*}
    \begin{aligned}
        \frac{\hat\sigma^2(x)}{\beta^{-1}\sigma^2(x)} = \frac{1}{n}\sum_{i = 1}^{\hat i_n} \Big\{ \check f_{x, h}(\what B_i) - \ell(\what B_i) \what{\pi} f_{x, h} \Big\}^2 \quad \text{and}\quad \frac{\beta^{-1}\sigma^{2}(x)}{\beta^{-1}\sigma^{2}(x)} = \beta^{-1}\Var(\check f_{x, h}(B_1))
    \end{aligned}
\end{equation*}
where we emphasize 
\begin{equation*}
    \begin{aligned}
        \big|\Var(\check f_{x, h}(B_1)) - \Var(\check f_{x, h}^\circ(B_1))\big| \leq 2\pi f_{x, h} \cdot \|\tau_{\check \alpha}\|_{\bbPcheck, 2} \Var^{1/2}(\check f_{x, h} (B_1)) + (\pi f_{x, h})^2 \cdot \|\tau_{\check \alpha}\|_{\bbPcheck, 2}^2. 
    \end{aligned}
\end{equation*}
Note that $\Var(\check f_x(B_1)) = \beta$. Using change of variables, Proposition \ref{prop:block-var} implies 
\begin{equation*}
    \pi f_{x, h} = \frac{1}{\beta^{-1/2}\sigma(x)} \int_0^1 h^{-1} \kernel\big( (z - x)/h \big) \pi(z) dz \leq C h^{1/2}.
\end{equation*}
So collecting what we have, we observe
\begin{equation*}
    \begin{aligned}
        \bigg| \frac{\hat \sigma^2(x)}{\beta^{-1}\sigma^2(x)} - 1 \bigg|  \leq& \big| \E\big[\{ \bbG_n^\zeta(f_{x, h})\}^2 |X_0^{n-1} \big] - \beta^{-1}\Var(\check f_{x, h}^\circ(B_1)) \big|\\
        &+ \beta^{-1}\big| \Var(\check f_{x, h}(B_1)) - \Var(\check f_{x, h}^\circ(B_1)) \big|
    \end{aligned}
\end{equation*}
where the second term of the right-hand side of the above display is further bounded by $Ch^{1/2}$ whenever $n \geq n_0$. Note that $\E\big[ \{\bbG_n^\zeta(f_{x, h})\}^2 |X_0^{n-1}\big]$ is different from $\what \Gamma(\ccf_{x, h}, \ccf_{x, h})$. Observe the bound $\sup_{x\in \calX}\big|\what \Gamma(\ccf_{x, h}, \ccf_{x, h}) - \beta^{-1}\Var(\ccf_{x, h}(B_1))\big| \leq \Delta(\check \classkernel^\circ)$ and the decomposition 
\begin{equation*}
    \begin{aligned}
    &\E\big[ \{ \bbG_n^\zeta(f_{x, h})\}^2|X_0^{n-1} \big]\\
    &=\what \Gamma(\ccf_{x, h}, \ccf_{x, h}) + (\what \pi f_{x, h} - \pi f_{x, h})^2\E\big[ \{ \bbG_n^\zeta(\ccone) \}^2 |X_0^{n-1} \big] - 2(\what \pi f_{x, h} - \pi f_{x, h})\E\big[ \tilde \bbG_n^\zeta(\ccf_{x, h}) \tilde \bbG_n^\zeta(\ccone) |X_0^{n-1}\big] .
    \end{aligned}
\end{equation*}
So observe the bound
\begin{equation*}
    \begin{aligned}
        \big| \E\big[\{ \bbG_n^\zeta(f_{x, h})\}^2 |X_0^{n-1} \big] - \beta^{-1}\Var(\check f_{x, h}^\circ(B_1)) \big| \leq& \Delta (\check \classkernel^\circ) + (\pi f_{x, h} - \what{\pi} f_{x, h})^2 \cdot \E\big[\{\tilde \bbG_n^\zeta(\ccone)\}^2|X_0^{n-1}\big]\\
        &+2 \big|\pi f_{x, h} - \what{\pi} f_{x, h} \big| \cdot \big| \E\big[ \tilde \bbG_n^\zeta(\check f_{x, h}^\circ)\cdot \tilde \bbG_n^\zeta(\ccone) |X_0^{n-1}\big] \big|.
    \end{aligned}
\end{equation*}
Notice that as long as $x \geq 0$, $|x - 1| \leq |x^2 - 1|$, which implies $| \hat\sigma_n(x)/\beta^{-1/2}\sigma(x) - 1  | \leq | \hat{\sigma}_n^2(x)/\beta^{-1}\sigma^2(x) - 1 |$. So for some sequence $\delta_n'$, we have
\begin{equation*}
    \begin{aligned}
        &\bbP\bigg( \sup_{x\in \calX}\bigg| \frac{\hat \sigma(x)}{\beta^{-1/2}\sigma(x)} - 1 \bigg| > \delta_n' \bigg) \\
        &\leq \bbP \big( \Delta(\check \classkernel^\circ) > c\delta_n'  \big) + \bbP\big( h^{1/2} > c\delta_n' \big)\\
        &+ \bbP\Big( \|\bbG_n\|_{\classkernel}^2 \cdot \beta^{-1}\E[\ell^2(B_1)] > cn\delta_n' \Big)
        + \bbP\Big( \|\bbG_n\|_{\classkernel} \cdot \big|\beta^{-1}\E[\ccf_{x, h}(B_1)\ell(B_1)] \big| > c\sqrt{n}\delta_n' \Big)\\
        &+ \bbP\bigg( \|\bbG_n\|_{\classkernel}^2 \cdot \bigg| \frac{1}{n}\sum_{i = 1}^{\hat i_n} \ell^2(\what B_i) - \beta^{-1}\E[\ell^2(B_1)] \bigg| > cn\delta_n' \bigg)\\
        &+ \bbP\bigg( \| \bbG_n \|_{\classkernel}\cdot \sup_{x} \bigg| \frac{1}{n}\sum_{i = 1}^{\hat i_n} \check f_{x, h}(\what B_i) \ell(\what B_i) - \beta^{-1}\E[\check f_{x, h}(B_1) \ell(B_1)] \bigg| >c \sqrt{n}\delta_n' \bigg)\\
        &+ \bbP\bigg( u(\classkernel, \pi) \cdot \| \bbG_n \|_{\classkernel}\cdot \bigg| \frac{1}{n}\sum_{i = 1}^{\hat i_n} \ell^2(\what B_i) - \beta^{-1}\E[\ell^2(B_1)] \bigg| > c\sqrt{n}\delta_n' \bigg)
    \end{aligned}
\end{equation*}

Set $\delta_n' = n^{-1/2 + \gamma}$, then Corollary \ref{cor:Delta-bound} implies $\bbP\big( \Delta(\check \classkernel^\circ) > c\delta_n' \big) \leq C\overline \Delta_n$ (see (\ref{eq:over-Delta}) for details on $\overline\Delta_n$). The assumption $\sqrt{n}h^{1/2} \leq n^\gamma$ further implies $\bbP\big( h^{1/2} > c\delta_n' \big) = 0$. 
Recall that $u(\classkernel, \pi) = \sup_{f_{x, h} \in \classkernel}|\pi f_{x, h}|\leq D_n$ as $u(\classkernel, \pi) \leq h^{1/2}$ and $D_n$ scales as $O(h^{-1/2})$ whenever $n\geq n_0$. Further observe $| \E[\ccf_x(B_1) \ell(B_1)] | \leq C\E[\ell^2(B_1)]D_n$. So by Corollary \ref{cor:Delta-bound}, when $n\geq n_0$, the last three terms in the above display can be bounded by
\begin{equation*}
    \bbP\Big( \| \bbG_n \|_\classkernel^2 > c \sqrt{n\delta_n'} \Big) + \bbP \Big( \| \bbG_n \|_\classkernel > c\sqrt{n}\delta_n'/D_n \Big) + C\overline \Delta_n.
\end{equation*}

Here $K_n \leq C\log n$ and $\tilde \xi_n$ defined in (\ref{eq:gauss-approx-bootstrap}) can be adopted to this setting, where $D_n$ in $\tilde \xi_n$ is substituted by $Ch^{-1/2}$. Note that $\sqrt{n\delta_n'} = n^{1/4 - \gamma/2}n^\gamma \geq n^\gamma/D_n$, so again refer to Corollary \ref{cor:Delta-bound}~(or the proof of Theorem \ref{thm: bootstrap consistency}) and observe
\begin{equation*}
    \begin{aligned}
        \bbP\big( \| \bbG_n \|_{\classkernel} > n^{\gamma}/D_n \big) \leq \bbP\big( \| G \|_{\classkernel} > n^{\gamma}/D_n \big) + C  \tilde \xi_n \leq C \cdot \{ \exp (-c n^{2\underline c}) + \tilde \xi_n\}, 
    \end{aligned}
\end{equation*}
where the last inequality follows from the fact that $\E[\|G\|_\classkernel] \leq C \sqrt{\log n}h^{-1/2}$~(see the proof of Theorem \ref{thm: bootstrap consistency}) and that the condition here implies $n^\gamma h \geq n^{\underline c}$. 
Hence we have 
\begin{equation*}
    \bbP\bigg( \sup_{x\in \calX}\bigg| \frac{\hat \sigma(x)}{\beta^{-1/2}\sigma(x)} - 1 \bigg| > \delta_n' \bigg) \leq C \big\{ \tilde \xi_n + \overline \Delta_n \big\}.
\end{equation*}
The conditions assumed on the sequences $n^{1/2 - \gamma}h^{-1}{\alpha_n}$ and $\sqrt{h}n^{1/4 + \overline{\gamma}/2}$~(decreasing and increasing respectively at a polynomial rate) imply that the sequences $\overline \Delta_n$ and $\tilde \xi_n$ vanish polynomially.

\end{proof}

\newcommand{\cnormkernel}{\check{f}}

\underline{\textit{Coupling between $\beta^{-1/2}\| \bbG_n^{(1)} \|_{\check{\classkernel}^\circ}$ and $\| G \|_{{\classkernel}}$}}: 
From the definition (\ref{eq:sum-decomp}), we know that 
\begin{equation*}
    \beta^{-1/2}\bbG_n^{(1)}(f_{x, h}^\circ) = n^{-1/2} \sum_{i = 1}^{n^*}\ccf_{x, h}(B_i)
\end{equation*} 
has center zero, i.e. $\E[\ccf_{x, h}(B_i)] = \E[\brf_{x, h}(B_i)] - \E_{\check \alpha}[\tau_{\check \alpha}]\pi f_{x, h} = 0$ by (\ref{eq:block-mean-pi}).
Then for a centered Gaussian process $G(f_{x, h})$ with covariance
\begin{equation*}
    \beta^{-1} \E[ (\brf_{x, h}(B_1) - \ell(B_1) \pi f_{x, h})(\brf_{y, h}(B_1) - \ell(B_1) \pi f_{y, h}) ],
\end{equation*}
apply Theorem A.1 of \cite{CCK2014bAoS}: there exists an absolute constant $c>0$ and a random variable $W_0$ satisfying $W_0 \stackrel{d}{=} \| G \|_{{\classkernel}}$, such that the following holds for any $\gamma_0 \in (0, 1)$,
\begin{equation}\label{eq:first-coupling}
    \begin{aligned}
        \bbP\left( \Big| \beta^{-1/2} \|\bbG_n^{(1)} \|_{\check{\classkernel}^\circ} - W_0 \Big| > \frac{\log n}{\sqrt{h\gamma_0 n}} + \frac{\log n}{\sqrt{h\gamma_0}n^{1/4}} + \frac{\log n}{\sqrt{h} \gamma_0^{1/3} n^{1/6}} \right) \leq c \left( \gamma_0 + \frac{\log n}{n} \right).
    \end{aligned}
\end{equation}
So if there exists sequences $\gamma_0, h \downarrow 0$~(we omit dependence of $\gamma_0$ on $n$ as well) that vanishes at a polynomial rate such that
    \begin{equation*}
    \sqrt{h\gamma_0 n}, \ \sqrt{h\gamma_0}n^{1/4}, \ \sqrt{h} \gamma_0^{1/3}n^{1/6} \ \text{increases polynomially} ,
\end{equation*}
then there exists polynomially decreasing sequences $\varepsilon_{0,n}, \delta_{0,n}$ satisfying 
\begin{equation*}
 \bbP\Big( \big| \beta^{-1/2}\| \bbG_n^{(1)}\|_{\check \classkernel^\circ} - W_0 \big| > \varepsilon_{0,n}\Big) \leq \delta_{0,n}.
 \label{second coupling}
\end{equation*}

\underline{\textit{Coupling between $\| \bbG_n^\zeta \|_{\classkernel^{(n)}}$ and $\| G \|_{\classkernel}$}}:  
For any fixed constant $\gamma \in (0, 1/2)$, define the events $\calE_j(\classkernel), j = 0, 1, 2, 3$ as was done in (\ref{eq:events}) but with the function class $\calF = \classkernel$. Motivated by Lemma \ref{lem:studentizing-rel-error}, we define an additional event,
\begin{equation*}
\begin{aligned}
    \calE_4(\classkernel) &:= \Big\{ \sup_{x\in \calX} \big| \hat \sigma(x) / (\beta^{-1/2}\sigma(x)) - 1 \big| \leq n^{-1/2 + \gamma} \Big\}.
\end{aligned}
\end{equation*}
The event of interest is
\begin{equation}\label{eq:calE-prime}
    \calE' := \cap_{j = 0}^4 \calE_j(\classkernel).
\end{equation}
Further recall the multiplier bootstrap statistics of interest
\begin{equation*}
    \bbG_n^\zeta\big(f_{x, h}^{(n)}\big) = \frac{1}{\sqrt{n}} \sum_{i = 1}^{\hat i_n} \zeta_i \cdot \Big\{ \check f_{x, h}^{(n)}(\what B_i) - \ell(\what B_i) \what{\pi} f_{x, h}^{(n)} \Big\} \quad \text{for $f_{x, h}^{(n)} \in \classkernel^{(n)}$.}
\end{equation*}

\begin{lemma}[Coupling]\label{lem: coupling}
Let $\{X_t\}_{t \in \mathbb N_0}$ be the low-frequency samples from (\ref{eq:diffusion}) with $(\varrho, b)\in \Theta$, while satisfying all the conditions listed in Corollary \ref{cor:Delta-bound}. Suppose the constant $\gamma$ and the bandwidth $h$ defined in Corollary \ref{cor:Delta-bound}
satisfies $\sqrt{n}h^{1/2} \leq n^\gamma$ and the sequence $h \alpha_n^{-2/3}$ increases polynomially. Then for any $X_{0}^{n-1}\in \calE'$,
there exists a random variable $W_1 \stackrel{d}{=} \| G \|_{\classkernel}$ such that
\[ \bbP\Big( \big|\|\bbG^\zeta_{n}\|_{\classkernel^{(n)}} - W_1 \big| > \varepsilon_{1,n}|X_0^{n-1}\Big) \leq \delta_{1,n} \]
where $\varepsilon_{1,n}, \delta_{1,n}$ vanishes at a polynomial rate. Further $\bbP(\calE') \geq 1 - O(\xi_n^*),$ where $\xi_n^*$ defined in (\ref{eq:xi_n_star}) vanishes at a polynomial rate. 
\end{lemma}

\begin{proof}[Proof of Lemma \ref{lem: coupling}]
The proof is divided into several steps. {Here, $c, C, C', n_0$} are generic positive constants that depend on the parameters specified in Corollary \ref{cor:Delta-bound} and its vlaues can differ from place to place. 

\underline{Step 1} (High-probability event). The condition $\sqrt{n}h^{1/2} \leq n^\gamma$ yields $n^{1/2 - \gamma} h^{-1} \alpha_n \leq h^{-3/2}\alpha_n$, so the polynomially growing sequence $h \alpha_n^{-2/3}$ implies that the sequence $n^{1/2 - \gamma}h^{-1}\alpha_n$ is polynomially vanishing. Also we can see that $\sqrt{h}n^{1/4 + \overline \gamma/2} \geq (nh^{3/2})^{1/2}$ and the fact that $\alpha_n$ has order strictly smaller than $O(n^{-1})$ implies that the sequence $\sqrt{h}n^{1/4 + \overline \gamma/2}$ grows polynomially. So we see that the conditions of Lemma \ref{lem:studentizing-rel-error} is satisfied, thereby implying 
$\bbP(\calE_4^c(\classkernel)) \leq C\big\{ \tilde \xi_n + \overline \Delta_n \big\}$, the order being polynomial. Further, we are assuming the conditions in Corollary \ref{cor:Delta-bound}~(hence the conditions in Lemma \ref{lem:diff-cond-check}), then we may apply the proof reasoning of Theorem \ref{thm: bootstrap consistency} so as to observe $\bbP(\calE_j^c(\classkernel))\leq C\overline \Delta_n$ for $j = 1, 2, 3$ and $\bbP(\calE_0^c(\classkernel)) \leq C\tilde \xi_n$.

Define the error
\[
\hat e(x) := \bigg( \frac{\beta^{-1/2}\sigma(x)}{\hat \sigma_n(x)} - 1 \bigg),
\]
and define a function class that depends on data, 
\begin{equation*}
    \classemp = \Big\{ \empkernel_{x, h} := f_{x, h}\cdot \hat e(x) : x \in \calX\Big\} \quad \text{with its envelope $V$}.
\end{equation*}
For a sequence $\delta_n'$, consider the decomposition $\bbG_n^\zeta(f_{x, h}^{(n)}) = \bbG_n^\zeta(f_{x, h}) + \bbG_n^\zeta( \empkernel_{x, h})$ which yields
\begin{equation*}
    \begin{aligned}
        \bbP\Big( \big| \| \bbG_n^\zeta \|_{\classkernel^{(n)}} - W_1 \big| > \delta_n' \big| X_0^{n-1}  \Big) \leq \bbP\Big( \| \bbG_n^\zeta \|_{\classemp} > \delta_n'/2 \big| X_0^{n-1} \Big) + \bbP\Big( \big| \| \bbG_n^\zeta \|_{\classkernel} - W_1 \big|  > \delta_n'/2 \big| X_0^{n-1} \Big).
    \end{aligned}
\end{equation*}
For what follows, we bound the two terms on the right-hand side of the above display.

\underline{Step 2}. On the event $\calE_4$ notice that $\E\big[ \{ \bbG_n^\zeta(\empkernel_{x, h}) \}^2 | X_0^{n-1}\big] = \hat e(x)^2 \hat{\sigma}^2(x)/\beta^{-1}\sigma^2(x) \leq n^{-1 + 2\gamma}$ for all $x \in \calX$. For a generic function $f$ on $E$, we define its (empirically) centered version for notational convenience,
\begin{equation*}
    f^{\diamond} := f - \what{\pi} f.    
\end{equation*}
Accordingly, define a function class for functions $\empkernel_{x, h}^\diamond = \empkernel_{x, h} - \what \pi \empkernel_{x, h} = \big( f_{x, h} - \what{\pi} f_{x, h} \big) \cdot \hat e(x) $,
\begin{equation}\label{eq:intermediate-ftn}
    \classemp^{\diamond} := \Big\{ \empkernel_{x, h}^{\diamond} =  \big( f_{x, h} - \what{\pi} f_{x, h} \big) \cdot \hat e(x) : x \in \calX \Big\}.
\end{equation}
Refer to (\ref{eq:non-cent-boot}) and observe the equivalence
\begin{equation*}
    \bbG_n^\zeta(\empkernel_{x, h}) = \frac{1}{\sqrt{n}} \sum_{i = 1}^{\hat i_n}\zeta_i \big( \check \empkernel_{x, h}(\what B_i) - \ell(\what B_i) \what \pi \empkernel_{x, h} \big) = \frac{1}{\sqrt{n}} \sum_{i = 1}^{\hat i_n} \zeta_i \check \empkernel^\diamond_{x, h} (\what B_i) = \tilde \bbG_n^\zeta (\check \empkernel_{x, h}^\diamond);
\end{equation*}
so we are dealing with a conditional centered Gaussian process indexed by $\check \empkernel_{x, h}^\diamond \in \check \classemp^\diamond = \{ \check \empkernel_{x, h}^{\diamond} : \empkernel_{x, h}^{ \diamond} \in \classemp^{\diamond} \}$.



On the event $\calE_0(\classkernel)$, we have $\sup_{x \in \calX} | \pi f_{x, h} - \what{\pi} f_{x, h} | \leq n^{-1/2 + \gamma}$. Further note that $\sup_{x \in \calX} | \pi f_x | \leq Ch^{1/2}$ for $n \geq n_0$. So on the event $\calE_0(\classkernel)$, we see that $\sup_{x \in \calX}\big| \what{\pi} f_{x, h}\big| \leq 1$ for $n\geq n_0$. Also observe $|\hat e(x)|\leq 1$ for all $x \in \calX$ on $\calE_4(\classkernel)$. Hence on the event $\calE_0(\classkernel)\cap\calE_4(\classkernel)$, we observe $\classemp^\diamond \subset \{ a\cdot (f_{x, h} - b): a, b \in [-1, 1] \}$, thereby implying $\classemp^{\diamond}$ is of VC-type, and further $\check \classemp^{\diamond}$ is VC-type via Lemma \ref{lem: covering number} (see proof of Lemma \ref{lem:diff-cond-check} for details). 


For an empirical discrete measure $\hat Q_n := \frac{1}{\hat i_n} \sum_{i = 1}^{\hat i_n} \delta_{\what B_i}$ on $\check E$, we see that the semi-metric $\big(\rho_2(\check \empkernel_{x, h}^{\diamond}, \check \empkernel_{y, h}^{\diamond}) \big)^2:= E\big[ ( \bbG_n^\zeta(f_{x, h}^\dagger) - \bbG_n^\zeta(f_{y, h}^\dagger) )^2 |X_0^{n-1}\big] = \E\big[ ( \tilde \bbG_n^\zeta(\check \empkernel_{x, h}^\diamond - \tilde \bbG_n^\zeta(\check \empkernel_{y, h}^\diamond) )^2 |X_0^{n-1} \big]$ on $\check \classemp^{\diamond}$ satisfies
\begin{equation*}
    \begin{aligned}
        \rho_2^2(\check \empkernel_{x, h}^{\diamond}, \check \empkernel_{y, h}^{\diamond}) = \frac{1}{n} \sum_{i = 1}^{\hat i_n} \Big\{ \check \empkernel_{x, h}^{\diamond}(\what B_i) - \check \empkernel_{y, h}^{\diamond}(\what B_i) \Big\}^2 \leq \| \check \empkernel_{x, h}^{\diamond} - \check \empkernel_{y, h}^{\diamond} \|_{\check Q_n, 2}^2.
    \end{aligned}
\end{equation*}
Further observe that on the event of interest, 
\begin{equation*}
    \sup_{x \in \calX}| \empkernel_{x, h}^{ \diamond} | \leq \sup_{x \in \calX}|\hat e(x)|\cdot (|f_{x, h}| + n^{-1}\sum_{i = 1}^{n}|f_{x, h}(X_i)|)  \leq C h^{-1/2}    ,
\end{equation*}
which implies that the envelope of the class $\classemp^{\diamond}$, say $V^{\diamond}$, scales as $h^{-1/2}$. Hence we observe
\begin{equation*}
    \| \check V^{\diamond} \|_{\check Q_n, 2}^2 \leq Cn h^{-1}\cdot \bigg( \frac{1}{n}\sum_{i =1}^{\hat i_n}\ell^2(\what B_i) - \beta^{-1}\E[\ell^2(B_1)] + \beta^{-1}\E[\ell^2(B_1)] \bigg) \leq C' nh^{-1}
\end{equation*}
on the event $\calE_3(\classkernel)$.

So for $X_0^n \in \calE_0(\classkernel) \cap \calE_3(\classkernel) \cap \calE_4(\classkernel)$, the class of functions $\check \classemp^{\diamond}$ on $\check E$ satisfies
\begin{equation*}
    N(\check\classemp^{\diamond}, \rho_2, \eta) \leq N\big(\check \classemp^{\diamond}, \| \cdot \|_{\check Q_n, 2}, \| \check V^{\diamond}\|_{\check Q_n, 2} \cdot \{\eta / ( \sqrt{n} h^{-1/2} )\}\big) \leq \bigg( \frac{c \sqrt{n}h^{-1/2}}{\eta} \bigg)^{C}.
\end{equation*}
Maximal inequality \cite{Dudley2014UCLT} implies a bound on the center of the supremum of the Gaussian empirical process, 
\begin{equation*}
    \begin{aligned}
        \E\big[ \| \bbG_n^\zeta \|_{\classemp} |X_0^{n-1}\big] = \E\big[ \|\tilde \bbG_n^\zeta\|_{\check \classemp^\diamond} |X_0^{n-1}\big] \leq \int_0^{n^{-1/2 + \gamma}} \sqrt{\log N(\check \classemp^{\diamond}, \rho_2, \eta) } d\eta \leq C\sqrt{\log n}\cdot n^{-1/2 + \gamma}.
    \end{aligned}
\end{equation*}
Setting $\delta_n' = n^{-1/2 + \overline \gamma}$ with $\overline \gamma > \gamma$, Borel-Sudakov-Tsirel'son inequality~\cite{ledoux2001concentration}
yields
\begin{equation*}
    \begin{aligned}
        \bbP\Big( \| \bbG_n^\zeta \|_{\classemp} > \delta_n'/2 \big| X_0^{n-1} \Big) &= \bbP\Big( \| \tilde \bbG_n^\zeta \|_{\check \classemp^\diamond} > \E\big[\| \tilde \bbG_n^\zeta \|_{\check \classemp^\diamond}|X_0^{n-1}\big] + \big\{ \delta_{n}'/2 - \E\big[\| \tilde \bbG_n^\zeta \|_{\check \classemp^\diamond}|X_0^{n-1}\big] \big\} \big| X_0^{n-1}\Big)\\
        &\leq C\exp(-cn^{2\overline \gamma - 2\gamma}).
    \end{aligned}
\end{equation*}

\underline{Step 3}. We emphasize that the generic constants $C, C' > 0$ in Step 3 depend only up to the parameter $\beta$ and $\|\tau_{\check \alpha}\|_{\bbPcheck, \psi_1}$. This is to highlight that we can use the conditional Strassen's theorem on the event of interest. 

Let $\delta_n = n^{-1/4 + \overline{\gamma}/2}$ for some $1/2 > \overline \gamma > \gamma > 0$. Then under the condition (\ref{eq:G-bootstrap-cond}), $K_n$ is at most logarithmic as the VC-type parameters for $\classkernel$ are constants independent to $n$. The conditions specified here suffices the conditions for bootstrap consistency, so we refer to the proof of Theorem \ref{thm: bootstrap consistency} and replicate the reasoning therein. Recall the decomposition $\bbG_n^\zeta(f_{x, h}) = \tilde \bbG_n^\zeta(\check f_{x, h}^\circ) - \tilde \bbG_n^\zeta(\ccone) \cdot \bbG_n(f_{x, h})/\sqrt{n}$ and the definition of $\tilde\bbG_n^\zeta(f_{x,h})$ in (\ref{eq:non-cent-boot}). Step 2 in the proof of Theorem \ref{thm: bootstrap consistency} directly yields
\begin{equation*}
    \bbP\big( \| \bbG_n \|_{\classkernel} \cdot | \tilde \bbG_n^\zeta(\ccone) | > \sqrt{n}\delta_n | X_0^{n-1} \big) \leq C \exp \big( -c n^{1/2-\overline \gamma} \big)
\end{equation*}
whenever $X_0^{n-1} \in \calE_0(\classkernel) \cap \calE_3(\classkernel)$. Hence for any measurable set $B$ on the real line, we observe
\begin{equation*}
    \begin{aligned}
        \bbP\big( \| \bbG_n^\zeta \|_{\classkernel} \in B |X_0^{n-1}\big) \leq \bbP\big( \| \tilde \bbG_n^\zeta \|_{\check \classkernel^\circ} \in B^{\delta_n} |X_0^{n-1}\big) + C \exp \big( -c n^{1/2-\overline \gamma} \big).
    \end{aligned}
\end{equation*}
Refer to Step 3 of the proof of Theorem \ref{thm: bootstrap consistency}, and we take the same discretization strategy over the class $\check \classkernel^\circ$, so as to attain
\begin{equation*}
    \begin{aligned}
        \bbP \big( \| \tilde\bbG_n^\zeta \|_{\check \classkernel_\varepsilon^\circ} > \delta_n | X_0^{n-1} \big)  \leq C\exp \{ -c n^{\overline \gamma - \gamma} \} \quad \text{and} \quad \bbP \big( \beta^{-1/2} \| W \|_{\check \classkernel^\circ_\varepsilon} > \delta_n \big) \leq C n^{-1/4 - \overline \gamma/2} h^{-1/2} K_n.
    \end{aligned}
\end{equation*}
on $\calE_1(\classkernel)\cap\calE_2(\classkernel)$. The first inequality in the above display implies 
\begin{equation*}
    \bbP\big( \|\tilde \bbG_n^\zeta \|_{\check \classkernel^\circ} \in B^{\delta_n} | X_0^{n-1} \big) \leq \bbP\big( \|\tilde \bbG_n^\zeta \|_{\check \classkernel_N^\circ} \in B^{2\delta_n} | X_0^{n-1} \big) + C \exp\{-cn^{\overline{\gamma} - \gamma}\} 
\end{equation*}


For a fixed $\kappa>0$, define $\psi: = 2\sqrt{\log n}/\kappa$ and $\varepsilon_{\psi, \kappa} := \sqrt{e^{1 - \psi^2\kappa^2}\psi^2 \kappa^2}$ so that $\psi^{-1} = \frac{\kappa}{2\sqrt{\log n}} < \kappa$ for $n \geq 2$ and
\begin{equation*}
    \frac{1}{1 - \varepsilon_{\psi, \kappa}} \leq \frac{1}{1 - 2\sqrt{\log n}/n} \leq 1 + \frac{4}{n}
\end{equation*}
for $n$ satisfying $2\sqrt{\log n} \leq n/2$. Choose $f$ according to Lemma \ref{lem : soft approx}, to be the smooth function corresponding to $\psi, \kappa$ and $B^{2\delta_n}$.

Recall that $\big(\tilde \bbG_n^\zeta(\check f_{x_j, h}^\circ)\big)_{j \in [N]}$ is conditionally centered Gaussian random vector with covariance $\E\big[ \tilde \bbG_n^\zeta(\check f_{x_j, h}^\circ) \tilde \bbG_n^\zeta(\check f_{x_k, h}^\circ) | X_0^{n-1} \big] = \what\Gamma(\check f_{x_j, h}^\circ, \check f_{x_k, h}^\circ)$. We compare $\big(\tilde \bbG_n^\zeta(\check f_{x_j, h}^\circ)\big)_{j \in [N]}$ with the discretized version of the reference Gaussian process $G(f_{x, h}) = \beta^{-1/2}W(f_{x, h}^\circ)$. As $\Delta(\check \classkernel^\circ_N) \leq \Delta(\check \classkernel^\circ)$, Lemma \ref{lem:Gaussian-comparison} item (ii) on the event $\calE_1(\classkernel)$ yields
\begin{equation*}\label{eq:coupling-second}
    \begin{aligned}
        &\E\big[ f\big(\| \tilde \bbG_n^\zeta\|_{\check\classkernel^\circ_N}\big)| X_0^{n-1} \big] - \E\big[ f(\beta^{-1/2}\| W \|_{\check \classkernel^\circ_N}) \big] \\
        &\leq C\cdot \left\{ \frac{1}{\kappa^2} \sqrt{\log n} \cdot n^{-1/2 + \gamma} + \frac{1}{\kappa} n^{-1/4 + \gamma /2} \sqrt{\log N} \right\}.
    \end{aligned}
\end{equation*}
Overall, the choice of the smooth function $f$ and Gaussian-to-Gaussian comparison yields
\begin{equation*}
    \begin{aligned}
        & \bbP\left( \| \tilde \bbG_{n}^{\zeta}\|_{\check\classkernel^\circ_N} \in B^{2\delta_n} |X_0^{n-1}\right) \\
        &\leq \E\big[ f\big( \| \tilde \bbG_{n}^{\zeta}\|_{\check\classkernel^\circ_N} \big) | X_0^{n-1}\big]/(1 - \varepsilon_{\psi, \kappa})  \\
        &\leq \E\big[ f(\beta^{-1/2}\|  W \|_{\check \classkernel^\circ_N}) \big]/(1 - \varepsilon_{\psi, \kappa}) + C\cdot \left\{ \frac{1}{\kappa^2} \sqrt{\log n} \cdot n^{-1/2 + \gamma}  + \frac{1}{\kappa} n^{-1/4 + \gamma /2} \sqrt{\log N} \right\}.
    \end{aligned}
\end{equation*}
As a last step, (i) invoke the upper bound of $f$ specified in Lemma \ref{lem : soft approx} with the fact $\varepsilon_{\psi, \kappa}/(1 - \varepsilon_{\psi, \kappa}) \leq C \sqrt{\log n / n}$ and (ii) take care of the discretization error of the reference Gaussian process $G(f_{x, h}) = \beta^{-1/2}W(\check f_{x, h}^\circ), f_{x, h} \in \classkernel$, so that we have
\begin{equation*}
    \begin{aligned}
        \E\big[ f(\beta^{-1/2 }\| W \|_{\check\classkernel^\circ_N}) \big]/(1 - \varepsilon_{\psi, \kappa}) &\leq \bbP \Big(  \beta^{-1/2 }\| W \|_{\check\classkernel^\circ_N} \in B^{2\delta_n + 3\kappa} \Big) + C\sqrt{\frac{\log n}{n}}\\
        &\leq \bbP\left( \| G \|_{\classkernel} \in B^{3\delta_n + 3\kappa} \right) +  C' \bigg\{ n^{-1/4 - \overline{\gamma}/2 } h^{-1/2} K_n + \sqrt{\frac{\log n}{n}} \bigg\}.
    \end{aligned}
\end{equation*}

Now choose $\kappa = \delta_n$ so that $n^{-1/4 + \gamma / 2} \kappa^{-1} = n^{\gamma/2 - \overline\gamma/2}$ and $\kappa^{-2}n^{-1/2 + \gamma} = n^{\gamma - \overline \gamma}$, which vanishes at a polynomial rate.
Also, set $\overline \gamma > 0$ and the order of $h$ such that $n^{1/2 - \overline \gamma}h^{1/2}$ vanishes polynomially, e.g. set $h = n^{-r}$ for some $r > 0$ and $\overline{\gamma} > 1/2 - r/2$. 
Note that we require $n^\gamma h$ to grow polyomially~(implied from (\ref{eq:G-bootstrap-cond})) so that $n^{-1/4-\overline \gamma / 2}h^{-1/2} = n^{-1/4}( n^{\overline \gamma }h )^{-1/2}$ naturally vanishes polynomially as well. Lastly as $\overline{\gamma } \in (0, 1/2)$,  $2\delta_n + 3\kappa = 5\delta_n$ vanishes polynomially. Set $\varepsilon_{1, n} =\delta_n$, so for each $X_0^{n-1} \in \calE'$, applying Strassen~(Theorem \ref{thm: Strassen}) gives us the $\varepsilon_{1, n}, \delta_{1, n}$ that vanish at a polynomial rate. Notice that such $\varepsilon_{1, n} = \delta_n$ dominates the order of $\delta_n' = n^{-1/2 + \overline \gamma}$, thereby implying
\begin{equation*}
    \bbP\Big( \| \bbG_n^\zeta \|_{\classemp} > \delta_n/2 \big|X_0^{n-1}\Big) \leq \bbP\Big( \| \bbG_n^\zeta \|_{\classemp} > \delta_n'/2 \big|X_0^{n-1}\Big) \leq C \exp (-c n^{2\overline \gamma - 2\gamma}) \leq C'\delta_{1, n}.
\end{equation*}

\end{proof}

\subsection{Polynomial guarantee of uniform confidence band}\label{sec : final analysis of band}
\newcommand{\poly}{o_{\mathrm{poly}}(1)}
Define $c(\alpha)$ to be the $(1 - \alpha)$th quantile of $\| G \|_{\classkernel}$
\begin{equation*}
    c(\alpha) := \inf \{ z: \bbP(\|G\|_{\classkernel} > z) \leq \alpha \},
\end{equation*}
where $G(f_{x, h})$ over $f_{x, h} \in \classkernel$ is the centered reference Gaussian process with covariance function 
\begin{equation*}
    \beta^{-1}\Cov\big(\ccf_{x, h}(B_1), \ccf_{y, h}(B_1)\big) = \beta^{-1}\E\big[ (\brf_{x, h}(B_1) - \ell(B_1) \pi f_{x, h})(\brf_{y, h}(B_1) - \ell(B_1) \pi f_{y, h}) \big].
\end{equation*}
We use the random variables $W_0$ and $W_1$ interchangeably~(each appearing in the two coupling results in Section \ref{sec:coupling}) whose distributions are identical to that of $\|G\|_{\classkernel}$.

For fixed data $X_0^{n-1} \in \calE'$ (see Lemma \ref{lem: coupling} and (\ref{eq:calE-prime})), we leverage another event
\begin{equation*}
    \calE'' := \Big\{ \big|\| \bbG_n^\zeta\|_{\classkernel^{(n)}} - W_1 \big| \leq \varepsilon_{1,n} \Big\}
\end{equation*}
and the fact that $W_1 \stackrel{d}{=}\|G\|_{\classkernel}$ to observe the inequality
\begin{equation*}
    \begin{aligned}
        &\bbP\Big( \| \bbG_n^\zeta \|_{\classkernel^{(n)}}  > c(\alpha - \delta_{1,n}) + \varepsilon_{1, n} \big| X_0^{n-1} \Big)\\
        &\leq \bbP\Big( \big\{ \| \bbG_n^\zeta \|_{\classkernel^{(n)}}  > c(\alpha - \delta_{1,n}) + \varepsilon_{1, n} \big\} \cap \calE'' \big|X_0^{n-1} \Big) + \delta_{1, n}\\
        &\leq \bbP\Big( W_1 > c(\alpha - \delta_{1, n}) \Big) + \delta_{1, n} \leq \alpha.
    \end{aligned}
\end{equation*}
So by the definition of quantiles, for every $X_0^{n-1} \in \calE'$ we have $\hat{c}_n(\alpha) \leq c(\alpha - \delta_{1, n}) + \varepsilon_{1, n}$. For the constant $C$ and the sequence $\xi_n^*$ specified in Corollary \ref{cor:Delta-bound}, repeating a similar argument results in
\begin{equation}
\begin{aligned}
    \bbP\left( \hat{c}_n(\alpha) < c(\alpha + \delta_{1,n}) - \varepsilon_{1,n} \right) \leq C\xi_n^* \quad \text{and} \quad \bbP\left( \hat{c}_n(\alpha) > c(\alpha - \delta_{1,n}) + \varepsilon_{1,n} \right) \leq C\xi_n^*.
    \end{aligned}
    \label{eq:quantile-est}
\end{equation}


\begin{proof}[Proof of Theorem \ref{thm : uniform confidence}] 
The proof is divided into several steps. Here, $c,c', C, C', n_0$ are generic positive constants that depend on the parameters specified in Corollary \ref{cor:Delta-bound}; its vlaues can differ from place to place.  

\underline{Step 1} (Lower bound coverage probability). Define the bias
\begin{equation*}
    c_n' := \sup_{x \in \calX}\sqrt{n}\big| \pi \kernel_{x, h} - \pi(x) \big| \big/ (\beta^{-1/2}\sigma(x)).
\end{equation*}
By Proposition \ref{prop:block-var}, we observe $c_n' \leq C \cdot \sup_{x \in \calX} \sqrt{nh}|\pi\kernel_{x, h} - \pi(x)|$. Recall that $\|\pi''\|_\infty = \sup_{x\in \calX} |\pi''(x)| < \infty$ and that $\kernel$ has a finite second moment. Then Taylor expand the smooth density $\pi$ to observe
\begin{equation*}
    \begin{aligned}
        c_n' &\leq C\cdot\sqrt{nh} \sup_{x \in \calX}\bigg| \frac{1}{2}h^2 \pi''(x) \int y^2 \kernel(y) dy + o(h^2) \bigg|
        &\leq C' h^{5/2} \sqrt{n}
    \end{aligned}
\end{equation*}

First invoke Lemma \ref{lem:studentizing-rel-error} and then isolate the bias $c_n'$ to observe
\begin{equation*}
    \begin{aligned}
        \bbP\big( \forall x \in \calX : \pi(x) \in \calI_\alpha(x) \big) \geq \bbP\big( \| \bbG_n \|_{\classkernel} \leq (1 - \varepsilon_{2, n}) \hat c_n(\alpha)  - c_n' \big) - \delta_{2, n}.
    \end{aligned}   
\end{equation*}
For some decreasing sequence $\delta_n > 0$, refer to the proof of Theorem \ref{thm:Gaussian-approximation} and set the remainder terms $R(\delta_n) := \bbP\big( \| \bbG_n^{(2)}\|_{\check\classkernel^\circ} > \delta_n/3 \big) + \bbP\big( \| \bbG_n^{(3)}\|_{\classkernel^\circ} > \delta_n/3 \big) + \bbP\big( \| \bbG_n^{(4)}\|_{\classkernel^\circ} > \delta_n/3 \big)$. Recall $n^\sharp = n$ when $m = 1$, so $\beta_n = \beta$ here.
So we observe
\begin{equation*}
    \begin{aligned}
        \bbP\big( \| \bbG_n \|_{\classkernel} \leq (1 - \varepsilon_{2, n}) \hat c_n(\alpha) - c_n' \big) &\geq \bbP \big( \beta^{-1/2} \| \bbG_n^{(1)} \|_{\check \classkernel^\circ} \leq (1 - \varepsilon_{2, n}) \hat c_n(\alpha) - c_n' - \delta_n \big) - R(\delta_n).
    \end{aligned}
\end{equation*}
By the coupling between $\beta^{-1/2} \| \bbG_n^{(1)} \|_{\check \classkernel^\circ}$ and $\|G\|_{\classkernel}$ established in (\ref{eq:first-coupling}), there exists $W_0 \stackrel{d}{=} \|G\|_\classkernel$ such that
\begin{equation*}
    \bbP \big( \beta^{-1/2} \| \bbG_n^{(1)} \|_{\check \classkernel^\circ} \leq (1 - \varepsilon_{2, n}) \hat c_n(\alpha) - c_n' - \delta_n \big) \geq \bbP\big( W_0 \leq (1 - \varepsilon_{2, n}) \hat c_n(\alpha) - \varepsilon_{0, n} - c_n' - \delta_n \big) - \delta_{0, n}.
\end{equation*}

Now set the short-hand for concise presentation
\begin{equation*}
    s_n := \varepsilon_{2, n}c(\alpha+\delta_{1, n}) + (1 - \varepsilon_{2, n})\varepsilon_{1, n} + \varepsilon_{0, n} + c_n' + \delta_n.
\end{equation*}
Then invoking (\ref{eq:quantile-est}) derived from the coupling between $\|\bbG_n^\zeta\|_{\classkernel^{(n)}}$ and $W_0 \stackrel{d}{=} \| G \|_{\classkernel}$ via Lemma \ref{lem: coupling}, we observe
\begin{equation*}
    \begin{aligned}
        \bbP\big( W_0 \leq (1 - \varepsilon_{2, n}) \hat c_n(\alpha) - \varepsilon_{0, n} - c_n' - \delta_n \big) \geq \bbP\big( W_0\leq  c(\alpha + \delta_{1, n}) - s_n \big) - C \xi_n^*.
    \end{aligned}
\end{equation*}
Note that the reference Gaussian process $\{G(f_{x, h})\}_{x \in \calX}$ is separable and also has variance $\Var(G(f_{x, h})) = \beta^{-1}\Var(\ccf_{x, h}(B_1))$. Recall $\Var(\brf_{x, h}(B_1)) = \beta$ and that $\sup_x \big| \Var(\brf_{x, h}(B_1)) - \Var(\ccf_{x, h}(B_1)) \big| \leq Ch^{1/2}$ for $n\geq n_0$ by (\ref{eq:bias-hn}); as long as $h_n = o(1)$ as $n \to \infty$, we observe $c \leq \Var(G(f_{x, h})) \leq C$ uniformly over $x \in \calX$ for large enough $n \geq n_0$.
Then we apply Levy's Gaussian anti-concentration~(Lemma \ref{lem:levy}) for $n\geq n_0$ and observe
\begin{equation*}\label{eq:ineq-five}
    \begin{aligned}
    &\bbP\big(W_0 \leq c(\alpha + \delta_{1, n}) - s_n \big)  \\
        &\geq  \bbP\left( W_0 \leq c(\alpha + \delta_{1,n}) \right) - cs_n \cdot \Big\{ \E\big[ \sup_{x \in \calX} |G(f_{x, h})/\Var^{1/2}(G(f_{x, h})) \big] + \sqrt{\log (1 \vee (1/s_n))} \Big\}\\
        &\geq 1 - \alpha - \delta_{1, n} - c'(\log^{1/2} n) s_n.
    \end{aligned}
\end{equation*}
By maximal inequality of standardized Gaussian process on the compact subset $\calX$~\cite{Dudley2014UCLT}, we have $\E\big[ \sup_{x \in \calX} |G(f_{x, h})/\Var^{1/2}(G(f_{x, h})) \big] \leq C$. Collecting the inequalities, observe
\begin{equation}\label{eq:near-final-lower}
    \bbP\big( \forall x \in \calX: \pi(x) \in \calI_\alpha(x) \big) \geq 1- \alpha - C \{\delta_{0, n} + \delta_{1, n} + \delta_{2, n} +  \xi_n^* + (\log^{1/2}n) s_n + R_n\} .
\end{equation}

\underline{Step 2} (Specifying the order). 
We show that the sequences appearing on the right-hand side of the inequality (\ref{eq:near-final-lower}) can be taken to vanish polynomially. For simplicity we set $h = n^{-1/5-\eta}$ for some fixed $\eta > 0$, which automatically implies polynomially vanishing bias $c_n'$. 

Next consider $\gamma$ such that $\gamma - 1/5 - \eta > 0$. Then the condition $n^\gamma h  \geq n^{\underline c}M_nK_n$ is met, and it further implies $n^{\overline \gamma /2} \geq n^{\underline c/ 4}\sqrt{M_n} h^{-1/2} K_n$~(refer to Step 1 of the proof of Theorem \ref{thm: bootstrap consistency}), hence implying 
\begin{equation*}
    \frac{h^{-1/2}}{n^{1/4 + \overline \gamma /2}} \leq \frac{1}{n^{\underline c / 4 + 1/4}}. 
\end{equation*}
Refer to the proof of Theorem \ref{thm:Gaussian-approximation}, then by setting $\delta_n = n^{-1/4 + \overline \gamma}$ here, we observe $R_n \leq Cn^{-\underline{c}/4 - 1/4}$ whenever $M_n \geq n^{\underline \gamma/2}$. 

Next the condition that $h\alpha_n^{-2/3} = n^{-1/5 - \eta}\alpha_n^{-2/3}$ increase polynomially implies that the sequences $n^{1/2-\gamma}h^{-1}\alpha_n$ and $\sqrt{h}n^{1/4 + \overline \gamma/2}$ decrease and increase polynomially respectively. So $\xi_n^*$ vanishes polynomially. Further setting $\gamma, \eta$ such that $\gamma > 2/5 - \eta/2$ implies the additional condition $\sqrt{n}h^{1/2} \leq n^\gamma$ is met---then we may directly refer to Lemma \ref{lem:studentizing-rel-error} and \ref{lem: coupling} to observe $\varepsilon_{1, n}, \varepsilon_{2, n}, \delta_{1, n}, \delta_{2, n}$ vanishes polynomially. 

Lastly, Set the sequence $\gamma_0$ defined in (\ref{eq:first-coupling}) to vanish polynomially, and set constant $\eta > 0$ such that $n^{1/15 - \eta/2}\gamma_0^{1/3}$ increases polynomially. Then by Theorem A.1 of \cite{CCK2014bAoS}~(see (\ref{eq:first-coupling})), the sequences $\varepsilon_{0,n}, \delta_{0,n}$ vanish polynomially. As $c_n'$ and $\delta_n$ is set to vanish polynomially, we observe $s_n$ to decrease polynomially. 

Overall, for the polynomial convergence of the uniform confidence level of the interval $\calI_\alpha(x)$, it is sufficient to find constants $\gamma, \eta$ and the sequences $\alpha_n, \gamma_0$ that satisfy the following conditions:
\begin{equation}\label{eq:final-growth-cond}
    \begin{aligned}
        &\gamma > (1/5 + \eta) \vee (2/5 - \eta/2) \quad \textrm{and}\\
        &n^{-1/5 - \eta}\alpha_n^{-2/3} \wedge n^{1/15 - \eta/2}\gamma_0^{1/3} \quad \text{increase polynomially}.
    \end{aligned}
\end{equation}
Set $\alpha_n = (\log n /n)^{s/(s + 1)}$~(see Remark \ref{rem:lem-cov}) for the smoothness parameter $s$. Then set $\gamma_0 = n^{-1/10}$ and fix $\eta \in (0, 1/15), \gamma$ such that $s/(s + 1) > 3/10 + 3\eta/2$ and $\gamma > 2/5 - \eta/2$. Then conditions in (\ref{eq:final-growth-cond}) are satisfied. Note that these conditions are mildly satisfied under the smoothness parameter $s = 1$, which is the case for the diffusion process of interest. 

\end{proof}

\end{document}